\documentclass[oneside]{amsart}
\usepackage{geometry}                
\usepackage{graphicx}
\usepackage{amsmath}
\usepackage{amssymb}
\usepackage{amsthm}
\usepackage{mathtools}
\usepackage{mathscinet}
\usepackage{tikz-cd}
\usepackage{dsfont}
\usepackage{color}
\usepackage{enumitem}
\usepackage{tensor}
\usepackage{nicefrac}
\usepackage{todonotes}
\usepackage{mathrsfs}
\usepackage{float}

\usepackage[]{hyperref}
\hypersetup{
	pdftitle={Equidistribution on Kuga-Sato Varieties of Torsion Points on CM Elliptic Curves},
	pdfauthor={Ilya Khayutin},
	bookmarksnumbered=true,
	bookmarksopen=true,
	bookmarksopenlevel=1,
	bookmarksdepth=3,
	colorlinks=true,
	citecolor=blue,
	pdfstartview=Fit,
	pdfpagemode=UseOutlines,
	pdfpagelayout=TwoPageRight
}

\usepackage[protrusion=true]{microtype}

\DeclareSymbolFont{legacymaths}{OT1}{cmr}{m}{n}
\SetSymbolFont{legacymaths}{bold}{OT1}{cmr}{bx}{n}

\makeatletter
\let\orgdescriptionlabel\descriptionlabel
\renewcommand*{\descriptionlabel}[1]{%
  \let\orglabel\label
  \let\label\@gobble
  \phantomsection
  \protected@edef\@currentlabel{#1}%
  \let\label\orglabel
  \orgdescriptionlabel{#1}%
}
\def\paragraph{
	\@startsection{paragraph}{4}
	\z@{.5\linespacing\@plus.7\linespacing}{-.5em}%
	{\normalfont\itshape}}
\makeatother

\DeclareFontFamily{U}{mathx}{\hyphenchar\font45}
\DeclareFontShape{U}{mathx}{m}{n}{
      <5> <6> <7> <8> <9> <10>
      <10.95> <12> <14.4> <17.28> <20.74> <24.88>
      mathx10
      }{}
\DeclareSymbolFont{mathx}{U}{mathx}{m}{n}
\DeclareFontSubstitution{U}{mathx}{m}{n}
\DeclareMathAccent{\widecheck}{0}{mathx}{"71}

\newcommand{\dif}{\,\mathrm{d}}

\DeclareMathOperator{\vol}{vol}
\DeclareMathOperator{\disc}{disc}
\DeclareMathOperator{\Nr}{Nr}

\DeclareMathOperator{\Tr}{Tr}
\DeclareMathOperator{\Lie}{Lie}

\DeclareMathOperator{\Gal}{Gal}
\DeclareMathOperator{\Hom}{Hom}

\DeclareMathOperator{\End}{End}
\DeclareMathOperator{\Aut}{Aut}

\DeclareMathOperator{\WR}{Res}

\DeclareMathOperator{\supp}{supp}

\DeclareMathOperator{\inv}{inv}

\DeclareMathOperator{\Ideals}{\mathcal{J}}

\DeclareMathOperator{\Pic}{Pic}

\DeclareMathOperator{\Res}{Res}
\DeclareMathOperator{\val}{val}

\newcommand{\Nrml}{\operatorname{N}}
\newcommand{\Cent}{\operatorname{Z}}
\newcommand{\Cor}{\operatorname{Cor}}
\newcommand{\Stab}{\operatorname{Stab}}
\newcommand{\RO}{\operatorname{RO}}
\newcommand{\ord}{\operatorname{ord}}
\newcommand{\Ad}{\operatorname{Ad}}

\newcommand{\meas}{m}
\renewcommand{\det}{\operatorname{det}} 

\DeclareMathOperator{\Gm}{\mathbb{G}_m}
\DeclareMathOperator{\Ga}{\mathbb{G}_a}

\newlength{\faktorheight}
\newcommand*{\dfaktor}[3]{
\mathchoice{
 \settototalheight{\faktorheight}{\ensuremath{#1}}%
  \raisebox{-0.5\faktorheight}{\ensuremath{#1}}
   \backslash
   \settototalheight{\faktorheight}{\ensuremath{#2}}%
  \raisebox{0.5\faktorheight}{\ensuremath{#2}}
  \slash
   \settototalheight{\faktorheight}{\ensuremath{#3}}%
  \raisebox{-0.5\faktorheight}{\ensuremath{#3}}
  }{
  \ensuremath{#1}
  \backslash
  \ensuremath{#2}
  \slash
  \ensuremath{#3}
  }
  {\let\GenerateWarning=\errmessage}
  {\let\GenerateWarning=\errmessage}
}

\newcommand*{\lfaktor}[2]{
\mathchoice{
 \settototalheight{\faktorheight}{\ensuremath{#1}}%
  \raisebox{-0.5\faktorheight}{\ensuremath{#1}}
	\backslash
   \settototalheight{\faktorheight}{\ensuremath{#2}}%
  \raisebox{0.5\faktorheight}{\ensuremath{#2}}
}
{
	\ensuremath{#1}
	\backslash
	\ensuremath{#2}
}
{\let\GenerateWarning=\errmessage}
{\let\GenerateWarning=\errmessage}
}

\newcommand*{\faktor}[2]{
\mathchoice{
 \settototalheight{\faktorheight}{\ensuremath{#1}}%
  \raisebox{0.5\faktorheight}{\ensuremath{#1}}
  \slash
   \settototalheight{\faktorheight}{\ensuremath{#2}}%
  \raisebox{-0.5\faktorheight}{\ensuremath{#2}}
}
{
	\ensuremath{#1}
	\slash
	\ensuremath{#2}
}
{\let\GenerateWarning=\errmessage}
{\let\GenerateWarning=\errmessage}
}

\newcommand*{\cyclic}[1]{
	\faktor{\mathbb{Z}}{#1\mathbb{Z}}
}

\newtheorem{thm}{Theorem}
\numberwithin{thm}{section}
\newtheorem*{thm*}{Statement of Theorem}
\newtheorem*{thm-quote}{Theorem}
\newtheorem{lem}[thm]{Lemma}
\newtheorem{prop}[thm]{Proposition}
\newtheorem{cor}[thm]{Corollary}

\newtheorem{conj}[thm]{Conjecture}
\theoremstyle{definition}
\newtheorem{defi}[thm]{Definition}

\theoremstyle{remark}
\newtheorem{remark}[thm]{Remark}

\title[Equidistribution on Kuga-Sato Varieties]{Equidistribution on Kuga-Sato Varieties of Torsion Points on CM Elliptic Curves}
\author[I. Khayutin]{Ilya Khayutin}
\address{Department of Mathematics, Northwestern University, Evanston, IL 60208, USA}

\begin{document}

\begin{abstract}

A connected Kuga-Sato variety $\mathbf{W}^r$ parameterizes tuples of $r$ points on elliptic curves (with level structure). A special point of $\mathbf{W}^r$ is a tuple of torsion points on a CM elliptic curve. A sequence of special points is strict if any CM elliptic curve appears at most finitely many times and no relation between the points in the tuple is satisfied infinitely often. The genus orbit of a special point is the $\Gal(\bar{\mathbb{Q}}/\mathbb{Q}^{\mathrm{ab}})$-orbit.
We show that genus orbits of special points in a strict sequence equidistribute in $\mathbf{W}^r(\mathbb{C})$, assuming a congruence condition at two fixed primes.

A genus orbit can be very sparse in the full Galois orbit. In particular, the number of torsion points on each elliptic curve in a genus orbit is not bounded below by the torsion order.

A genus orbit corresponds to a toral packet in an extension of $\mathbf{SL}_2$ by a vector representation. These packets also arise in the study by Aka, Einsiedler and Shapira of grids orthogonal to lattice points on the $2$-sphere. As an application we establish their joint equidistribution conjecture assuming two split primes.
\end{abstract}
\maketitle
\bgroup
\hypersetup{linkcolor = blue}
\tableofcontents
\egroup

\section{Introduction}
Our results can be presented either from a viewpoint of arithmetic geometry or homogeneous dynamics. We discuss first the arithmetic statements.
\subsection{Equidistribution of Genus Orbits of Special Points}
\subsubsection{Kuga-Sato Varieties}
Let $Y$ be a connected complex modular curve with neat level structure, i.e.\ $Y$ is a manifold. Assume that $Y$ can be defined over $\mathbb{Q}$, e.g. $Y=Y_1(N)$ for $N\geq 4$.
Denote by $W^r\to Y$ the $r$-fold complex Kuga-Sato variety over $Y$. A point $(A,\mathfrak{l})$ on $Y$ corresponds to a complex elliptic curve $A$ with level structure $\mathfrak{l}$. A point of $W^r$ above $(A,\mathfrak{l})$ corresponds to a tuple of $r$ complex points $X_1,\ldots,X_r\in A$. In particular, the fiber of $W^r\to Y$ over $(A,\mathfrak{l})$ is isomorphic to $A^r$. The universal elliptic curve $\mathcal{E}\to Y$ coincides with $W^1\to Y$.

\subsubsection{Special Points and Genus Orbits}
A special point of $Y$ is an elliptic curve with CM and appropriate level structure. A special point of $W^r$ coincides with $r$ torsion points over a special point of $Y$.

The theory of Complex Multiplication implies that special points of $W^r$ are algebraic, and defined over an abelian extension of an imaginary quadratic field. In this paper we study $\Gal(\bar{\mathbb{Q}}/\mathbb{Q}^\mathrm{ab})$-orbits of special points on $W^r$. We baptize them \emph{genus orbits} due to the evident relation with principal genus theory of quadratic fields. In particular, the genus orbit is always defined over the genus field of the imaginary quadratic field.

Genus orbits also arise naturally in geometric questions about lattice points on the $2$-sphere and their orthogonal grids as studied by Einsiedler, Aka and Shapira \cite{AES}. The first motivation to study genus orbits is a conjecture of \cite{AES}; a partial resolution of this conjecture is elaborated upon in \S \ref{intro:grids}. The second motivation stems in homogeneous dynamics -- understanding the asymptotic distribution of periodic torus orbits.

\subsubsection{Equidistribution}\label{intro:equidist}
The complex variety $W^r$ carries a natural uniform probability measure -- $\meas$. The push-forward of $\meas$ to $Y$ is the normalized hyperbolic volume measure and the conditional measure on each fiber $A^r$ is the probability Haar measure.
Equivalently, the uniform measure can be constructed using the uniformization of $W^r$. Denote $\mathbf{V}\coloneqq \Ga^{\times 2}$ and let
$\mathbf{P}^r\simeq \mathbf{SL}_2\ltimes \mathbf{V}^{\oplus r}$ where $\mathbf{SL}_2$ acts on $\mathbf{V}^{\oplus r}$ diagonally. Fix an embedding $\mathbf{SL}_2\hookrightarrow\mathbf{P}^r$, e.g.\ using the semi-direct product structure. This embedding defines the zero section.
There is a lattice $\Gamma<\mathbf{P}^r(\mathbb{R})$ such that
\begin{equation*}
W^r\simeq \dfaktor{\Gamma}{\mathbf{P}^r(\mathbb{R})}{K_\infty} \, ,
\end{equation*}
where $K_\infty=\mathbf{SO}_2(\mathbb{R})<\mathbf{SL}_2(\mathbb{R})\hookrightarrow \mathbf{P}^r(\mathbb{R})$ is a maximal compact subgroup. The isomorphism holds in the category of real analytic manifolds and the probability Haar measure on $\lfaktor{\Gamma}{\mathbf{P}^r(\mathbb{R})}$ descends to the uniform measure $\meas$ on $W^r$.

\begin{defi}\label{defi:strict-galois}
Write a point of $W^r$ as $(A,\mathfrak{l};X_1,\ldots,X_r)$.
A sequence of points $$\{(A_i,\mathfrak{l}_i;X_1^i,\ldots,X_r^i)\}_i\subset W^r$$ is strict if no fixed elliptic curve appears infinitely often in the sequence and for all $0\neq \left(m_1,\ldots,m_r\right)\in \mathbb{Z}^r$ the equation $m_1X_1^i+\ldots+m_r X_r^i=0$ holds finitely many times at most.
If $r=1$ the latter condition is equal to $\ord_{A_i}(X_1^i)\to_{i\to\infty}\infty$, where $\ord_{A_i}(X_1^i)$ is the order of the torsion point $X_1^i$ in $A_i$.
\end{defi}

Notice that a non-strict sequence has always a subsequence that is trapped in a weakly special subvariety of $W^r$. The weakly special subvariety in question is either the fiber over a point of $Y$ or a closed embedding $W^{r-1}\hookrightarrow W^r$ whose image is the geometric locus of an equation $m_1X_1+\ldots+m_r X_r=0$. In particular, a general sequence of special points can be decomposed into sequences each of which is appropriately strict in a weakly special subvariety.

The following is a version of our main theorem stated in the language of Galois orbits.
\begin{thm}\label{thm:genus-equidistribution}
Let $\{x_i=(A_i,\mathfrak{l}_i;X_1^i,\ldots,X_r^i)\}_i\subset W^r$ be a strict sequence of special points.
Denote the discriminant of the CM order of $A_i$ by $D_i$ and let $N_i$ be the order of $(X_1^i,\ldots,X_r^i)$ in $A^r$.
Write $D_i=D_i^\mathrm{fund} f_i^2$ where $D_i^\mathrm{fund}$ is a fundamental discriminant and $f_i$ is the conductor.
Fix two distinct primes $p_1$, $p_2$ and assume for all $i$
\begin{enumerate}
\item
\begin{align*}
\left(\frac{D_i^\mathrm{fund}}{p_1}\right)&=\left(\frac{D_i^\mathrm{fund}}{p_2}\right)=1\\
\val_{p_1}(f_i)&,\val_{p_2}(f_i)\ll 1\, ,
\end{align*}
\item
\begin{equation*}
\val_{p_1}(N_i),\val_{p_2}(N_i)\ll 1 \, .
\end{equation*}
\end{enumerate}
Then for every continuous compactly supported $f\colon W_r\to\mathbb{C}$
\begin{equation*}
\frac{1}{\left|\Gal(\bar{\mathbb{Q}}/\mathbb{Q}^\mathrm{ab}).x_i\right|}
\sum_{y\in \Gal(\bar{\mathbb{Q}}/\mathbb{Q}^\mathrm{ab}).x_i}f(y)
\to_{i\to\infty} \int_{W^r} f\dif \meas \, .
\end{equation*}
Stated otherwise, the normalized counting measures on the genus orbits of $x_i$ converge weak-$*$ to the uniform measure on $W^r$ when $i\to\infty$.
\end{thm}

\subsubsection{Full Galois Orbits}
A genus orbit can be very sparse in the full Galois orbit of a special point. For example, a special point $(A,\mathfrak{l}; X)\in W^1$, when $A$ has CM by an order of \emph{prime} discriminant $D$ and $X$ is of order $|D|$, has a genus orbit in which over each elliptic curve there is a single torsion point. The full Galois orbit has $\gg_\varepsilon |D|^{1-\varepsilon}$ torsion points over each elliptic curve in the orbit.

This makes proving equidistribution of full Galois orbits significantly simpler. A.\ Venkatesh has observed that in the full Galois orbit of a special point $(A,\mathfrak{l};X)\in W^1$ the fiber of $\Gal(\bar{\mathbb{Q}}/\mathbb{Q}).(A,\mathfrak{l};X)$ over most elliptic curves $\tensor[^\sigma]{A}{}\in\Gal(\bar{\mathbb{Q}}/\mathbb{Q}).A$ becomes equidistributed in $\tensor[^\sigma]{A}{}$ in a quantitative and uniform way.
Writing down the full Galois orbit as an orbit of the full id\`ele class group of the CM field we can identify exactly the torsion points of $\tensor[^\sigma]{A}{}$ appearing in the fiber. This turns out to be a big enough set so that its equidistribution in $\tensor[^\sigma]{A}{}$ can be verified in an elementary way. The same statement or method cannot in general apply to the genus orbit as in the example above it contains a single torsion point over each elliptic curve in the orbit.
Equidistribution of genus orbits in the total space $W^r$ is somewhat delicate as it fails fiber-wise.
\subsubsection{Methods}
The proof of Theorem \ref{thm:genus-equidistribution} combines measure rigidity for diagonal actions, specifically \cite[Theorem 1.6]{ELJoinings}; the relative trace method for bounding accumulation on intermediate orbits -- first introduced by the author in \cite{KhJoint}; and the subconvex bound in the level aspect of Duke, Friedlander and Iwaniec \cite{DFI}. This paper has two main novel contributions.

\paragraph{Arithmetic Invariants}
The first one is the fine \emph{arithmetic} analysis of the orbit space of the quotient of the two dimensional affine space $\mathbf{V}$ by a rational anisotropic torus $\mathbf{T}<\mathbf{SL}_2$. Specifically, $\mathbf{T}$ is the norm $1$ torus of the  imaginary quadratic field associated to an elliptic curve with CM.
The GIT quotient of  $\mathbf{V}$ by a torus $\mathbf{T}$ is nothing but the affine line. Yet this quotient does not suffice for our purposes as it can at best parameterize orbits of $\mathbf{T}(\mathbb{Q})$ on $\mathbb{Q}^2$. We need to parametrize orbits of a certain compact-open subgroup of $\mathbf{T}(\mathbb{A}_f)$. We achieve such a parametrization by constructing an explicit invariant function valued in invertible fractional ideals of a quadratic order, with some extra level structure and a restriction on the Picard class.
This invariant theoretic problem is related to the question of understanding $\mathcal{T}(\mathbb{Z})$-orbits for an integral non-smooth model $\mathcal{T}$ of $\mathbf{T}$.
This analysis is carried out in \S\ref{sec:orbit-space} with complimentary local computations in  Appendix \ref{app:L-functions}.

\paragraph{Subconvexity}
The second new ingredient that is applicable specifically to the setting at hand is subconvexity of certain Hecke $L$-functions which is used to bound the arithmetic sums produced by the relative trace method.
The application of subconvexity can be seen as a substitute for the sieve method used in \cite{KhJoint}; it has a significant advantage of not requiring any assumptions about exceptional Landau-Siegel zeros. The use of subconvexity is described in \S\ref{sec:subconvex}.

\paragraph{Orthogonal Grids as Intersection of Periodic Orbits}
We also present a new viewpoint on the construction of grids orthogonal to lattice points in $\mathbb{Z}^3$ as studied by  Einsiedler, Aka and Shapira in \cite{AES}, cf.\ \S \ref{intro:grids}. We demonstrate that their construction is equivalent to the intersection of two periodic orbits in $\lfaktor{\mathbf{SL_3(\mathbb{Z})}}{\mathbf{SL}_3(\mathbb{R})}$. This viewpoint is mostly used to reprove well-known result. Nevertheless, the benefit of the intersection representation is an elegant explicit description of a joint ad\`elic torus action on the lattice points and their orthogonal grids. This can probably be achieved also by classical means, but I believe the intersection pictures demystifies many properties of the construction -- including the squaring of the Picard action as described in Remark \ref{rem:squaring}.

\subsubsection{Results without a Congruence Assumption}
The congruence assumption at the primes $p_1$, $p_2$ in Theorem \ref{thm:genus-equidistribution} provides invariance of any weak-$*$ limit measure under a split torus at two places. This invariance is required to apply measure rigidity results of Einsiedler and Lindenstrauss \cite{ELJoinings}. It is important to note that the methods of this paper provide useful information even without a congruence assumption.

The complex universal elliptic curve $\mathcal{E}$ is uniformized by $\mathbb{H}\times \mathbb{C}$. For any point $x=(x_\mathbb{H},x_\mathbb{C})\in \mathbb{H}\times \mathbb{C}$ denote by $B(x,R,r)$ the product of the hyperbolic ball of radius $R$ in $\mathbb{H}$ around $x_\mathbb{H}$ and the Euclidean ball of radius $r$ in $\mathbb{C}$ around $x_\mathbb{C}$.
By abuse of notation we also denote by $B(x,R,r)$ its projection to $\mathcal{E}$ under the quotient map -- this is an open neighborhood of the point $x\in\mathcal{E}$.
\begin{thm}\label{thm:genus-non-split}
Let $\{x_i\}_i\subset W^1=\mathcal{E}$ be a strict sequence of special points and denote by $\mu_i$ the Borel probability measure on $\mathcal{E}$ defined by
\begin{equation*}
\mu_i=\frac{1}{\left|\Gal(\bar{\mathbb{Q}}/\mathbb{Q}^\mathrm{ab}).x_i\right|}\sum_{y\in \Gal(\bar{\mathbb{Q}}/\mathbb{Q}^\mathrm{ab}).x_i} \delta_y \, .
\end{equation*}
Let $\mathcal{A}\subset \mathcal{E}$ be a torsion section and denote by $\nu_\mathcal{A}$ the uniform probability measure on $\mathcal{A}$ then
\begin{equation*}
\limsup_{i\to\infty} \int  \mu_i(B(x,R,r)) \dif \nu_\mathcal{A}(x)\ll \left(\cosh(2R)-1\right)r^2 \, .
\end{equation*}
Moreover, if $\mu_i\xrightarrow[i\to\infty]{\textrm{weak-}*}\mu$ then $\mu(\mathcal{A})=0$.
\end{thm}
The torsion section $\mathcal{A}$ is also a modular curve by itself, of level higher than $Y$ or equal to it; the probability measure $\nu_A$ is the normalized hyperbolic volume measure on the modular curve $\mathcal{A}$.
The expression $\left(\cosh(2R)-1\right)r^2$ is proportional to the product of the volume of a hyperbolic circle of radius $R$ and an Euclidean circle of radius $r$. This theorem may be interpreted as stating that the mass a genus orbit of a special point puts in a ball around a typical point on a torsion section is eventually bounded by the uniform measure of this ball; where typical is with respect to the volume measure on the torsion section.

\subsubsection{Further Discussion}
\paragraph{The Weil Pairing on a Genus Orbit}
For any integer $N$ and elliptic curve $A$ the Weil pairing $w\colon A[N]\times A[N]\to\mu_n$ is a non-degenerate bilinear alternating and Galois equivariant pairing valued in roots of unity of order $N$. As $\mathbb{Q}^\mathrm{ab}$ is the cyclotomic extension of $\mathbb{Q}$ we have for any $\sigma\in\Gal(\bar{\mathbb{Q}}/\mathbb{Q}^\mathrm{ab})$ and $P,Q\in  A[N]$
\begin{equation*}
w(\tensor[^\sigma]{P}{},\tensor[^\sigma]{Q}{})=w(P,Q)\in\mu_N \, .
\end{equation*}
In particular, the genus orbit of a special point $(A,\mathfrak{l};X_1,\ldots,X_r)\in W^r$ of torsion order $N$ has the pleasant property that for any $1\leq i\neq j\leq r$ the Weil pairing $w(X_i,X_j)$ is a well-defined invariant of the orbit. This is of course meaningful only for $r\geq 2$.

\paragraph{Genus Orbits on the Modular Curve}\label{intro:genus-modular-curve}
The situation for genus orbits on $Y$ and $W^r$ for $r\geq 1$ is very different. In the modular curve the genus orbit is rather large in the full Galois orbit, due to the fact that the index of a principal genus subgroup in the Picard group of a quadratic order of discriminant $D<0$ is of size $\ll 2^{\omega(D)}$. In particular, equidistribution of Genus orbits on the modular curve is closely related to equidistribution of full orbits proven by Duke \cite{Duke}. Under the extra assumption of a fixed split prime equidistribution of genus orbits of CM points has already been established by Linnik \cite{LinnikBook}. Without a split prime assumptions it follows from the Waldspurger formula \cite{Waldspurger} and the subconvex bounds of \cite{DFI-II,DFI-IIerr,Harcos,Michel04}.

\subsection{Equidistribution of Torus Orbits}
We proceed to describe our results in terms of periodic torus orbits and homogeneous dynamics. This is the framework used in the proofs. Our description is ad\`elic due to the fact that the periodic torus orbits in question are cumbersome to define classically.

\subsubsection{Ad\`elic Homogeneous Space}\label{sec:adelic-homogeneous}
A connected complex modular curve $Y$ can be uniformized by the complex upper-half plane $\mathbb{H}$. This provides a diffeomorphism
\begin{equation*}
Y\simeq \dfaktor{\mathbf{SL}_2(\mathbb{Q})}{\mathbf{SL}_2(\mathbb{A})}{K_\infty U_f} \, ,
\end{equation*}
where $K_\infty=\mathbf{SO}_2(\mathbb{R})<\mathbf{SL}_2(\mathbb{R})$ is an archimedean maximal compact subgroup and $U_f<\mathbf{SL}_2(\mathbb{A}_f)$ is a neat compact-open subgroup contained in $\mathbf{SL}_2(\hat{\mathbb{Z}})$.

For any $r\in\mathbb{N}$ we have defined $\mathbf{P}^r\coloneqq \mathbf{SL}_2\ltimes \mathbf{V}^{\oplus r}$ where $\mathbf{V}\coloneqq \Ga^{\times 2}$. Moreover, we have fixed an embedding $\mathbf{SL}_2\hookrightarrow\mathbf{P}^r$ corresponding to the zero section. Denote
$U_f^\ltimes\coloneqq U_f\ltimes \mathbf{V}(\hat{\mathbb{Z}})^{\oplus r}$ -- this is a compact-open subgroup of $\mathbf{P}^r(\mathbb{A}_f)$.
The connected complex $r$ Kuga-Sato variety above $Y$ satisfies
\begin{equation*}
W^r\simeq \dfaktor{\mathbf{P}^r(\mathbb{Q})}{\mathbf{P}^r(\mathbb{A})}{K_\infty U_f^\ltimes} \, ,
\end{equation*}
where the isomorphism is in the category of real analytic manifolds.
The lattice $\Gamma$ from \S\ref{intro:equidist} is equal to $\mathbf{P}^r(\mathbb{Q})\cap U_f^\ltimes$.
Notice that this is not the standard presentation of a Kuga-Sato variety as a mixed Shimura variety. The mixed Shimura model of $W^r$ is a double quotient of $\mathbf{GL}_2\ltimes \mathbf{V}^{\oplus r}$. The presentation using $\mathbf{SL}_2$ is adapted to the setting of genus orbits, as maximal tori in $\mathbf{SL}_2$ are norm $1$ tori of quadratic fields. The norm $1$ torus is mapped by the Artin reciprocity map of class field theory to a quotient of $\Gal(\bar{\mathbb{Q}}/\mathbb{Q}^\mathrm{ab})$.

Instead of studying the double quotient $W^r$ we shall establish equidistribution in the ad\`elic homogeneous space
\begin{equation*}
\left[\mathbf{P}^r(\mathbb{A})\right]\coloneqq \lfaktor{\mathbf{P}^r(\mathbb{Q})}{\mathbf{P}^r(\mathbb{A})} \, .
\end{equation*}
The pertinent statement for $W^r$ with any level structure will then follow by push-forward of measures from the ad\`elic quotient to the double quotient.

\subsubsection{Homogeneous Toral Sets}
Every maximal torus $\mathbf{T}<\mathbf{SL}_2$ defined and anisotropic over $\mathbb{Q}$ is of the form $\mathbf{T}\simeq \WR_{E/\mathbb{Q}}^1 \Gm$ for a quadratic field extension $E/\mathbb{Q}$.
Fix $\xi=(l,\mathrm{x})\in\mathbf{P}^r(\mathbb{A})$ with $l\in\mathbf{SL}_2(\mathbb{A})$ and $\mathrm{x}\in\mathbf{V}(\mathbb{A})^{\oplus r}$. The closed subset
$$\mathcal{H}=\left[\mathbf{T}(\mathbb{A})\xi\right]\subset \left[\mathbf{P}^r(\mathbb{A})\right]$$
is called a \emph{homogeneous toral set} \cite{ELMVCubic} and carries a unique $\Ad_{\xi^{-1}}\mathbf{T}(\mathbb{A})$-invariant Borel probability measure that we call the \emph{periodic measure}. We say that $\mathcal{H}$ is $K_\infty$-invariant if $\Ad_{\xi_\infty^{-1}}\mathbf{T}(\mathbb{R})=K_\infty$. In this case the splitting field $E/\mathbb{Q}$ of $\mathbf{T}$ is imaginary and $\mathrm{x}_\infty=0$.

In \S \ref{sec:order} we attach to a $K_\infty$-invariant homogeneous toral set an order $\Lambda$ in the splitting field $E/\mathbb{Q}$. If the homogeneous toral set arises by class field theory from a genus orbit of a special point then this order coincides with the endomorphism ring of the associated CM elliptic curve. The discriminant $D$ of a homogeneous toral set is defined to be the discriminant of this order. The torsion order $N\in\mathbb{N}$ of a homogeneous toral set  $\left[\mathbf{T}(\mathbb{A})(l,\mathrm{x})\right]$ is the order of the non-archimedean part $\mathrm{x}_f\in\mathbf{V}^{\oplus r}(\mathbb{A}_f)$ in the torsion group $\mathbf{V}^{\oplus r}(\mathbb{A}_f)/l_f.\mathbf{V}^{\oplus r}(\hat{\mathbb{Z}})$.

\subsubsection{Main Theorems}
We state our main results in their final form.
\begin{defi}\label{defi:strict-homogeneous}
A sequence of $K_\infty$-invariant homogeneous toral sets $\left\{\left[\mathbf{T}_i(\mathbb{A})(l_i,\mathrm{x}^i)\right]\right\}_i$ in $\left[\mathbf{P}^r(\mathbb{A})\right]$
with discriminants $\{D_i\}_i$ is \emph{strict} if $|D_i|\to_{i\to\infty}\infty$ and for every $0\neq m=(m_1,\ldots,m_r)\in\mathbb{Q}^r$ the sequence of elements
\begin{equation*}
l_i^{-1}.\left(\sum_{k=1}^r m_k \mathrm{x}_k^i\right)\in\mathbf{V}(\mathbb{A})
\end{equation*}
escapes all compact sets in $\mathbf{V}(\mathbb{A})$.

Let $N_{m,i}$ be the order of the non-archimedean part $\sum_{k=1}^r m_k \mathrm{x}_{k,f}^i\in \mathbf{V}(\mathbb{A}^f)$ in the torsion group $\mathbf{V}(\mathbb{A}_f)/l_{i,f}.\mathbf{V}(\hat{\mathbb{Z}})$. The last condition is equivalent to $N_{m,i}\to_{i\to\infty}\infty$ for every $0\neq m \in \mathbb{Q}^r$. In particular, for $r=1$ the latter condition is equivalent to $N_i\to_{i\to\infty}\infty$ where $N_i$ is the torsion order of the $i$'th homogeneous toral set.
\end{defi}

\begin{thm}\label{thm:main}
Let $\mathcal{H}_i\subset \left[\mathbf{P}^r(\mathbb{A})\right]$ be a strict sequence of $K_\infty$-invariant homogeneous toral sets. Let $D_i$ be the discriminant of $\mathcal{H}_i$ and $N_i$ the torsion order. Write $D_i=D_i^\mathrm{fund}f_i^2$ where $D_i^\mathrm{fund}$ is a fundamental discriminant and $f_i$ is the conductor. Denote by $\mu_i$ the periodic measure supported on $\mathcal{H}_i$.

Fix two distinct primes $p_1,p_2$ and assume for all $i$

\begin{align}
\label{eq:congruence-1}\tag{$\spadesuit$}
\left(\frac{D_i^\mathrm{fund}}{p_1}\right)&=
\left(\frac{D_i^\mathrm{fund}}{p_2}\right)=+1\\
\val_{p_1}(f_i),&\val_{p_2}(f_i)\ll 1 \, ,
\nonumber
\end{align}

\begin{equation}
\label{eq:congruence-2}\tag{$\spadesuit\spadesuit$}
\val_{p_1}(N_i),\val_{p_2}(N_i)\ll 1 \, .
\end{equation}

Then $\mu_i\xrightarrow[i\to\infty]{\mathrm{weak}-*}m$ where $m$ is the probability Haar measure on $\left[\mathbf{P}^r(\mathbb{A})\right]$.
\end{thm}

I suspect that the theorem holds without the congruence conditions \eqref{eq:congruence-1} and \eqref{eq:congruence-2} above, yet these conditions are crucial to the proof method in this article. In the directions of removing the congruence conditions we have the following.

\begin{thm}\label{thm:main2}
Let $\mathcal{H}_i\subset \left[\mathbf{P}^1(\mathbb{A})\right]$ be a strict sequence of $K_\infty$-invariant homogeneous toral sets. Denote by $\mu_i$ the periodic measure supported on $\mathcal{H}_i$ and assume $\mu_i\xrightarrow[i\to\infty]{\textrm{weak}-*}\mu$. Then for any $\mathrm{y}\in\mathbf{V}(\mathbb{A})$
\begin{equation*}
\mu\left(\left[\mathbf{SL}_2(\mathbb{A})(e,\mathrm{y})\right]\right)=0 \, .
\end{equation*}
If $\mathrm{y}_\infty=0$ denote by $\nu_\mathrm{y}$ the $\Ad_{(e,-\mathrm{y})}\mathbf{SL}_2(\mathbb{A})$-invariant probability measure on $\left[\mathbf{SL}_2(\mathbb{A})(e,\mathrm{y})\right]$. Then
\begin{equation*}
\limsup_{i\to\infty } \int \mu_i(xB(R_G,R_V))  \dif \nu_\mathrm{y}(x)\ll (\cosh(2R_G)-1)R_V^2 \, ,
\end{equation*}
where $B(R_G,R_V)=B_\infty(R_G,R_V)\cdot \mathbf{P}^1(\hat{\mathbb{Z}})\subset\mathbf{P}^1(\mathbb{A})$ and $B_\infty(R_G,R_V)$ is the product of a Ball of radius $R_G$ on $\mathbf{SL}_2(\mathbb{R})$ and a ball of radius $R_V$ on $\mathbf{V}(\mathbb{R})$
\end{thm}

\subsection{Grids Orthogonal To Integral Points}\label{intro:grids}
As an application of our results we present a partial resolution of a conjecture about joint equidistribution of lattice points on spheres and their orthogonal grids.
\subsubsection{The Orthogonal Complement Construction}
The following construction has been studied extensively by Aka, Einsiedler and Shapira  \cite{AES, AES2} in all dimensions $d\geq 3$.
In some cases this construction has already been investigated by Maass \cite{Maass56,Maass59}, see also the appendix by R.\ Zhang to the arxiv version of \cite{AES}.

Let $d\geq 3$. Denote by $\langle,\rangle\colon \mathbb{R}^d\times\mathbb{R}^d\to\mathbb{R}_{\geq 0}$ the Euclidean inner-product.
For each $D\in\mathbb{N}$ consider the primitive integral points of norm $\sqrt{D}$ in $\mathbb{Z}^d$
\begin{equation*}
\mathscr{H}_D\coloneqq\left\{x\in\mathbb{Z}^d_\mathrm{primitive} \mid \langle x,x \rangle=D \right\} \, .
\end{equation*}
Denote $\mathbb{D}(d)\coloneqq \left\{ D\in\mathbb{N} \mid \mathscr{H}_D\neq \emptyset \right\}$ and assume $D\in\mathbb{D}(d)$.
For each $x\in \mathscr{H}_D$ we denote by $x^\perp(\mathbb{Z})$ the integral lattice orthogonal to $x$
\begin{equation*}
x^\perp(\mathbb{Z})\coloneqq \left\{y\in \mathbb{Z}^d \mid \langle y, x\rangle =0 \right\} \, .
\end{equation*}
We also denote by $x^\perp$ the affine scheme defined over $\mathbb{Q}$ representing the linear subspace of the affine $d$-space perpendicular to $x$, i.e. $x^\perp(F)\coloneqq \{y\in F^d \mid \langle y, x\rangle =0  \}$ for any algebra $F/\mathbb{Q}$.
The group $x^\perp(\mathbb{Z})$ is a lattice of rank $d-1$ in the $d-1$ dimensional space $x^\perp(\mathbb{R})\coloneqq \{y\in\mathbb{R}^d \mid \langle y, x\rangle =0  \}$. The space $x^\perp(\mathbb{R})$ carries a volume form defined by the restriction of the inner-product $\langle,\rangle$.
The covolume of $x^\perp(\mathbb{Z})$ in $x^\perp(\mathbb{R})$ is then $\sqrt{D}$. Let $x^1\in\mathbb{Z}^d$ satisfy $\langle x^1,x \rangle=1$. The point $x^1$ always exists because $x$ is primitive but is not unique. Rather it defines a coset $x^1+x^\perp(\mathbb{Z})$. The orthogonal projection of $x^1+x^\perp(\mathbb{Z})$ to $x^\perp(\mathbb{R})$ is $x^\mathrm{tors}+x^\perp(\mathbb{Z})$ where
\begin{equation*}
x^\mathrm{tors}\coloneqq x^1-\frac{x}{D} \, .
\end{equation*}
It is easy to check that $x^\mathrm{tors}+x^\perp(\mathbb{Z})$ defines a single torsion point of order $D$ in the torus $\faktor{x^\perp(\mathbb{R})}{x^\perp(\mathbb{Z})}$.

Set $\hat{e}\in\mathbb{Z}^d$ to be the unique integral unit vector whose stabilizer in $\mathbf{SO}_d$ is $\mathbf{SO}_{d-1}$. Fix an orientation-preserving isomorphism of inner-product spaces $\hat{e}^\perp(\mathbb{R})\simeq \mathbb{R}^{d-1}$ which sends ${\hat{e}^\perp}_\mathbb{Z}$ to $\mathbb{Z}^{d-1}$. The isomorphism is unique up to composition with an element of $\mathbf{SO}_{d-1}(\mathbb{Z})$.
This isomorphism identifies the space of unimodular lattices of rank $d-1$ in $\hat{e}^\perp(\mathbb{R})$ with $\lfaktor{\mathbf{SL}_{d-1}(\mathbb{Z})}{\mathbf{SL}_{d-1}(\mathbb{R})}$. Moreover, define $\mathbf{ASL}_{d-1}\coloneqq \mathbf{SL}_{d-1}\ltimes \Ga^{\times (d-1)}$ then the space of pairs $\left(L,X\right)$ where $L\subset \hat{e}^\perp(\mathbb{R})$ is a unimodular lattice and $X\in \faktor{\hat{e}^\perp(\mathbb{R})}{L}$ is identified with $\lfaktor{\mathbf{ASL}_{d-1}(\mathbb{Z})}{\mathbf{ASL}_{d-1}(\mathbb{R})}$. In several sources, including \cite{AES2}, pairs $\left(L,X\right)$ as above are named \emph{grids}. We often replace the pair $(L,X)$ by the equivalent datum of the lattice coset $X+L$.

Choosing any element $g\in\mathbf{SO}_d(\mathbb{R})$ such that $g.x=\hat{e}$ we construct a well-defined $\mathbf{SO}_{d-1}(\mathbb{R})$-orbit $\mathbf{SO}_{d-1}(\mathbb{R})g.D^{-\frac{1}{2(d-1)}}\left(x^\perp(\mathbb{Z}), x^\mathrm{tors}\right)$ of a unimodular lattice in $\hat{e}^\perp(\mathbb{R})$ and a torsion point of order $D$. This is a well-defined class in $\dfaktor{\mathbf{ASL}_{d-1}(\mathbb{Z})}{\mathbf{ASL}_{d-1}(\mathbb{R})}{\mathbf{SO}_{d-1}(\mathbb{R})}$.

To summarize to each element $x\in\mathscr{H}_D$ we have associate a class $$\mathrm{Grid}(x)\in \dfaktor{\mathbf{ASL}_{d-1}(\mathbb{Z})}{\mathbf{ASL}_{d-1}(\mathbb{R})}{\mathbf{SO}_{d-1}(\mathbb{R})}$$ corresponding to a lattice of rank $d-1$ in $\hat{e}^\perp(\mathbb{R})$, up to rotation by $\mathbf{SO}_{d-1}(\mathbb{R})$, and a torsion point of order $D$.  This correspondence between $\mathcal{H}_D$ and $\mathbf{SO}_{d-1}(\mathbb{R})$-orbits on $\lfaktor{\mathbf{ASL}_{d-1}(\mathbb{Z})}{\mathbf{ASL}_{d-1}(\mathbb{R})}$ does not depend on any choice involved in the process.

\subsubsection{The Joint Equidistribution Conjecture}
The equidistribution of the sets $D^{-1/2}\mathcal{H}^D$ in the unit sphere $ \mathbf{S}^{d-1}(\mathbb{R})$ for $\mathbb{D}(d)\ni D\to\infty$ is well-known. For $d=3$ this is Duke's theorem \cite{Duke}, see also the pertinent work of Iwaniec \cite{Iwaniec}, and earlier results of Linnik \cite{LinnikSphereRussian,LinnikSphereEnglish,LinnikBook} that required a congruence condition on the sequence $D\to\infty$. For $d\geq 4$ this can be proved by the circle method and is attributed to Kloosterman \cite{Kloosterman4}. The following is a conjecture of Aka, Einsiedler and Shapira \cite{AES,AES2}.
\begin{conj}[Joint Equidistribution]\label{conj:joint-equidistirbution}
Define for each $D\in\mathbb{D}(d)$ the following finite set and Borel probability measure
\begin{align*}
\mathscr{J}_D&\coloneqq \left\{ \left(D^{-1/2} x,\mathrm{Grid}(x)\right) \mid x\in\mathscr{H}_D  \right\}
\subset \mathbf{S}^{d-1}(\mathbb{R})\times \dfaktor{\mathbf{ASL}_{d-1}(\mathbb{Z})}{\mathbf{ASL}_{d-1}(\mathbb{R})}{\mathbf{SO}_{d-1}(\mathbb{R})} \, ,\\
\mu^\mathrm{Grid}_D&\coloneqq \frac{1}{|\mathscr{J}_D|}\sum_{z\in\mathscr{J}_D} \delta_z \, .
\end{align*}
Then when $D\to\infty$ along $\mathbb{D}(d)$ the measures $\mu^\mathrm{Grid}_D$ converge weak-$*$ to the normalized Haar measure on $\mathbf{S}^{d-1}(\mathbb{R})\times \dfaktor{\mathbf{ASL}_{d-1}(\mathbb{Z})}{\mathbf{ASL}_{d-1}(\mathbb{R})}{\mathbf{SO}_{d-1}(\mathbb{R})}$.
\end{conj}
In \cite{AES2} the conjecture has been fully resolved for $d\geq 6$, and assuming a congruence condition at a single prime for $d=4,5$. The method uses measure rigidity for unipotent flows and homogeneous dynamics. It relies heavily on the fact that points in $\mathscr{J}_D$ are equivalent to orbits on a homogeneous space with stabilizer $\mathbf{SO}_{d-1}$ which for $d\geq 4$ is a semi-simple group; hence $\mathbf{SO}_{d-1}(\mathbb{Q}_p)$ is virtually generated by unipotents for all primes $p$ where $\mathbf{SO}_{d-1}$ splits.

\subsubsection{Joint Equidistribution in Three Dimensions}
The case $d=3$ is substantially different due to the fact that $\mathbf{SO}_{d-1}=\mathbf{SO}_2$ is a torus and Ratner's theorems do not apply. In \cite{AES} progress is made for the following weaker conjecture.
\begin{conj}[Weaker Joint Equidistribution]
For $d=3$ consider the push-forward $\mu_D^\mathrm{Lat}$ of $\mu_D^\mathrm{Grid}$ to $\mathbf{S}^2(\mathbb{R})\times \dfaktor{\mathbf{SL}_2(\mathbb{Z})}{\mathbf{SL}_2(\mathbb{R})}{\mathbf{SO}_2(\mathbb{R})}$. This is the normalized counting measure on pairs of points in $D^{-1/2}\mathscr{H}_D$ and their orthogonal lattices, with the order $D$ torsion point discarded. Then as $\mathbb{D}(3)\ni D\to \infty$ the sequence of measures $\mu_D^\mathrm{Lat}$ converges weak-$*$ to the normalized Haar measure on $\mathbf{S}^2(\mathbb{R})\times \dfaktor{\mathbf{SL}_2(\mathbb{Z})}{\mathbf{SL}_2(\mathbb{R})}{\mathbf{SO}_2(\mathbb{R})}$.
\end{conj}

The push-forward of $\mu_D^\mathrm{Grid}$ to $\mathbf{S}^{d-1}(\mathbb{R})$ is obviously the normalized counting measure on $D^{-1/2}\mathscr{H}_D$ whose equidistribution has been established by Duke \cite{Duke}. Less obvious is that the push-forward of $\mu_D^\mathrm{Grid}$ to $\dfaktor{\mathbf{ASL}_2(\mathbb{Z})}{\mathbf{ASL}_2(\mathbb{R})}{\mathbf{SO}_2(\mathbb{R})}$ is a genus orbit of a special point
-- equivalently -- a $K_\infty$-invariant homogeneous toral set for maximal torus $\mathbf{T}<\mathbf{SL}_2$ defined  and anisotropic over $\mathbb{Q}$. This is established in \S \ref{sec:joint}. The fact that $\mu_D^\mathrm{Lat}$ is a packet of torus orbits, i.e.\ projection of a homogeneous toral set to the real quotient, is well-known.

The equidistribution of the push-forward $\mu_D^\mathrm{Lat}$ is known either by analytic methods, see discussion in \S\ref{intro:genus-modular-curve}, or assuming a split prime using Linnik's ergodic method \cite{LinnikBook}. In \cite{AES} these two results in combination with the joining rigidity theorem of Einsiedler and Lindenstrauss \cite{ELJoinings} are used to establish
\begin{thm}[\cite{AES}]
Let $\{D_i\}\subset \mathbb{D}(3)$ such that $D_i\to_{i\to\infty}\infty$.
Assume there are distinct primes $p_1$, $p_2$ such that for all $i$
\begin{equation*}
\left(\frac{D_i}{p_1}\right)=\left(\frac{D_i}{p_2}\right)=1 \, .
\end{equation*}
Then $\mu_{D_i}^\mathrm{Lat}$ converge to the Haar measure on $\mathbf{S}^2(\mathbb{R})\times \dfaktor{\mathbf{SL}_2(\mathbb{Z})}{\mathbf{SL}_2(\mathbb{R})}{\mathbf{SO}_2(\mathbb{R})}$
as $i\to\infty$.
\end{thm}
The congruence condition at $p_1$ and $p_2$ is required for the associated orbits to be invariant under a split torus at two places so that the joining rigidity theorem of \cite{ELJoinings} applies. The original theorem at \cite{AES} required the discriminant to be fundamental, but this restriction is superfluous as the equidistribution theorems on each factor are known for general discriminants. Under the assumption of a fixed split prime they go back to Linnik. A similar result for $d=4,5$ without a congruence condition has been established by \cite{EWR} using effective methods in unipotent dynamics.

We demonstrate in \S \ref{sec:joint} that the push-forward of the full joint measure $\mu_D^\mathrm{Grid}$ to $\dfaktor{\mathbf{ASL}_2(\mathbb{Z})}{\mathbf{ASL}_2(\mathbb{R})}{\mathbf{SO}_2(\mathbb{R})}$ is a genus orbit of a special point. The method of \cite{AES} in conjunction with Theorem \ref{thm:main} imply the following theorem.
\begin{thm}\label{thm:joint}
Let $\{D_i\}\subset \mathbb{D}(3)$ such that $D_i\to_{i\to\infty}\infty$.
Assume there are distinct primes $p_1$, $p_2$ such that for all $i$
\begin{equation*}
\left(\frac{D_i}{p_1}\right)=\left(\frac{D_i}{p_2}\right)=1 \, .
\end{equation*}
Then $\mu_{D_i}^\mathrm{Grid}$ converge weak-$*$ to the Haar measure on $\mathbf{S}^2(\mathbb{R})\times \dfaktor{\mathbf{ASL}_2(\mathbb{Z})}{\mathbf{ASL}_2(\mathbb{R})}{\mathbf{SO}_2(\mathbb{R})}$
as $i\to\infty$.
\end{thm}
Notice that removing the congruence conditions at $p_1$ and $p_2$ in Theorem \ref{thm:main} would not allow us to strengthen Theorem \ref{thm:joint} as this condition is still required for the application of the joining rigidity theorem of \cite{ELJoinings}.

\subsection{Acknowledgments}
I would like to thank Menny Aka, Manfred Einsiedler, Manuel L\"uthi, Peter Sarnak, Yunqing Tang, Akshay Venkatesh, Andreas Wieser and Shou-Wu Zhang for fruitful and enlightening discussions. I am grateful to Asaf Katz for very useful and illuminating comments on a previous version of this manuscript. I am much obliged to the referee for several important suggestions that have improved the exposition.

The author was a Schmidt fellow at the Institute for Advanced Study during the academic year 2017-2018.

\subsection{Organization of the Paper}\hfill\\
In \S \ref{sec:prelim} we introduce rigorously homogeneous toral sets and describe their fundamental properties.\\
In \S \ref{sec:rigidity} we apply measure rigidity theorems of Einsiedler and Lindenstrauss for higher rank toral actions to reduce the equidistribution question to the problem of showing non-concentration on intermediate orbits. In this section we also reduce the general case to $r=1$.\\
In \S \ref{sec:cross-correlation} we describe the geometric expansion of the cross-correlation for Bowen ball test functions, which was first introduced as a tool for controlling concentration in \cite{KhJoint}. We demonstrate that the cross-correlation between a homogeneous toral set and a periodic orbit of $\mathbf{SL}_2$ can be understood in terms of orbits of a compact subgroup of an ad\`elic torus on the unipotent radical.\\
In \S \ref{sec:orbit-space} we introduce arithmetic invariants parameterizing the orbits introduced in the previous section and study their properties.\\
In \S \ref{sec:subconvex} we use the results of the previous two sections to bound the cross-correlation by a short sum over integral ideals in a quadratic order with level structure. The gist of this section is the application of the subconvex bound of Duke, Friedlander and Iwaniec \cite{DFI} to find an asymptotic upper bound for these sums. \\
In \S \ref{sec:equidistribution} we combine results from previous sections to prove the main theorems about equidistribution of homogeneous toral sets and genus orbits.\\
In \S \ref{sec:joint} we present a description of the orthogonal grid correspondence as an intersection of two periodic orbits in a homogeneous space. We show that the orthogonal grids form a genus orbit and prove our theorem about joint equidistribution of lattice points and orthogonal grids.\\
In the appendix \S \ref{app:L-functions} we compute local properties of modified Hecke $L$-functions which are used in \S \ref{sec:subconvex}.

\section{Preliminaries}\label{sec:prelim}
\subsection{Notations}
\begin{enumerate}
\item Algebraic varieties are denoted by bold face letters. They are defined over $\mathbb{Q}$ unless stated otherwise.
\item If $B$ is an algebra over a commutative ring $A$ with a well-defined norm map $\Nr\colon B\to A$, e.g.\ a field extension $E/\mathbb{Q}$, then we denote by $B^\times$ the invertible elements in $B$ and by $B^{(1)}$ the elements of norm $1$ in $B$.
\item If $E/\mathbb{Q}$ is a field extension we denote by $\mathcal{O}_E$ the maximal order in $E$. For any rational place $v<\infty$ we set $E_v\coloneqq E\otimes_\mathbb{Q} \mathbb{Q}_v\simeq \prod_{w\mid v} E_w$. The maximal order in the \'etale-algebra $E_v$ is denoted by $\mathcal{O}_{E_v}$.
\item For a field $E/\mathbb{Q}$ we denote by $\mathbb{A}_E$ the ad\`ele ring of $E$ and let $\mathbb{A}_E^\times$ be the group of invertible ad\`eles. We shall use the subscript $f$ to denote the non-archimedean part, e.g.\ $\mathbb{A}_{E,f}$ and $\mathbb{A}_{E,f}^\times$ are the ring of finite ad\`eles and the group of finite id\`eles respectively.

Abusing the previous notation, let $\mathbb{A}_E^{(1)}$ denote the ad\`eles which are of norm $1$ \emph{everywhere}, i.e.\ $\mathbb{A}_E^{(1)}\coloneqq \prod'_v E_v^{(1)}$ where $v$ runs over the places of $\mathbb{Q}$ and the restricted direct product is with respect to the compact subgroup $\mathcal{O}_{E_v}^{(1)}$ for $v<\infty$.

\item For an affine algebraic group $\mathbf{M}$ defined over $\mathbb{Q}$ we denote $\left[\mathbf{M}(\mathbb{A})\right]\coloneqq \lfaktor{\mathbf{M}(\mathbb{Q})}{\mathbf{M}(\mathbb{A})}$. Moreover, for any subset $K\subset \mathbf{M}(\mathbb{A})$ let $[K]\subset \left[\mathbf{M}(\mathbb{A})\right]$ be the image of $K$ under the quotient map.  For an element $g\in\mathbf{M}(\mathbb{A})$ we denote by $g_v$ the $v$-local part of $g$ for any rational place $v$. We also use the notation $g_f\in\mathbf{M}(\mathbb{A}_f)$ for the non-archimedean part.

\item For $\mathbf{M}$ an affine perfect group defined and anisotropic over $\mathbb{Q}$
we denote either by $\meas_{\mathbf{M}(\mathbb{A})}$ the covolume $1$ Haar measure on $\mathbf{M}(\mathbb{A})$ and by $\meas_\mathbf{M}$ the probability Haar measure on $\left[\mathbf{M}(\mathbb{A})\right]$. For any local field $\mathbb{Q}_v$ denote by $\meas_{\mathbf{M}(\mathbb{Q}_v)}$ a Haar measure on $\mathbf{M}(\mathbb{Q}_v)$.
\end{enumerate}

\subsection{Homogeneous Sets and Periodic Measures}
For brevity we denote henceforth $\mathbf{G}\coloneqq \mathbf{SL}_2$.

\begin{defi}
For any linear subgroup $\mathbf{H}<\mathbf{P}^r$ defined and anisotropic over $\mathbb{Q}$
an $\mathbf{H}$-homogeneous set is a closed subset of $\left[\mathbf{P}^r(\mathbb{A})\right]$ of the form
\begin{equation*}
\mathcal{H}=\left[\mathbf{H}(\mathbb{A})g\right] \, ,
\end{equation*}
where $g\in\mathbf{P}^r(\mathbb{A})$. The homogeneous set $\mathcal{H}$ is invariant under the right action of $\Ad_{g^{-1}}\mathbf{H}(\mathbb{A})$ and supports a unique $\Ad_{g^{-1}}\mathbf{H}(\mathbb{A})$-invariant Borel probability measure which we call the \emph{periodic measure}. This measure is the pushforward by the right translation by $g$ of the Haar measure on $\left[\mathbf{H}(\mathbb{A})\right]$ which is finite because $\mathbf{H}$ is assumed to be anisotropic over $\mathbb{Q}$.

If $\mathbf{H}=\mathbf{T}$ is a maximal torus in $\mathbf{P}^r$ then $\mathcal{H}$ is called a \emph{homogeneous toral set}.
Recall from \S \ref{sec:adelic-homogeneous} that $\mathbf{K}_\infty=\mathbf{SO}_2(\mathbb{R})<\mathbf{G}(\mathbb{R})\hookrightarrow\mathbf{P}^r(\mathbb{R})$ is a fixed maximal compact subgroup.
A homogeneous toral set is $K_\infty$-invariant if $\Ad_{g_\infty^{-1}}\mathbf{T}(\mathbb{R})=K_\infty$.
\end{defi}

\begin{remark}
Notice that for any $\gamma\in \mathbf{P}^r(\mathbb{Q})$ the data $(\mathbf{H},g)$ and $(\Ad_\gamma \mathbf{H}, \gamma g)$ define the same homogeneous set with an identical periodic measure.
\end{remark}
\begin{remark}
A homogeneous toral set $\left[\mathbf{T}(\mathbb{A})(l,\mathrm{x})\right]$ satisfying $\mathbf{T}<\mathbf{G}$ is $K_\infty$-invariant if and only if $\Ad_{l_\infty}^{-1}\mathbf{T}(\mathbb{R})=K_\infty$ in $\mathbf{G}(\mathbb{R})$ and $\mathrm{x}_\infty=0$.
\end{remark}

Our interest lies is homogeneous toral sets in $\mathbf{P}^r$.
Homogeneous sets for the subgroup $\mathbf{G}<\mathbf{P}^1$ will arise in the process of analyzing the possible limits of periodic measures on homogeneous toral sets.

\begin{defi}
Let $\mathcal{H}=\left[\mathbf{T}(\mathbb{A})(l,\mathrm{x})\right]$ be a $K_\infty$-invariant homogeneous toral set.
The projection of $\mathcal{H}$ to $W^r$ is called a \emph{genus packet}. It is a finite collection of CM elliptic curves with torsion points, cf.\ \S \ref{sec:torsion-pts}.

Let $E/\mathbb{Q}$ be the imaginary quadratic extension splitting $\mathbf{T}\simeq \Res^1_{E/\mathbb{Q}} \Gm$.
The main theorem of complex multiplication implies that a genus packet is a single orbit of the Galois subgroup $\Gal(\bar{\mathbb{Q}}/\mathbb{Q}^\mathrm{ab})<\Gal(\bar{\mathbb{Q}}/E)$ which corresponds by Artin reciprocity to the subgroup of the id\`ele class group of $E$ of elements that are everywhere of norm $1$.
\end{defi}

\subsection{Standard Form of Homogeneous Sets}
We show that homogeneous sets for reductive groups can be put into a form that simplifies computations.
\begin{lem}\label{lem:reductive-conj-into-G}
Let $\mathbf{H}<\mathbf{P}^r$ be a reductive algebraic group defined over $\mathbb{Q}$ then $\mathbf{H}$ is conjugate to a subgroup of $\mathbf{G}$ by an element of $\mathbf{P}^r(\mathbb{Q})$
\end{lem}
\begin{proof}
Consider the following composite map
\begin{equation}\label{eq:H-to-SL2-proj}
\mathbf{H}\hookrightarrow \mathbf{P}^r\to\mathbf{G} \, ,
\end{equation}
where the last map is the quotient by the unipotent radical.
The action of $\mathbf{G}$ on $\mathbf{V}^{\oplus r}$ composed with \eqref{eq:H-to-SL2-proj} defines a rational representation of $\mathbf{H}$ on $\mathbf{V}^{\oplus r}$.

Next consider the composite map
\begin{equation}\label{eq:H-cocycle-map}
\mathbf{H}\hookrightarrow \mathbf{P}^r\to\mathbf{V}^{\oplus r} \, ,
\end{equation}
where the last map is the projection onto the second coordinate of $\mathbf{P}^r=\mathbf{G} \ltimes \mathbf{V}^{\oplus r}$. This map is a rational cocycle in $Z^1(\mathbf{H},\mathbf{V}^{\oplus r})$ with respect to the action induced by \eqref{eq:H-to-SL2-proj}. Because $\mathbf{H}$ is assumed to be reductive over $\mathbb{Q}$ its rational cohomology is trivial \cite{Hochschild}. The fact that the cocycle \eqref{eq:H-cocycle-map} is a coboundary implies the claim.
\end{proof}

\begin{defi}
Let $\mathbf{H}<\mathbf{P}^r$ be a reductive subgroup defined and anisotropic over $\mathbb{Q}$.
A \emph{standard form} of an $\mathbf{H}$-homogeneous set $\mathcal{H}$ is a linear  subgroup $\mathbf{H}_0<\mathbf{G}$ that is $\mathbf{P}^r(\mathbb{Q})$-conjugate to $\mathbf{H}$ and group elements $l_\infty\in\mathbf{G}(\mathbb{R})$, $\mathrm{x}_\infty\in\mathbf{V}(\mathbb{R})$, $X\in\mathbf{V}^{\oplus r}(\mathbb{Q})$, $k\in \mathbf{P}^r(\hat{\mathbb{Z}})$ such that
\begin{equation*}
\mathcal{H}=\left[\mathbf{H}_0(e,X)(l_\infty,\mathrm{x}_\infty-X)_\infty k\right] \, ,
\end{equation*}
where $(l_\infty,\mathrm{x}_\infty-X)_\infty$ is an element of $\mathbf{P}^r(\mathbb{R})$.
Notice that if $\mathbf{H}=\mathbf{T}$ is a maximal torus anisotropic over $\mathbb{R}$ then $\mathcal{H}$ is $K_\infty$-invariant if and only if $\mathrm{x}_\infty=0$ and $\Ad_{l_\infty^{-1}}\mathbf{T}(\mathbb{R})=K_\infty$.
\end{defi}

\begin{cor}\label{cor:standard-form}
Every reductive homogeneous set in $\mathbf{P}^r$ has a standard form.
\end{cor}
\begin{proof}
Any reductive homogeneous set is equivalent by Lemma \ref{lem:reductive-conj-into-G} to an $\mathbf{H}_0$-homogeneous set with $\mathbf{H}_0<\mathbf{G}$. The claim follows from strong approximation for $\mathbf{SL}_2$ and $\mathbb{G}_a$.
\end{proof}

\subsection{Volume of Homogeneous Set}
\begin{defi}
Fix a compact $K_\infty$-invariant neighborhood of the identity $\Omega_\infty\subset\mathbf{P}^r(\mathbb{R})$. Let $\mathcal{H}=\left[\mathbf{H}(\mathbb{A})g\right]$ be a homogeneous set in $[\mathbf{P}^r(\mathbb{A})]$ and denote by $\meas_{\mathbf{H}(\mathbb{A})}$ the covolume $1$ Haar measure on $\mathbf{H}(\mathbb{A})$. The volume of $\mathcal{H}$ (with respect to $\Omega_\infty$) is defined as
\begin{equation*}
\vol(\mathcal{H})\coloneqq \meas_{\mathbf{H}(\mathbb{A})}\left(\Ad_g \left(\Omega_\infty \times \mathbf{P}^r(\hat{\mathbb{Z}})\right)\right)^{-1} \, .
\end{equation*}
\end{defi}
\begin{remark} Notice that although the volume depends on the choice of $\Omega_\infty$ we have for any $K_\infty$-invariant compact identity neighborhoods $\Omega_\infty,\Omega_\infty'\subset\mathbf{P}^r(\mathbb{R})$
\begin{equation*}
\vol_{\Omega_\infty}(\mathcal{H}) \ll_{\Omega_\infty,\Omega_\infty'} \vol_{\Omega_\infty'}(\mathcal{H})
\ll_{\Omega_\infty,\Omega_\infty'} \vol_{\Omega_\infty}(\mathcal{H}) \, ,
\end{equation*}
where the implicit constants do not depend on $\mathcal{H}$. Moreover, for a $K_\infty$-invariant homogeneous toral set $\mathcal{H}=\left[\mathbf{T}(\mathbb{A})g\right]$
\begin{equation*}
\vol(\mathcal{H})= \meas_{\mathbf{T}(\mathbb{A}_f)}\left(\Ad_{g_f}  \mathbf{P}^r(\hat{\mathbb{Z}})\right)
\end{equation*}
for any choice of $K_\infty$-invariant identity neighborhood $\Omega_\infty$, where we have normalized $\meas_{\mathbf{T}(\mathbb{A}_f)}$ so that $\meas_{\mathbf{T}(\mathbb{A})}=\meas_{\mathbf{T}(\mathbb{R})}\times \meas_{\mathbf{T}(\mathbb{A}_f)}$ with $\meas_{\mathbf{T}(\mathbb{R})}$ a probability measure.
\end{remark}

\subsection{The Quadratic Order of a Homogeneous Toral Set}\label{sec:order}
\begin{defi}\label{defi:order}
Let $\mathcal{H}=\left[\mathbf{T}(\mathbb{A})(l,\mathrm{x})\right]\subset \left[\mathbf{P}^r(\mathbb{A})\right]$ be a $K_\infty$-invariant homogeneous toral set such that $\mathbf{T}<\mathbf{G}$. The torus $\mathbf{T}<\mathbf{G}$ satisfies $\mathbf{T}\simeq \Res^1_{E/\mathbb{Q}}\Gm$ where $E/\mathbb{Q}$ is the quadratic imaginary extension splitting $\mathbf{T}$. Moreover, there is a commutative algebra scheme $\mathbf{E}<\mathbf{M}_2$ defined over $\mathbb{Q}$ such that $\mathbf{E}(\mathbb{F})\simeq E \otimes_\mathbb{Q} F$ for any algebra $F/\mathbb{Q}$ and $\mathbf{T}=\mathbf{SL}_1(\mathbf{E})$. We fix once and for all an isomorphism $\mathbf{E}(\mathbb{Q})\simeq E$ and identify these fields.
\begin{enumerate}
\item
For any finite place $v<\infty$ define
\begin{equation*}
\Lambda_v\coloneqq \mathbf{E}(\mathbb{Q}_v)\cap \Ad_{l_v} \mathbf{M}_2(\mathbb{Z}_v) \, .
\end{equation*}
The ring $\Lambda_v$ is an order in the quadratic \'{e}tale-algebra $\mathbf{E}(\mathbb{Q}_v)\simeq E\otimes \mathbb{Q}_v$.

\item
Notice that $\Lambda_v$ is the $v$-adic closure of the order $\mathbf{E}(\mathbb{Q})\cap \mathbf{M}_2(\mathbb{Z})$ whenever $l_v\in \mathbf{SL}_2(\mathbb{Z}_v)$. In particular, $\Lambda_v=\mathcal{O}_{E_v}$ for almost all $v$ and we can define the following intersection in $\mathbf{E}(\mathbb{Q})=E$
\begin{equation*}
\Lambda\coloneqq\cap_{v<\infty} \Lambda_v<E \, .
\end{equation*}
The lattice $\Lambda$ is an order in $E$ that is singular exactly at the primes where $\Lambda_v$  is non-maximal.

\item
The discriminant of $\mathcal{H}$ is defined by $\disc(\mathcal{H})\coloneqq \disc(\Lambda)=\prod_v |\disc(\Lambda_v)|_v^{-1}$. Notice that the product is well defined because $|\disc(\Lambda_v)|_v^{-1}=1$ for almost all places $v$.
\end{enumerate}

Notice that the quadratic order depends only on the homogeneous toral set $\left[\mathbf{T}(\mathbb{A})l\right]\subset \left[\mathbf{G(\mathbb{A})}\right]$ which is the image of $\mathcal{H}$ in $\left[\mathbf{G(\mathbb{A})}\right]$.

\begin{prop}\label{prop:complex-torus-cm}
Let $\left[\mathbf{T}(\mathbb{A})l\right]\subset \left[\mathbf{G(\mathbb{A})}\right]$ be a $K_\infty$-invariant homogeneous toral set with quadratic order $\Lambda<E$. The push-forward $\mathcal{P}$ of $\left[\mathbf{T}(\mathbb{A})l\right]$ to  $\dfaktor{\mathbf{G}(\mathbb{Z})}{\mathbf{G}(\mathbb{R})}{K_\infty}$ is a finite set of complex tori of the form $\mathbb{C}/\mathfrak{a}$ where $\mathfrak{a}$ is a proper fractional ideal of $\Lambda$
\end{prop}
\begin{proof}
This claim is standard.
The finiteness result follows from the finiteness of the class number of the torus $\mathbf{T}$ -- essentially -- the finiteness of the class group of an imaginary quadratic field. An argument analogous to \cite[Corollary 4.4]{ELMVPeriodic} for imaginary quadratic fields establishes the second part.
\end{proof}
\end{defi}

\subsubsection{Torsion Points}\label{sec:torsion-pts}
Let $\mathcal{H}=\left[\mathbf{T}(\mathbb{A})(l,\mathrm{x})\right]\subset \left[\mathbf{P}^r(\mathbb{A})\right]$ be a $K_\infty$-invariant homogeneous toral set satisfying $\mathbf{T}<\mathbf{G}$.
We discuss how the element $\mathrm{x}\in\mathbf{V}^{\oplus r}$ gives rise to a tuple of torsion points in $\mathbb{C}/\Lambda$.

\begin{defi}
Fix an integer $k\in\mathbb{N}$.
For any $\mathrm{x}\in\mathbf{V}^{\oplus k}(\mathbb{A})$ we define the order $\ord_{\mathcal{H}}(\mathrm{x})$ to be the torsion order of the non-archimedean part $\mathrm{x}_f\in\mathbf{V}^{\oplus k}(\mathbb{A}_f)$ in the torsion group $\mathbf{V}^{\oplus k}(\mathbb{A}_f)/l_f.\mathbf{V}^{\oplus k}(\hat{\mathbb{Z}})$. Similarly, for any prime $p$ we set $\ord_{\mathcal{H}}(\mathrm{x}_p)$ to be the torsion order of $\mathrm{x}_p$ in $\mathbf{V}^{\oplus k}(\mathbb{Q}_p)/l_p.\mathbf{V}^{\oplus k}(\mathbb{Z})$. We always have
\begin{equation*}
\ord_{\mathcal{H}}(\mathrm{x})=\prod_p \ord_{\mathcal{H}}(\mathrm{x}_p) \, .
\end{equation*}
\end{defi}

\begin{defi}\label{defi:jmath-iso}
\hfill
\begin{enumerate}[label=(\alph*)]
\item
The action of $\mathbf{M}_2$ on $\mathbf{V}$ makes $\mathbf{V}(\mathbb{Q})$ a $1$-dimensional vector space for the field $E=\mathbf{E}(\mathbb{Q})$. We can check locally that the order of $E$ stabilizing the lattice $\cap_{v<\infty} l_v.\mathbb{Z}_v^2\subset \mathbf{V}(\mathbb{Q})$ is exactly $\Lambda$. Any quadratic order is monogenic hence all proper fractional $\Lambda$-ideals are principle. This implies that we can choose a linear $E$-isomorphism $\jmath\colon \mathbf{V}(\mathbb{Q})\to E$ mapping $\cap_{v<\infty} l_v.\mathbb{Z}_v^2$ to $\Lambda$; this isomorphism is unique up to composition with multiplication by an element of $\Lambda^\times$.

\item
The linear isomorphism $\jmath$ induces by base change for all rational places $v$ a linear isomorphism $\jmath_v\colon \mathbf{V}(\mathbb{Q}_v)\to E_v=\mathbf{E}\otimes \mathbb{Q}_v$ sending $l_v.\mathbb{Z}_v^2$ to $\Lambda_v$ for $v<\infty$. These isomorphisms combine to an ad\`elic isomorphism
\begin{equation*}
\jmath_{\mathbb{A}}\colon \mathbf{V}(\mathbb{A})\to \mathbb{A}_E
\end{equation*}
that sends $\mathbf{V}(\mathbb{Q})$ to $E$ and the compact subgroup $l_v.\mathbb{Z}_v^2$ to  $\Lambda_v$ for all $v<\infty$. Again, this isomorphism is defined up to global multiplication by an element of $\Lambda^\times$.

\item
We derive an isomorphism of quotients
\begin{equation}\label{eq:AE-V-iso}
\jmath_{/\Lambda}\colon
\dfaktor{\mathbf{V}(\mathbb{Q})}{\mathbf{V}(\mathbb{A})}{l_f.\prod_{v<\infty}\mathbb{Z}_v^2}
\to
\dfaktor{E}{\mathbb{A}_E}{\prod_{v<\infty} \Lambda_v} \, .
\end{equation}
Fixing one of the two-possible field isomorphisms $E\otimes\mathbb{R}\simeq \mathbb{C}$ the right-hand side of \eqref{eq:AE-V-iso} is naturally identified with the complex torus $\mathbb{C} / \Lambda$.

\item
If $\mathrm{x}=(\mathrm{x}_1,\ldots,\mathrm{x}_r)\in\mathbf{V}(\mathbb{A})^{\oplus r}$ then $$\jmath_{/\Lambda}(\mathrm{x})\coloneqq\left(\jmath_{/\Lambda}(\mathrm{x}_1),\ldots,\jmath_{/\Lambda}(\mathrm{x}_r)\right) \in \left(\mathbb{C}/\Lambda\right)^{\oplus r}$$
is a tuple of $r$ points in $\mathbb{C} / \Lambda$. This tuple is uniquely defined up to  mutual complex conjugation and diagonal multiplication by an element of $\Lambda^\times$. Moreover, if $\mathcal{H}$ is $K_\infty$-invariant then the classes of $\mathrm{x}_i$ in the left-hand side of \eqref{eq:AE-V-iso} have a zero archimedean components, equivalently, $\jmath_{/\Lambda}(\mathrm{x}_i)$ are torsion points in $\mathbb{C} / \Lambda$, i.e.\ $\jmath_{/\Lambda}(\mathrm{x}_i)\in E / \Lambda$. The torsion order of $\jmath_{/\Lambda}(\mathrm{x}_i)$ is equal to $\ord_{\mathcal{H}}(\mathrm{x}_i)$.

\item For any rational place $v$ the isomorphism $\jmath_v\colon\mathbf{V}(\mathbb{Q}_v)\to E_v$ induces and isomorphism
\begin{equation*}
\jmath_v\colon \End_{\mathbb{Q}_v}(E_v)\to \mathbf{M}_2(\mathbb{Q}_v) \, .
\end{equation*}
This isomorphism sends $\End(\Lambda_v)$ to $\Ad_{l_v}\mathbf{M}_2(\mathbb{Z}_v)$ and $\Aut^1(\Lambda_v)$ to $\Ad_{l_v}\mathbf{G}(\mathbb{Z}_v)$, where $\Aut^1(\Lambda_v)$ is the group of $\mathbb{Q}_v$-automorphism of $\Lambda_v$ of determinant $1$.
\end{enumerate}
\end{defi}

\subsection{Local Stabilizers and the Class Number Formula}
\begin{defi}\label{defi:Lambda_v(x)+Pic}
\hfill
\begin{enumerate}
\item
Define for any $v<\infty$
\begin{align*}
\Lambda_v^\times(\mathrm{x})&\coloneqq  \left\{\lambda\in \Lambda_v^\times \mid \lambda.\mathrm{x}_v-\mathrm{x}_v\in \Lambda_v^{\oplus r} \right\}\\
&=\left\{\lambda\in \Lambda_v^\times \mid \lambda.\jmath_{/\Lambda}(\mathrm{x})-\jmath_{/\Lambda}(\mathrm{x})\equiv 0 \mod{\Lambda_v^{\oplus r}} \right\}
\, ,\\
\Lambda_v^{(1)}(\mathrm{x})&\coloneqq \Lambda_v^{(1)}\cap \Lambda_v^\times(\mathrm{x}) \, .
\end{align*}
In the second line above we consider $\jmath_{/\Lambda}(\mathrm{x})$ as an element of $(E_v/\Lambda_v)^{\oplus r}$.
Moreover, set
\begin{align*}
\Lambda_f^\times(\mathrm{x})&\coloneqq  \prod_{v<\infty} \Lambda_v^\times(\mathrm{x}) <\mathbb{A}^\times_{E,f} \, ,\\
\Lambda_f^{(1)}(\mathrm{x})&\coloneqq \prod_{v<\infty} \Lambda_v^{(1)}(\mathrm{x}) < \mathbb{A}^{(1)}_{E,f}
=\mathbf{T}(\mathbb{A}_f) \, .
\end{align*}
\item
Define
\begin{equation*}
\Pic(\Lambda,\mathrm{x})\coloneqq \dfaktor{E^\times}{\mathbb{A}_{E,f}}{\Lambda_f^\times(\mathrm{x})} \, .
\end{equation*}
This is a finite abelian group (a quotient of a ray class group of $E$). For $\jmath_{/\Lambda}(\mathrm{x})=0\mod \Lambda^{\oplus r}$ the group $\Pic(\Lambda, \mathrm{x})$ coincides with the regular Picard group of $\Lambda$ and generally it is a finite extension of it. When $\Lambda=\mathcal{O}_E$ is the maximal order and $r=1$ the group  $\Pic(\mathcal{O}_E, \mathrm{x})$ is a standard ray class group of $E$ with modulus equal to $\left(\mathcal{O}_E:\jmath_{/\Lambda}(\mathrm{x})\right)$.

\item
Set $\Pic^\mathrm{pg}(\Lambda,\mathrm{x})$ to be the image of $\mathbb{A}_{E,f}^{(1)}$ in $\Pic(\Lambda,\mathrm{x})$. As a group
\begin{equation*}
\Pic^\mathrm{pg}(\Lambda,\mathrm{x})\simeq \dfaktor{E^{(1)}}{\mathbb{A}_{E,f}^{(1)}}{\Lambda_f^{(1)}(\mathrm{x})} \, .
\end{equation*}

Notice that when $r=1$ and $\Lambda=\mathcal{O}_E$ is the maximal order $\Pic^\mathrm{pg}(\mathcal{O}_E,\mathrm{x})$ is the principal genus subgroup of the ray class group $\Pic(\mathcal{O}_E, \mathrm{x})$ as defined by Hasse \cite{HassePG}. Moreover when $\jmath_{/\Lambda}(\mathrm{x})\equiv 0 \mod \mathcal{O}_E$ it coincides with the classical principal genus subgroup of Gauss.
\end{enumerate}
\end{defi}

\begin{lem}\label{lem:vol-Lambda_f^{(1)}}
For any place $v<\infty$
\begin{align*}
\Ad_{l_v} \mathbf{G}(\mathbb{Z}_v) \cap \mathbf{T}(\mathbb{Q}_v)&=\Lambda_v^{(1)} \, ,\\
\Ad_{(l_v,\mathrm{x}_v)} \mathbf{P}^r(\mathbb{Z}_v) \cap \mathbf{T}(\mathbb{Q}_v)&=\Lambda_v^{(1)}(\mathrm{x}) \, .
\end{align*}
In particular, $\vol(\mathcal{H})=\meas_{\mathbf{T}(\mathbb{A}_f)}\left(\Lambda_f^{(1)}(\mathrm{x})\right)$.
\end{lem}
\qed

\begin{prop}[Class Number Formula]\label{prop:toral-volume}
Let $\mathbf{T}<\mathbf{G}$ be a maximal torus defined and anisotropic over $\mathbb{Q}$.
The volume of a $K_\infty$-invariant homogeneous toral set $\mathcal{H}=\left[\mathbf{T}(\mathbb{A})(l,\mathrm{x})\right]\subset \left[\mathbf{P}^r(\mathbb{A})\right]$ is equal to
\begin{equation*}
\vol(\mathcal{H})=\frac{w_E}{2\pi\left|\mathcal{O}_E^\times\right|\left[\Pic(\Lambda,\mathrm{x}): \Pic^\mathrm{pg}(\Lambda,\mathrm{x}) \right]}\sqrt{|D|}L(1,\chi_E)\prod_{p\mid f}\left(1-\frac{\chi_E(p)}{p}\right)
\prod_{v<\infty}\left[\Lambda_v^\times \colon \Lambda_v^\times(\mathrm{x})\right] \, ,
\end{equation*}
where $\chi_E$ is the real Dirichlet character attached to the quadratic extension $E/\mathbb{Q}$ by class field theory, $w_E$ is the number of roots of unity in $E$ and $f$ is the conductor of the order $\Lambda$.
\end{prop}
\begin{proof}
Let $\Lambda^\times(\mathrm{x})\coloneqq \left\{\lambda\in \Lambda^\times \mid \lambda\cdot \jmath_{/\Lambda}(\mathrm{x})-\jmath_{/\Lambda}(\mathrm{x})\equiv 0 \mod \Lambda^{\oplus r} \right\}$.
Using Lemma \ref{lem:vol-Lambda_f^{(1)}} and the following short exact sequence
\begin{equation*}
1\to \lfaktor{\Lambda^\times(\mathrm{x})}{\Lambda_f^{(1)}(\mathrm{x})}\to \lfaktor{\mathbf{T}(\mathbb{Q})}{\mathbf{T}(\mathbb{A})}\to
\Pic^\mathrm{pg}(\Lambda,\mathrm{x})\to 1 \, .
\end{equation*}
we deduce that $\vol(\mathcal{H})=\left|\Pic^\mathrm{pg}(\Lambda,\mathrm{x})\right|\left|\Lambda^\times(\mathrm{x})\right|^{-1}$. Hence we can write
\begin{equation}\label{eq:vol(H)-indices}
\vol(\mathcal{H})=\frac{\left|\Pic(\mathcal{O}_E)\right|\left[\Pic(\Lambda):\Pic(\mathcal{O}_E)\right]
\left[\Pic(\Lambda,\mathrm{x}):\Pic(\Lambda)\right]} {\left[\Pic(\Lambda,\mathrm{x}):\Pic^\mathrm{pg}(\Lambda,\mathrm{x})\right]\left|\Lambda^\times(\mathrm{x})\right|} \, .
\end{equation}
To compute the first term in the numerator of \eqref{eq:vol(H)-indices} we use the analytic class number formula
\begin{equation*}
\left|\Pic(\mathcal{O}_E)\right|=\frac{w_E}{2\pi} \sqrt{|D_E|}L(1,\chi_E) \, ,
\end{equation*}
where $D_E$ is the discriminant of $\mathcal{O}_E$.
We evaluate the term $\left[\Pic(\Lambda):\Pic(\mathcal{O}_E)\right]$ using the exact sequence
\begin{equation*}
1\to\dfaktor{\mathcal{O}_E^\times}{\prod_{v<\infty}\mathcal{O}_{E_v}^\times}{\prod_{v<\infty}\Lambda_v^\times}\to\Pic(\Lambda)\to\Pic(\mathcal{O}_E)\to 1
\end{equation*}
which implies using Corollary \ref{cor:units-volume-ratio}
\begin{equation*}
\left[\Pic(\Lambda):\Pic(\mathcal{O}_E)\right]=\frac{\left|\Lambda^\times\right|}{\left|\mathcal{O}_E^\times\right|}
\prod_{v<\infty}\left[\mathcal{O}_{E_v}^\times: \Lambda_v^\times\right]=\frac{\left|\Lambda^\times\right|}{\left|\mathcal{O}_E^\times\right|}
f\prod_{p\mid f}\left(1-\frac{\chi_E(p)}{p}\right) \, .
\end{equation*}
Finally to compute $\left[\Pic(\Lambda,\mathrm{x}):\Pic(\Lambda)\right]$ we consider the short exact sequence
\begin{equation*}
1\to\dfaktor{\Lambda^\times}{\prod_{v<\infty}\Lambda_v^\times}{\prod_{v<\infty}\Lambda_v^\times(\mathrm{x})}
 \to\Pic(\Lambda,\mathrm{x})\to\Pic(\Lambda)\to 1
\end{equation*}
to see that
\begin{equation*}
\left[\Pic(\Lambda,\mathrm{x}):\Pic(\Lambda)\right]=\frac{\left|\Lambda^\times(\mathrm{x})\right|}{\left|\Lambda^\times\right|}
\prod_{v<\infty} \left[\Lambda_v^\times\colon \Lambda_v^\times(\mathrm{x})\right] \, .
\end{equation*}
The final claim follows by combining all the formulae above with \eqref{eq:vol(H)-indices} and the relation $D=D_E f^2$.
\end{proof}

\section{Structure of Limits of Periodic Measures}\label{sec:rigidity}
\subsection{Invariance of Limits}
In the following proposition we show how the congruence assumptions at two primes for the discriminants and torsion order imply that any weak-$*$ limit must be invariant under a split torus at two places. We also present some relevant consequences of Duke's theorem for $\mathbf{SL}_2$.
\begin{prop}\label{prop:limit-invariance}
Let $\left\{\mathcal{H}_i\subset\left[\mathbf{P}^r(\mathbb{A})\right]\right\}_i$ be a sequence of $K_\infty$-invariant homogeneous toral sets. Let $\mu_i$ be the periodic measures on $\mathcal{H}_i$. Denote by $D_i$ and $N_i$ the discriminant and torsion order of $\mathcal{H}_i$. Write $D_i=D_i^\mathrm{fund}f_i^2$ where $D_i^\mathrm{fund}$ is a fundamental discriminant and $f_i$ is the conductor.

Fix $p_1$,$p_2$ and assume for all $i$
\begin{enumerate}
\item
\begin{align*}
\left(\frac{D_i^\mathrm{fund}}{p_1}\right)&=\left(\frac{D_i^\mathrm{fund}}{p_2}\right)=1 \, ,\\
\val_{p_1}(f_i)&,\val_{p_2}(f_i)\ll 1 \, .
\end{align*}
\item
\begin{equation*}
\val_{p_1}(N_i),\val_{p_2}(N_i)\ll 1 \, .
\end{equation*}
\end{enumerate}

Then there is a pre-compact sequence $\left\{\xi_i\right\}_i\subset \mathbf{P}^r(\mathbb{A})$ such that $\mathcal{H}_i \xi_i$ is $A_{p_1}\times A_{p_2}$ invariant for all $i$.
The sequence $\mu_i$ is tight and if $\mu_i\xrightarrow[i\to\infty]{\textrm{weak}-*}\mu$ then there is
$\xi\in \overline{\left\{\xi_i \right\}_i}$ such that the measure $\xi_*.\mu$ is $A_{p_1}\times A_{p_2}$-invariant and projects to the Haar measure on $\left[\mathbf{G}(\mathbb{A})\right]$.
\end{prop}
\begin{proof}
Write $\mathcal{H}_i$ in standard form as $\left[\mathbf{T}_i(e,\mathrm{x}^i)(l_{i,\infty},-\mathrm{x}^i)_\infty \right]\subset\left[\mathbf{P}^r(\mathbb{A})k_i\right]$.
The push-forward of $\mu_i$ to $\left[\mathbf{G}(\mathbb{A})\right]$ is a sequence of periodic measure on homogeneous toral sets with volume going to infinity. A generalization of Duke's theorem \cite{Duke} for homogeneous toral sets in $\left[\mathbf{SL}_2(\mathbb{A})\right]$ implies that the push-forward converges to the Haar measure, cf. discussion in \S\ref{intro:genus-modular-curve} in general or \cite{LinnikBook} assuming a splitting condition at a single prime. Because $\left[\mathbf{P}^r(\mathbb{A})\right]$ is a compact extension of $\left[\mathbf{G}(\mathbb{A})\right]$ the sequence $\{\mu_i\}_i$ is tight and a weak-$*$ limit $\mu$ is necessarily a probability measure whose push-forward to $\left[\mathbf{G}(\mathbb{A})\right]$ is the Haar measure.

Fix $p\in\{p_1,p_2\}$ and $i\in\mathbb{N}$.
The assumption $\left(\frac{D_i^\mathrm{fund}}{p}\right)=1$ implies that $p$ splits at $E=\mathbf{E}_i(\mathbb{Q})$ and $\mathbf{T}_i(\mathbb{Q}_p)=E_p^{(1)}$ is a split rank-$1$ torus. The local discriminant map from the variety of split tori in $\mathbf{G}(\mathbb{Q}_p)$ to $\mathbb{Q}_p^\times$ is proper, cf.\ the definition of local discriminant in \cite{ELMVCubic}. Hence the assumption $\val_p(f_i)\ll 1$ implies that there is a fixed compact set $C_p\subset\mathbf{G}(\mathbb{Q}_p)$, independent of $i$, such that for all $i$ there is some $g_{p,i}\in C_p$ satisfying $\mathbf{T}_i(\mathbb{Q}_p)=\Ad_{g_{p,i}} A_p$.

The assumption $\val_p(N_i)\ll1$ implies that there is some $m_p\geq 0$ such that ${p}^{m_p}\mathrm{x}_{i,p}\in \mathbf{V}(\mathbb{Z}_p)^{\oplus r}$ for all $i$. In particular, for each $i$ there is an element $\mathrm{w}_{p,i}\in p^{-m_p}\mathbf{V}(\mathbb{Z}_{p_1})^{\oplus r}$ satisfying $\Ad_{k_i^{-1}(e,-\mathrm{x}_{i,p})}\mathbf{T}_i(\mathbb{Q}_p)=\Ad_{(g_{p,i},\mathrm{w}_{p,i})}A_{p}$.

Define $\xi_i\in\mathbf{P}^r(\mathbb{A})$ to coincide with $(g_{p,i},\mathrm{w}_{p,i})\in\mathbf{P}^r(\mathbb{Q}_p)$ at the places $p=p_1,p_2$ and set the local component of $\xi_i$ at all other places to be the identity. The sequence $\left\{\xi_i\right\}_i$ is then contained in the compact set $\prod_{p\in\{p_1,p_2\}} C_p\times p^{-k_p}\mathbf{V}(\mathbb{Z}_p)^{\oplus r}$. It obviously satisfies the claimed properties.
\end{proof}

\subsection{Measure Rigidity for \texorpdfstring{$r=1$}{r=1}}
In this section we present the consequences of measure rigidity for higher rank diagonalizable actions to limits of periodic measures of homogeneous toral sets. The main theorem we use from homogeneous dynamics is due to Einsiedler and Lindenstrauss \cite{ELJoinings}. This theorem
establishes that a measure on the homogeneous space of a perfect algebraic group -- invariant under a split torus at two places and whose push-forward to the homogeneous space of the semi-simple part is Haar -- is necessarily algebraic.
\begin{thm}\label{thm:rigidity-r=1}
Let $\mu$ be a probability measure on $\left[\mathbf{P}^1(\mathbb{A})\right]$ such that the push-forward of $\mu$ to $\left[\mathbf{G}(\mathbb{A})\right]$ is the Haar measure.

If $\mu$ is $A_{p_1}\times A_{p_2}$-invariant and ergodic then either $\mu$ is the Haar measure on $\left[\mathbf{P}^1(\mathbb{A})\right]$ or $\mu$ is the periodic measure supported on $\left[\mathbf{G}(\mathbb{A})(e,\mathrm{y}) \right]$ for some $\mathrm{y}\in\mathbf{V}(\mathbb{A})$ such that $\mathrm{y}_{p_1}=\mathrm{y}_{p_2}=0$.
\end{thm}

\begin{cor}\label{cor:structure-of-limit-r=1}
For each $\mathrm{y}\in\mathbf{V}(\mathbb{A})$ denote by $\nu_y$ the periodic measure supported on $\left[\mathbf{G}(\mathbb{A})(e,\mathrm{y})\right]\subset\left[\mathbf{P}^1(\mathbb{A})\right]$.

Assume $r=1$ in the setting of Proposition \ref{prop:limit-invariance} then there is a Borel probability measure $\mathcal{P}$ on $\mathbf{V}(\mathbb{A})$ and $c\geq0$ such that
\begin{equation*}
\xi_*.\mu= (1-c)\meas_{\mathbf{P}^1}+c\int_{\mathbf{V}(\mathbb{A})} \nu_\mathrm{y} \dif\mathcal{P}(\mathrm{y})
\end{equation*}
and $\mathrm{y}_{p_1}=\mathrm{y}_{p_2}=0$ for $\mathcal{P}$-almost every $\mathrm{y}\in\mathbf{V}(\mathbb{A})$.
\end{cor}
\begin{proof}
Follows by applying Theorem \ref{thm:rigidity-r=1} to the $A_{p_1}\times A_{p_2}$ ergodic decomposition of $\xi_*.\mu$.
\end{proof}

Before proving Theorem \ref{thm:rigidity-r=1} we need to show the following standard lemma which is an application of Goursat's Lemma.
\begin{lem}\label{lem:P1-finite-index}
Let $S$ be a finite set of rational places. The groups $\mathbf{P}^1(\mathbb{Q}_S)$ and $\mathbf{SL}_2(\mathbb{Q}_S)$ contain no non-trivial closed subgroups of finite index.
\end{lem}
\begin{proof}
Let $v$ be any rational place. The only non-trivial normal subgroup of $\mathbf{SL}_2(\mathbb{Q}_v)$ is $\mu_2(\mathbb{Q}_v)$ \cite{DicksonPSLn}, hence it has no non-trivial subgroups of finite index. The abelian group $\mathbf{V}(\mathbb{Q}_v)\simeq \mathbb{Q}_v^2$ also has no non-trivial closed finite index subgroups. This is equivalent to the statement that the Pontryagin dual of $\mathbb{Q}_v^2$ has no closed torsion subgroups which is the case because $\mathbb{Q}_v^2$ is self-dual.

We would like to deduce that $\mathbf{P}^1(\mathbb{Q}_v)$ has no non-trivial closed subgroups of finite index. If $H_0<\mathbf{P}^1(\mathbb{Q}_v)$ is closed and has finite index then $H_0$ surjects onto $\mathbf{SL}_2(\mathbb{Q}_v)$. Hence if $\omega_1,\ldots, \omega_n$ are representatives for the classes of $\faktor{\mathbf{P}^1(\mathbb{Q}_v)}{H_0}$ then there are some $h_1,\ldots,h_n\in H_0$ so that for all $i$ the image of $\omega_i h_i$ in $\mathbf{SL}_2(\mathbb{Q}_v)$ is the identity. The elements $\omega_1h_1,\ldots,\omega_n h_n$ are then also representatives for the classes of $\faktor{\mathbf{V}(\mathbb{Q}_v)}{\mathbf{V}(\mathbb{Q}_v)\cap H_0}$. We deduce that $\mathbf{V}(\mathbb{Q}_v)\cap H_0=\mathbf{V}(\mathbb{Q}_v)$ and $H_0=\mathbf{P}^1(\mathbb{Q}_v)$.

Finally, set $\mathbf{H}$ to be either $\mathbf{SL}_2$ or $\mathbf{P}^1$ and  let $H<\mathbf{H}(\mathbb{Q}_S)$ be a closed subgroup of finite index. We  prove the claim by induction on the size of $S$, the case of $S$ being a singleton has been already demonstrated. Fix $v\in S$ and set $G_1=\mathbf{H}(\mathbb{Q}_v)$ and $G_2=\prod_{v\neq s \in S} \mathbf{H}(\mathbb{Q}_s)$. The projection of $H$ on $G_1$ and $G_2$ is a finite index subgroup, hence it is surjective by the induction hypothesis. Write $\prod_{s\in S} \mathbf{H}(\mathbb{Q}_s)\simeq G_1\times G_2$. Goursat's Lemma implies that if $N_1\coloneqq G_1\times\{e\}\cap H$ and $N_2\coloneqq \{e\}\times G_2\cap H$ then the image of $H$ in $\faktor{G_1}{N_1}\times\faktor{G_2}{N_2}$ is the graph of a group isomorphism. This image is a finite index subgroup, which implies that $\faktor{G_1}{N_1}\simeq \faktor{G_2}{N_2}$ are finite groups. Because all finite-index subgroups of $G_1$ and $G_2$ are the whole group we deduce that $G_1\times\{e\},\{e\}\times G_2\subset H$ and $H=G_1\times G_2$ as claimed.
\end{proof}

\begin{proof}[Proof of Theorem \ref{thm:rigidity-r=1}]
Let $S$ be any set of places for $\mathbb{Q}$ containing $\infty$, $p_1$, $p_2$ and set
\begin{equation}\label{eq:Ys-define}
W^1_S\coloneqq \lfaktor{\Gamma^1}{\mathbf{P}^1(\mathbb{Q}_S)}
\simeq \dfaktor{\mathbf{P}^1(\mathbb{Q})}{\mathbf{P}^1(\mathbb{A})}{K^S} \, ,
\end{equation}
where $K^S\coloneqq \prod_{v\not \in S} \mathbf{P}^1(\mathbb{Z}_v)$ and $\Gamma^1\coloneqq \mathbf{P}^1(\mathbb{Q})\cap K^S$ is a congruence lattice embedded diagonally in $\mathbf{P}^1(\mathbb{Q}_S)=\prod_{s\in S} \mathbf{P}^1(\mathbb{Q}_s)$. The isomorphism in \eqref{eq:Ys-define} holds because $\mathbf{P}^1=\mathbf{SL}_2\ltimes \Ga^{\times 2}$ has the strong approximation property. This follows from strong approximation for $\mathbf{SL}_2$ and $\Ga$.
Denote by $\mu_S$ the push-forward of the measure $\mu$ to $W^1_S$ under the quotient by $K^S$. This is an $A_{p_1}\times A_{p_2}$-invariant and ergodic probability Borel measure on $W^1_S$.

There is a quotient map
\begin{equation*}
W^1_S\to Y_S\coloneqq \lfaktor{\Gamma}{\mathbf{SL}_2(\mathbb{Q}_S)} \, ,
\end{equation*}
where $\Gamma\coloneqq \mathbf{SL}_2\left(\mathbb{Q}\left[\prod_{\infty\neq p\in S}\frac{1}{p}\right]\right)=\mathbf{SL}_2(\mathbb{Q})\cap \prod_{v\not\in S} \mathbf{SL}_2(\mathbb{Z}_v)$. The push-forward of $\mu_S$ to $Y_S$ is the Haar measure.

Apply \cite[Theorem 1.6]{ELJoinings} to deduce that the measure
$\mu_S$ is the algebraic measure supported on $[Lg]$ where $L<\mathbf{L}(\mathbb{Q}_S)$ is a closed finite index subgroup, $\mathbf{L}<\mathbf{P}^1$ is an algebraic subgroup defined over $\mathbb{Q}$, $g\in\mathbf{P}^1(\mathbb{Q}_S)$ and $A_{p_1}\times A_{p_2} <g^{-1} L g$.  Because the push-forward of $\mu_S$ to $Y_S$ is the Haar measure we deduce that the map from $\mathbf{L}$ to $\mathbf{SL}_2$ is surjective.

We have a short exact sequence
\begin{equation*}
1\to \mathbf{V}\cap \mathbf{L}\to \mathbf{L}\to\mathbf{SL}_2\to 1 \, .
\end{equation*}
Hence $\mathbf{V}\cap\mathbf{L}$ is the radical of $\mathbf{L}$ and $\mathbf{SL}_2$ is its semisimple factor. In particular, $\mathbf{L}\simeq \mathbf{SL}_2\ltimes \left(\mathbf{V}\cap \mathbf{L}\right)$ and $\mathbf{V}\cap \mathbf{L}$ is an $\mathbf{SL}_2$ sub-representation of $\mathbf{V}$. The only possible sub-representations are either $0$ or $\mathbf{V}$ itself. In the latter case $\mathbf{L}\simeq \mathbf{P}^1$ and because $\mathbf{P}^1$ is connected this implies $\mathbf{L}=\mathbf{P}^1$.

If $\mathbf{V}\cap\mathbf{L}=0$ then $\mathbf{L}\simeq \mathbf{SL}_2$. By Lemma \ref{lem:reductive-conj-into-G} the subgroup $\mathbf{L}$ is then conjugate to $\mathbf{G}$ by an element of $\mathbf{P}^1(\mathbb{Q})$

Whether $\mathbf{L}\simeq\mathbf{G}$ or $\mathbf{L}=\mathbf{P}^1$ Lemma \ref{lem:P1-finite-index} implies that $L=\mathbf{L}(\mathbb{Q}_S)$.
Taking an inverse limit over $S$ by a standard argument, cf.\ \cite[Proof of Theorem 4.4]{KhJoint}, we deduce that either $\mu$ is the Haar measure or it is the periodic measure on $\left[\mathbf{G}(\mathbb{A})\xi \right]$ for some $\xi\in\mathbf{P}^1(\mathbb{A})$ such that $A_p<\Ad_{\xi_p}^{-1}\mathbf{G}(\mathbb{Q}_p)$ for all $p\in\{p_1,p_2\}$.  Replacing $\xi$ by $(g,e)\xi$ for some $g\in\mathbf{G}(\mathbb{A})$ we can assume without loss of generality that $\xi=(e,\mathrm{y})$ for some $\mathrm{y}\in\mathbf{V}(\mathbb{A})$. The condition $A_{p_{1,2}}<\Ad_{\xi_{p_{1,2}}^{-1}}\mathbf{G}(\mathbb{Q}_{p_{1,2}})$ implies $\mathrm{y}_{p_1}=\mathrm{y}_{p_2}=0$.
\end{proof}

\subsection{Reduction to the \texorpdfstring{$r=1$}{r=1} Case}
In this section we show how the Theorem \ref{thm:main} for $r\geq 2$ reduces to the case of $r=1$. The proof is another application of measure rigidity for higher rank diagonal actions.
\begin{defi}
For each $0\neq m=\left(m_1,\ldots,m_r\right)\in\mathbb{Q}^r$ define a surjective homomorphism of algebraic groups over $\mathbb{Q}$
\begin{equation*}
\operatorname{\pi}_{m}\colon\mathbf{P}^{r}\to\mathbf{P}^1
\end{equation*}
by
\begin{equation*}
\operatorname{\pi}_{m}(g,y_1,\ldots,v_r)=\left(g,\sum_{k=1}^r m_k y_k\right) \, .
\end{equation*}
\end{defi}

\begin{thm}\label{thm:rigidity-r>1}
Let $\mu$ be a probability measure on $\left[\mathbf{P}^r(\mathbb{A})\right]$ such that for each $0\neq m\in\mathbb{Q}^r$ the push-forward $\pi_{m*}\mu$ is the Haar measure on $\left[\mathbf{P}^{1}(\mathbb{A})\right]$.
If $\mu$ is $A_{p_1}\times A_{p_2}$-invariant then $\mu$ is the Haar measure on $\left[\mathbf{P}^r(\mathbb{A})\right]$.
\end{thm}

\begin{lem}\label{lem:Haar-ergodic-r=1}
The Haar measure $m_{\mathbf{P}^1}$ on $[\mathbf{P}^1(\mathbb{A})]$ is $A_{p_1}$-ergodic.
\end{lem}

\begin{proof}
Let $a\in A_{p_1}$ be an element that generates an unbounded subgroup.
The group $\mathbf{P}^1(\mathbb{Q}_{p_1})$ is topologically generated by the stable and unstable horospherical subgroups of $a$. The Mautner phenomena implies that any $a$-invariant vector in $L^2([\mathbf{P}^1(\mathbb{A})],m_{\mathbf{P}^1})$ is $\mathbf{P}^1(\mathbb{Q}_{p_1})$-invariant. A standard argument using strong approximation, cf.\cite[Lemma 3.22]{GMO}, says that any  $\mathbf{P}^1(\mathbb{Q}_{p_1})$-invariant vector is in $\mathbb{C}\cdot 1$.
\end{proof}

\begin{lem}\label{lem:strict-subrepresentation}
Let $\mathbf{V}_0<\mathbf{V}^{\oplus r}$ be a rational $\mathbf{SL}_2$ sub-representation. If $\pi_m(\mathbf{V}_0)=\mathbf{V}$ for all $0\neq m\in \mathbb{Q}^r$ then $\mathbf{V}_0=\mathbf{V}^{\oplus r}$.
\end{lem}

\begin{proof}
The endomorphism ring $\End_\mathbf{G}(\mathbf{V})$ is by definition $\mathbf{Z}_{\mathbf{M}_2(\mathbb{Q})}(\mathbf{SL}_2(\mathbb{Q}))\simeq \mathbb{Q}$. In particular $\Hom_\mathbf{G}(\mathbf{V}^{\oplus k},\mathbf{V})\simeq \End_\mathbf{G}(\mathbf{V})^k \simeq \mathbb{Q}^k$ for any $k\in\mathbb{N}$.
The map $m\mapsto \pi_m$ is an injective $\mathbb{Q}$-linear map from $\mathbb{Q}^r$ to $\Hom_\mathbf{G}(\mathbf{V}^{\oplus r},\mathbf{V})$. Hence it is also an isomorphism.

The inclusion $\mathbf{V}_0<\mathbf{V}$ induces a linear map
\begin{equation*}
\Hom_\mathbf{G}(\mathbf{V}^{\oplus r},\mathbf{V})\to \Hom_\mathbf{G}(\mathbf{V}_0,\mathbf{V}) \, .
\end{equation*}
The assumption in the claim implies that this map is injective. This implies $\dim_{\mathbb{Q}} \Hom_\mathbf{G}(\mathbf{V}_0,\mathbf{V})\geq r$ and $\dim_{\mathbb{Q}}\mathbf{V}_0\geq 2r$ which completes the proof.
\end{proof}

\begin{proof}[Proof of Theorem \ref{thm:rigidity-r>1}]
Let $$\mu=\int_{\left[\mathbf{P}^r(\mathbb{A})\right]} \mu_x \dif\mu(x)$$ be the $A_{p_1}\times A_{p_2}$ ergodic decomposition of $\mu$. For any $0\neq m \in \mathbb{Q}^r$ the push-forward by $\pi_m$ of the ergodic decomposition to $[\mathbf{P}^1(\mathbb{A})]$ is a decomposition of the Haar measure to $A_{p_1}$-invariant measures. The Haar measure on $[\mathbf{P}^1(\mathbb{A})]$ is $A_{p_1}$-ergodic by Lemma \ref{lem:Haar-ergodic-r=1}.
This implies that for almost every $x$ the push-forward of $\mu_x$ to $[\mathbf{P}^1(\mathbb{A})]$ is Haar. In particular, for almost every $x$ the push-forward of $\mu_x$ to $[\mathbf{G}(\mathbb{A})]$ is Haar.

By the same argument as in the proof of Theorem \ref{thm:rigidity-r=1} we deduce that almost each $\mu_x$ is the invariant measure supported on $[\mathbf{L}(\mathbb{A})\xi]$ where $\mathbf{L}<\mathbf{P}^r$ is an algebraic subgroup  defined over $\mathbb{Q}$, $\xi\in\mathbf{P}^{r}(\mathbb{A})$ and $\mathbf{L}$ surjects onto $\mathbf{P}^1$ under $\pi_m$ for any $0\neq m\in\mathbb{Q}^r$. Consequentially, $\mathbf{L}\simeq \mathbf{SL}_2\ltimes \mathbf{V}_0$ where $\mathbf{V}_0<\mathbf{V}^{\oplus r}$ is an $\mathbf{SL}_2$ sub-representation such that $\pi_m(\mathbf{V}_0)=\mathbf{V}$ for any $m\neq 0$ and Lemma \ref{lem:strict-subrepresentation} implies $\mathbf{V}_0=\mathbf{V}^{\oplus r}$.
\end{proof}

\begin{cor}\label{cor:reduction-to-r=1}
Assume Theorem \ref{thm:main} holds for $r=1$ then Theorem \ref{thm:main} holds for $r>1$.
\end{cor}
\begin{proof}
We use the notations of Theorem \ref{thm:main} and write $\mathcal{H}_i=\left[\mathbf{T}_i(\mathbb{A})(l_i,\mathrm{x}^i)\right]$ where $\mathrm{x}\in\mathbf{V}^{\oplus r}(\mathbb{A})$. Consider for any $0\neq m\in\mathbb{Q}^r$ the sequence of push-forward  measures $\pi_{m*}.\mu_i$. These are the periodic measures on the homogeneous toral sets
\begin{equation*}
\pi_{m}(\mathcal{H}_i)=\left[\mathbf{T}_i(\mathbb{A})(l_i,\langle m,\mathrm{x}^i\rangle)\right]
\subset \left[\mathbf{P}^1(\mathbb{A})\right] \, ,
\end{equation*}
where $\langle m,\mathrm{x}^i\rangle=\sum_{k=1}^r m_k\mathrm{x}^i_k\in\mathbf{V}(\mathbb{A})$.
We would like to apply Theorem \ref{thm:main} with $r=1$ to the sequence $\left\{\pi_{m}(\mathcal{H}_i)\right\}_i$; we need to verify the strictness assumptions and the congruence conditions \eqref{eq:congruence-1} and \eqref{eq:congruence-2}.

The homogeneous toral set $\pi_{m}(\mathcal{H}_i)$ has the same discriminant as $\mathcal{H}_i$, in particular, its discriminant goes to infinity and congruence condition \eqref{eq:congruence-1} holds. To verify congruence condition \eqref{eq:congruence-2} for $p\in\{p_1,p_2\}$ notice first
\begin{align*}
\ord_{\mathcal{H}_i}
\left(\mathrm{x}^i_{1,f},\ldots,\mathrm{x}^i_{r,f}\right)&=
\operatorname{lcm}\left(\ord_{\mathcal{H}_i} (\mathrm{x}^i_{1,f}),\ldots,
\ord_{\mathcal{H}_i}(\mathrm{x}^i_{r,f})\right) \, ,\\
\ord_{\pi_{m}(\mathcal{H}_i)} \langle m,\mathrm{x}^i\rangle&\leq \operatorname{lcd}(m_1,\ldots,m_r)
\operatorname{lcm}\left(\ord_{\mathcal{H}_i} (\mathrm{x}^i_{1,f}),\ldots,
\ord_{\mathcal{H}_i}(\mathrm{x}^i_{r,f})\right) \, ,
\end{align*}
where $\operatorname{lcd}$ is the lowest common denominator. We deduce that for any fixed $m$ the congruence condition \eqref{eq:congruence-2} holds. The strictness assumption for  $\left\{\pi_{m}(\mathcal{H}_i)\right\}_i$ follows immediately from the strictness assumption for $\left\{\mathcal{H}_i\right\}_i$. The assumption that Theorem \ref{thm:main} holds for $r=1$ now implies that $\pi_{m,*}.\mu_i$ converge weak-$*$ to $\meas_{\mathbf{P}^1}$.

Let $\mu$ be any weak-$*$ limit point of $\mu_i$. Proposition \ref{prop:limit-invariance} says that $\mu$ is a probability measure and there is some $\xi\in\mathbf{P}^r(\mathbb{A})$ such that $\xi_*.\mu$ is $A_{p_1}\times A_{p_2}$-invariant. The discussion above implies for any $0\neq m\in\mathbb{Q}^r$ that $\pi_{m*}.\mu$ is the Haar measure, hence $\pi_{m*}.\xi_*.\mu=\pi_m(\xi)_*.\pi_{m,*}.\mu$ is also the Haar measure on $\left[\mathbf{P}^1(\mathbb{A})\right]$. Theorem \ref{thm:rigidity-r>1} now implies that $\xi_*.\mu$ is the Haar measure on $\left[\mathbf{P}^r(\mathbb{A})\right]$; thus any weak-$*$ limit point of $\{\mu_i\}_i$ is the Haar measure.
\end{proof}

Following this result we shall be interested henceforth only in homogeneous toral sets in $\left[\mathbf{P}^1(\mathbb{A})\right]$.

\section{Geometric Expansion of the Cross-correlation}\label{sec:cross-correlation}
In this section we discuss the cross-correlation and its geometric expansion as a relative trace. This is our main tool in excluding the possibility in Theorem \ref{thm:rigidity-r=1} of concentration on periodic orbits of $\mathbf{G}(\mathbb{A})$.

\subsection{Standing Assumptions}\label{sec:rt-assumptions}
Throughout this section we fix a $K_\infty$-invariant $\mathbf{T}$-homogeneous set $\mathcal{H}=\left[\mathbf{T}(\mathbb{A})(l,\mathrm{x})\right]\subset \left[\mathbf{P}^1(\mathbb{A})\right]$ such that $\mathbf{T}<\mathbf{G}$ is a maximal torus.
We denote by $\mu$ the periodic measure on $\mathcal{H}$. In addition, we fix a $\mathbf{G}$-homogeneous set $[\mathbf{G}(\mathbb{A})(e,\mathrm{y})]\subset \left[\mathbf{P}^1(\mathbb{A})\right]$ and denote by $\nu$ the periodic measure supported on it.

\subsection{Geometric Expansion}
\begin{defi} \label{defi:cross-correlation}
\hfill
\begin{enumerate}
\item
For any compactly supported bounded measurable function $f\colon \mathbf{P}^1(\mathbb{A})\to \mathbb{C}$ define\\ $K_f\colon \left[\mathbf{P}^1(\mathbb{A})\right]^{\times 2}\to \mathbb{C}$ by
\begin{equation*}
K_f(x,y)=\sum_{\gamma\in\mathbf{P}^1(\mathbb{Q})} f(x^{-1} \gamma y) \, .
\end{equation*}
Notice that for any compact subset $C\subset \mathbf{P}^1(\mathbb{A})$ if $x,y\in C$ then all summands above vanish except perhaps for the finitely many summands corresponding to $\mathbf{P}^1(\mathbb{Q}) \cap C \supp(f) C^{-1}$. In particular, the function $K_f$ is bounded on compact sets.
\item For any Borel probability measures $\lambda_1,\lambda_2$ on $\left[\mathbf{P}^1(\mathbb{A})\right]$ and $f$ as above set
\begin{equation*}
\Cor(\lambda_1,\lambda_2)[f]\coloneqq\int \dif \lambda_1(x) \int \dif \lambda_2(y) K_f(x,y)
\end{equation*}
whenever the integral is defined. We will be interested exclusively in non-negative real functions for which the integral is always defined and takes values in $\mathbb{R}\cup \{\infty\}$.
\end{enumerate}
\end{defi}
The cross-correlation is closely related to the notion of ``pair correlation'' from statistical mechanics. Indeed, if $B\subset \mathbf{P}^1(\mathbb{A})$ is a symmetric neighborhood of the identity and $f=\mathds{1}_B$ then it is immediate to see that
\begin{equation*}
\Cor(\lambda_1,\lambda_2)[f]\leq \lambda_1\times \lambda_2 \left(\left\{(x,y)\in \left[\mathbf{P}^1(\mathbb{A})\right] \mid x\in yB \right\}\right)\;.
\end{equation*}
Moreover, the inequality above is actually an equality if the map $x\mapsto xB$ is injective on $\supp \lambda_1 \cap \supp \lambda_2$. Classical ``pair-correlation'' functions can be recovered from this definition when $\lambda_1=\lambda_2$ and $f$ is a suitable potential function.  For us, it will be important to work with different measures $\lambda_1$, $\lambda_2$ and the test function $f$ will be the characteristic function of a Bowen ball. Probabilistically, the cross-correlation then measures that the probability that $x$ is $B$-close to $y$ when $x$ is chosen according to the law $\lambda_1$ and $y$ is chosen according to $\lambda_2$ independently of $x$.

The reason the cross-correlation is useful to us is that this probabilistic quantity can actually be studied by methods from automorphic forms. When $\lambda_1$ and $\lambda_2$ are periodic measures then the cross-correlation is a relative trace in disguise.

The cross-correlation between two different measures has been first introduced in \cite{KhJoint} as a tool to establish non-accumulation of periodic orbits on intermediate measures. It has been inspired by the relation between the self-correlation of a measure, the rate of large deviations and the Kolmogorov-Sinai entropy. This connection is implicit already in the work of Linnik \cite{LinnikBook} as the translation between his basic lemma and the rate of large deviations is essentially the geometric expansion of a relative trace.

\begin{prop}\label{prop:geom-expansion}
Fix $f\colon\mathbf{P}^1(\mathbb{A})\to \mathbb{R}_{\geq 0}$ measurable and compactly supported.
Let $\mu$ and $\nu$ be the periodic measures from \S\ref{sec:rt-assumptions}. Then
\begin{equation*}
\Cor(\mu,\nu)[f]=\RO_f(0)+
\sum_{[0]\neq [\mathrm{v}]\in\mathbf{T}(\mathbb{Q}) \backslash \mathbf{V}(\mathbb{Q})} \RO_f(\mathrm{v}) \, .
\end{equation*}
Where we denote
\begin{equation*}
\RO_f(\mathrm{v})\coloneqq
\begin{cases}
\int_{\mathbf{T}\times \mathbf{G}\, (\mathbb{A})} \dif(t,g)\,
f([t(l,\mathrm{x})]^{-1}(e,\mathrm{v})g(e,\mathrm{y})) & \mathrm{v}\neq 0\\
\int_{\mathbf{G}(\mathbb{A})} \dif g\, f((l,\mathrm{x})^{-1}g(e,\mathrm{y})) & \mathrm{v}=0
\end{cases}
\, .
\end{equation*}
\end{prop}
\begin{remark}
Notice that a change of variables $g\mapsto l^{-1}t^{-1}g$ and $t\mapsto t^{-1}$ implies for $\mathrm{v}\neq 0$
\begin{equation}\label{eq:RO-short-v}
\RO_f(\mathrm{v})=\int_{\mathbf{T}\times \mathbf{G}\, (\mathbb{A})} \dif(t,g)
f(g, g.\mathrm{y}+l^{-1}t.\mathrm{v}-l^{-1}.\mathrm{x}) \, .
\end{equation}
Similarly,
\begin{equation}\label{eq:RO-short-0}
\RO_f(0)= \int_{\mathbf{G}(\mathbb{A})} \dif g
f(g, g.\mathrm{y}-l^{-1}.\mathrm{x}) \, .
\end{equation}
\end{remark}

For the proof of Proposition \ref{prop:geom-expansion} we need to understand elementary properties of the group action from the following definition.

\begin{defi} \label{defi:M-action}
Set $\mathbf{M}\coloneqq \mathbf{T}\times\mathbf{G}$. We let $\mathbf{M}$ act on $\mathbf{P}^1$ using the left action of $\mathbf{T}$ on $\mathbf{P}^1$ and the right action of $\mathbf{G}$.
\end{defi}

\begin{lem}\label{lem:M-stab}
Let $\gamma\in\mathbf{P}^1(\mathbb{Q})$ then the stabilizer subgroup for the action from Definition \ref{defi:M-action} is
\begin{equation*}
\mathbf{M}_\gamma\simeq \begin{cases}
e & \gamma\not\in \mathbf{G}(\mathbb{Q})\\
\mathbf{T} & \gamma\in \mathbf{G}(\mathbb{Q})
\end{cases} \, ,
\end{equation*}
where in the second case the isomorphism is given by $t\mapsto (t,\gamma^{-1} t \gamma)$.
\end{lem}
\begin{proof}
The stabilizer of any non-zero point of $\mathbf{V}$ under the $\mathbf{G}=\mathbf{SL}_2$ action is a unipotent subgroup, hence it intersects $\mathbf{T}$ trivially. In particular, $\mathbf{T}$ acts faithfully on $\mathbf{V}$. The formula for the stabilizer is an immediate computation using this fact.
\end{proof}

\begin{proof}[Proof of Proposition \ref{prop:geom-expansion}]
This is a geometric expansion of a relative trace. The situation is rather simple because for all $\gamma\in\mathbf{P}^1(\mathbb{Q})$ the stabilizers $\mathbf{M}_\gamma$ are isotropic over $\mathbb{Q}$, i.e.\ all $\gamma$ are elliptic. In what follows the restrictive assumption that $f$ is non-negative renders all series absolutely convergent.

Let $f_0\colon \mathbf{P}^1(\mathbb{A})\to \mathbb{C}$ be defined by $f_0(h)\coloneqq f((l,\mathrm{x})^{-1} h (e,\mathrm{y}))$ and for each coset $[m_{\mathbb{Q}}]\in\lfaktor{\mathbf{M}_\gamma}{\mathbf{M}}(\mathbb{Q})$ fix an arbitrary representative $m_{\mathbb{Q}}\in\mathbf{M}(\mathbb{Q})$. Then Tonelli's theorem for non-negative functions implies
\begin{align*}
\Cor(\mu,\nu)[f]&=\int_{\left[\mathbf{T}(\mathbb{A})\right]} \dif t \int_{\left[\mathbf{G}(\mathbb{A})\right]} \dif g\, K_{f_0}(t,g)\\
&=\sum_{\gamma\in \mathbf{P}^1(\mathbb{Q})}
\int_{\left[\mathbf{T}(\mathbb{A})\right]} \dif t \int_{\left[\mathbf{G}(\mathbb{A})\right]} \dif g\, f_0(t^{-1}\gamma g)
=\sum_{\gamma\in \mathbf{P}^1(\mathbb{Q})}
\int_{\left[\mathbf{M}(\mathbb{A})\right]} \dif m \, f_0(m^{-1}.\gamma)\\
&=\sum_{[\gamma]\in \mathbf{T}(\mathbb{Q}) \backslash \mathbf{P}^1(\mathbb{Q}) \slash \mathbf{G}(\mathbb{Q})}
\sum_{m_{\mathbb{Q}}\in \mathbf{M}_\gamma \backslash \mathbf{M} (\mathbb{Q})}
\int_{\left[\mathbf{M}(\mathbb{A})\right]} \dif m \, f_0(m^{-1} m_{\mathbb{Q}}.\gamma)\\
&=\sum_{[\gamma]\in \mathbf{T}(\mathbb{Q}) \backslash \mathbf{P}^1(\mathbb{Q}) \slash \mathbf{G}(\mathbb{Q})}
\sum_{m_{\mathbb{Q}}\in \mathbf{M}_\gamma \backslash \mathbf{M} (\mathbb{Q})}
\int_{m_{\mathbb{Q}}\mathcal{F}} \dif m \, f_0(m^{-1}.\gamma) \, ,
\end{align*}
where $\mathcal{F}\subset \mathbf{M}(\mathbb{A})$ is a fundamental domain for the left action of $\mathbf{M}(\mathbb{Q})$ on $\mathbf{M}(\mathbb{A})$. Consider the inner sum above for each fixed double coset $[\gamma]$. Notice that the function $m\mapsto f_0(m^{-1}.\gamma)$  is invariant under the action of $\mathbf{M}_{\gamma}(\mathbb{A})$ on the left. The set $\sqcup_{m_{\mathbb{Q}}\in \mathbf{M}_\gamma \backslash \mathbf{M} (\mathbb{Q})} m_{\mathbb{Q}}\mathcal{F}$ is a fundamental domain for the left action of $\mathbf{M}_\gamma(\mathbb{Q})$ on $\mathbf{M}(\mathbb{A})$. In particular, the inner sum is a single integral over $\lfaktor{\mathbf{M}_\gamma(\mathbb{Q})}{\mathbf{M}(\mathbb{A})}$.

Because the right action of $\mathbf{G}(\mathbb{Q})$ on $\mathbf{P}^1(\mathbb{Q})$ is transitive on the $\mathbf{G}$-coordinate and keeps the $\mathbf{V}$-coordinate invariant
for each double coset $[\gamma]$ one can choose a representative $[\gamma]=[(e,\mathrm{v})]$ with $\mathrm{v}\in\mathbf{V}(\mathbb{Q})$ unique up to the action of $\mathbf{T}(\mathbb{Q})$ on $\mathbf{V}(\mathbb{Q})$.
The inner summand above is then equal to
\begin{align*}
\int_{\mathbf{M}_\gamma(\mathbb{Q})\backslash \mathbf{M}(\mathbb{A})} \dif m\, f_0(m^{-1}.\gamma)
&=\meas_{\mathbf{M}_\gamma}\left(\left[\mathbf{M}_\gamma(\mathbb{A})\right]\right)
\int_{\mathbf{M}_\gamma\backslash \mathbf{M}\,(\mathbb{A})} \dif m\, f_0(m^{-1}.\gamma)
\\
&=\meas_{\mathbf{M}_\gamma}\left(\left[\mathbf{M}_\gamma(\mathbb{A})\right]\right)
\int_{\mathbf{M}_\gamma\backslash \mathbf{T}\times \mathbf{G}\,(\mathbb{A})} \dif (t,g)\, f_0(t^{-1}g,t^{-1}.\mathrm{v}) \, .
\end{align*}
If $\mathrm{v}=0$ then by Lemma \ref{lem:M-stab} the stabilizer $\mathbf{M}_\gamma$ is the torus $\mathbf{T}$ embedded diagonally in $\mathbf{T}\times\mathbf{G}$ and the integral above reduces to an integral over $\mathbf{G}(\mathbb{A})$. Otherwise, the stabilizer is trivial and the integral is over $\mathbf{T}(\mathbb{A}) \times \mathbf{G}(\mathbb{A})$. In both cases, the measure of $\left[\mathbf{M}_\gamma(\mathbb{A})\right]$ is $1$ in our normalization.
The claim follows by substituting the definition of $f_0$.
\end{proof}

\subsection{The Trivial Orbital Integral}
The following proposition shows that the contribution of the $[0]$ relative orbital integral eventually vanishes in any strict sequence of $K_\infty$-invariant homogeneous toral sets.
\begin{prop}\label{prop:trivial-contribution-vanishes}
Fix $f\colon\mathbf{P}^1(\mathbb{A})\to \mathbb{R}_{\geq 0}$ measurable and supported on $B_G\times B_V$ for some compact subsets $B_G\subset\mathbf{G}(\mathbb{A})$, $B_V\subset\mathbf{V}(\mathbb{A})$. If
\begin{equation*}
l^{-1}.\mathrm{x}\not \in -B_V+B_G.\mathrm{y}\subset \mathbf{V}(\mathbb{A})
\end{equation*}
then $\RO_f(0)=0$.
\end{prop}
\begin{proof}
Consider formula \eqref{eq:RO-short-0} for $\RO_f(0)$.
If $(l,\mathrm{x})$ satisfies the condition in the claim then $(g,g.\mathrm{y}-l^{-1}.\mathrm{x})\not\in \supp f$ for any $g\in\mathbf{G}(\mathbb{A})$.
\end{proof}

\subsection{Decomposition of Orbital Integrals}
The following definition is slightly non-standard.
\begin{defi}
A compactly supported $f\colon\mathbf{P}^1(\mathbb{A})\to \mathbb{R}_{\geq 0}$ is a \emph{standard test function} if $f=\prod_v f_v$ where $f_v\colon\mathbf{P}^1(\mathbb{Q}_v)\to \mathbb{R}_{\geq 0}$ is compactly supported for all $v$ and $f_v=\mathds{1}_{\mathbf{P}^1(\mathbb{Z}_v)}$ for almost all $v$. Denote also $f_f\colon \mathbf{P}^1(\mathbb{A}_f)\to\mathbb{R}_{\geq 0}\coloneqq \prod_{v<\infty} f_v$.
\end{defi}

\begin{lem}\label{lem:RO-infty-finite}
Let $f$ be a standard test function. The relative orbital integral of $\mathrm{v}\neq0$  decomposes as a product of an archimedean contribution and a finite one in the following way
\begin{align*}
\RO_f(\mathrm{v})&=\RO_f(\mathrm{v})_\infty \RO_f(\mathrm{v})_f \, ,\\
\RO_f(\mathrm{v})_\infty&\coloneqq
\int_{\mathbf{T}\times \mathbf{G}\, (\mathbb{R})} \dif(t,g)
f([t(l_\infty,0)]^{-1}(e,\mathrm{v})g(e,\mathrm{y}_\infty))\\
&=\int_{K_\infty \times \mathbf{G}(\mathbb{R})} \dif(k,g)
f(k(l_\infty,0)^{-1}(e,\mathrm{v})g(e,\mathrm{y}_\infty))\\
&=\int_{K_\infty \times \mathbf{G}(\mathbb{R})} \dif(k,g)
f(g, g.\mathrm{y}_\infty+kl_\infty^{-1}.\mathrm{v}) \, ,\\
\RO_f(\mathrm{v})_f&\coloneqq
\int_{\mathbf{T}\times \mathbf{G}\, (\mathbb{A}_f)} \dif(t,g)
f([t(l_f,\mathrm{x}_f)]^{-1}(e,\mathrm{v})g(e,\mathrm{y}_f))\\
&=\int_{\mathbf{T}\times \mathbf{G}\, (\mathbb{A}_f)} \dif(t,g)
f(g, g.\mathrm{y}_f+l_f^{-1}t.\mathrm{v}-l_f^{-1}.\mathrm{x}_f) \, .
\end{align*}
\end{lem}
\begin{proof}
The decomposition follows from the product structure of the Haar measures on $\mathbf{G}(\mathbb{A})$ and $\mathbf{T}(\mathbb{A})$. The explicit formulae are derived in the same manner as \eqref{eq:RO-short-v} and \eqref{eq:RO-short-0}.
\end{proof}

\subsection{The Archimedean Orbital Integral}
We compute a simple bound on the archimedean contribution.
\begin{prop}\label{prop:RO-archimedean}
Let $f=\prod_v f_v$ be a standard test function with archimedean factor $f_\infty$.
Fix left $K_\infty$-invariant compact subsets $B_{G,\infty}\subset \mathbf{G}(\mathbb{R})$ and $B_{V,\infty}\subset \mathbf{V}(\mathbb{R})$ and set $B_\infty=B_{G,\infty}\times B_{V,\infty}\subset \mathbf{P}^1(\mathbb{R})$.
Assume $f_\infty=\mathds{1}_{B_\infty}$. Then
\begin{equation*}
\RO_f(\mathrm{v})_\infty\leq
\meas_{\mathbf{G}(\mathbb{R})}\left((e,\mathrm{y}_\infty)\cdot B_\infty^{-1} B_\infty \cdot
(e,-\mathrm{y}_\infty)\right) \, .
\end{equation*}
Moreover $\RO_f(\mathrm{v})_\infty$ vanishes unless
\begin{equation*}
l_\infty^{-1}.\mathrm{v}\in B_{V,\infty}-B_{G,\infty}.\mathrm{y}_\infty \, .
\end{equation*}
\end{prop}
\begin{proof}
The vanishing condition follows from the formula in Lemma \ref{lem:RO-infty-finite} and by examining the support of $f_\infty$. Assume henceforth $\RO_f(\mathrm{v})_\infty\neq 0$. It follows that there is some $b\in K_\infty B_\infty=B_\infty$ and $g_0\in\mathbf{G}(\mathbb{R})$ such that $(l_\infty,0)^{-1}(e,\mathrm{v})g_0(e,\mathrm{y}_\infty)=b$.

Using the $K_\infty$-invariance we deduce
\begin{align*}
\RO_f(\mathrm{v})_\infty&=\int_{K_\infty \times \mathbf{G}(\mathbb{R})} \dif(k,g)
f(k(l_\infty,0)^{-1}(e,\mathrm{v})g(e,\mathrm{y}_\infty))\\
&= \int_{\mathbf{G}(\mathbb{R})} \dif g\,
f(b(e,-y_\infty)g_0^{-1}g(e,\mathrm{y}_\infty))
= \int_{\mathbf{G}(\mathbb{R})} \dif g\,
f(b(e,-y_\infty)g(e,\mathrm{y}_\infty))\\
&=\meas_{\mathbf{G}(\mathbb{R})}\left((e,\mathrm{y}_\infty)b^{-1}B_\infty(e,-\mathrm{y}_\infty)
\right)\leq
\meas_{\mathbf{G}(\mathbb{R})}\left((e,\mathrm{y}_\infty)\cdot B_\infty^{-1} B_\infty\cdot(e,-\mathrm{y}_\infty)
\right) \, .
\end{align*}
\end{proof}

\begin{defi}\label{defi:test-function-archimedean}
Denote by $\|\bullet\|_2$ the standard Euclidean norm on $\mathbf{V}(\mathbb{R}^2)=\mathbb{R}^2$. This is the unique inner-product norm that is $K_\infty$-invariant and such that the co-volume of $\mathbf{V}(\mathbb{Z})=\mathbb{Z}^2$ in $\mathbf{V}(\mathbb{R}^2)$ is $1$.

For any $R_G,R_V\geq0$ we use the following notations for the closed balls around the identity
\begin{align*}
B_{G,\infty}(R_G)&\coloneqq
K_\infty \left\{\exp\left(H\begin{pmatrix}
1/2 & 0 \\ 0 & -1/2
\end{pmatrix}\right)
\mid |H|< R_G \right\}
K_\infty \, ,\\
B_{V,\infty}(R_V)&\coloneqq \left\{\mathrm{w}_\infty\in\mathbf{V}(\mathbb{R}) \mid \|\mathrm{w}_\infty\|_2< R_v  \right\} \, .
\end{align*}
\end{defi}

\begin{cor}\label{cor:archimedean-ball-norm}
In the setting of Proposition \ref{prop:RO-archimedean} with $B_{G,\infty}=B_{G,\infty}(R_G)$ and $B_{V,\infty}=B_{V,\infty}(R_V)$ for some $R_G,R_V\geq 0$
if $\RO_f(\mathrm{v})\neq 0$ then
\begin{equation*}
\max\left\{\exp\left(-\frac{R_G}{2}\right)\|\mathrm{y}_\infty\|_2-R_V ,0\right\}<
\|l_\infty^{-1}.\mathrm{v}\|_2 < \exp\left(\frac{R_G}{2}\right)\|\mathrm{y}_\infty\|_2+R_V \, .
\end{equation*}
\end{cor}
\begin{proof}
This follows from the condition $l_\infty^{-1}.\mathrm{v}\in B_{V,\infty}-B_{G,\infty}.\mathrm{y}_\infty$ in Proposition \ref{prop:RO-archimedean} above. The set $B_{V,\infty}-B_{G,\infty}.\mathrm{y}_\infty$ is easy to write down explicitly with this choice of $B_{G,\infty}$ and $B_{V,\infty}$ -- it is an annulus centered at the origin with outer radius equal to $\exp\left(\frac{R_G}{2}\right)\|\mathrm{y}_\infty\|_2+R_V$ and inner radius equal to $\max\{\exp\left(-\frac{R_G}{2}\right)\|\mathrm{y}_\infty\|_2-R_V,0\}$.
\end{proof}

\subsection{The Non-Archimedean Contribution}
\subsubsection{Bowen Balls}
Our test function at a non-archimedean place will be the characteristic function of a homogeneous Bowen ball.
\begin{defi}\label{defi:bowen-ball}
Let $p$ be a rational prime and denote by $A_p<\mathbf{G}(\mathbb{Q}_p)$ the standard diagonal subgroup. Fix $a\in A_p$ a generator of $A_p/A_p^\circ\simeq \mathbb{Z}$ where $A_p^\circ<A_p$ is the maximal compact-open subgroup. Let $\tau>0$ be an integer.

Let $\mathbf{H}$ be either $\mathbf{G},\mathbf{V}$ or $\mathbf{P}^1$ and define the $\mathbf{H}(\mathbb{Z}_p)$ Bowen ball of level $2\tau$ to be
\begin{align*}
\mathbf{H}(\mathbb{Z}_p)^{(-\tau,\tau)}\coloneqq\bigcap_{k=-\tau}^\tau a^k \mathbf{H}(\mathbb{Z}_p) a^{-k} \, .
\end{align*}
Notice that the definition above does not depend on the specific choice of a generator $a$.
Moreover the subgroup $\mathbf{G}(\mathbb{Z}_p)^{(-\tau,\tau)}$ stabilizes the $\mathbb{Z}_p$-lattice   $\mathbf{V}(\mathbb{Z}_p)^{(-\tau,\tau)}$ under the standard action of $\mathbf{G}$ on $\mathbf{V}$.
\end{defi}

\begin{lem}
Let $v$ be a finite rational place.
Recall from Lemma \ref{lem:vol-Lambda_f^{(1)}} that $\Lambda_v^{(1)}(\mathrm{x})=\Ad_{(l_v,x_v)}\mathbf{P}^1(\mathbb{Z}_v)\cap \mathbf{T}(\mathbb{Q}_v)$.
If $\Ad_{(l_v,x_v)^{-1}}\mathbf{T}(\mathbb{Q}_v)=\mathbf{A}(\mathbb{Q}_v)$ then
a direct verification shows that $$\Lambda_v^{(1)}(\mathrm{x}) (l_v,x_v) \mathbf{P}^1(\mathbb{Z}_v)^{(-\tau,\tau)}=(l_v,x_v)\mathbf{P}^1(\mathbb{Z}_v)^{(-\tau,\tau)}$$
for all $n$.
\end{lem}
\qed
\begin{lem}\label{lem:Bowen-jmath}
Let $v<\infty$ and assume $\Ad_{(l_v,x_v)^{-1}}\mathbf{T}(\mathbb{Q}_v)=\mathbf{A}(\mathbb{Q}_v)$, in particular, $v$ splits in $E$. If $p$ is the rational prime associated to $v$ write $p=\pi \tensor[^\sigma]{\pi}{}\in E_v=\prod_{w\mid v} E_w$ where $\pi,\tensor[^\sigma]{\pi}{}$ are uniformizers of $E_w$ for $w\mid v$. For any $n\in\mathbb{Z}_{\geq 0}$
\begin{align*}
\jmath_v\left(l_v.\mathbf{V}(\mathbb{Z}_v)^{(-\tau,\tau)}\right)&=\bigcap_{k=-\tau}^\tau \frac{\pi^k}{\tensor[^\sigma]{\pi}{^k}}\Lambda_v\subset \Lambda_v \, ,\\
\jmath_v\left(\Ad_{l_v}\mathbf{G}(\mathbb{Z}_v)^{(-\tau,\tau)}\right)&=\bigcap_{k=-\tau}^\tau \End\left(\frac{\pi^k}{\tensor[^\sigma]{\pi}{^k}}\Lambda_v\right)\subset \End(\Lambda_v) \, .
\end{align*}
\end{lem}
\begin{proof}
Follows directly from Definition \ref{defi:jmath-iso}.
\end{proof}

\begin{prop}\label{prop:geom-non-archimedean-sum}
Assume  $\Ad_{(l_{p_1},x_{p_1})^{-1}}\mathbf{T}(\mathbb{Q}_{p_1})=\mathbf{A}(\mathbb{Q}_{p_1})$.
Let $f=f_\infty \cdot \prod_{v<\infty} \mathds{1}_{B_v}$ where $B_v=\mathbf{P}^1(\mathbb{Z}_v)$ for all $v\not\in \left\{\infty,p_1\right\}$ and $B_{p_1}=\mathbf{P}^1(\mathbb{Z}_{p_1})^{(-\tau,\tau)}$ for a fixed $\tau>0$.
Denote $B_f=\prod_{v<\infty} B_v$ and write $B_f=B_{G_f}\ltimes B_{V,f}$. The compact groups $B_{G,f}<\mathbf{G}(\mathbb{A}_f)$, $B_{V,f}<\mathbf{V}(\mathbb{A}_f)$ are a product of the standard maximal subgroups at all places except $p_1$ and a Bowen ball of level $2\tau$ at $p_1$.

  Define the following $\Lambda_f^{(1)}(\mathrm{x})$-invariant compact open subset of $\mathbf{V}(\mathbb{A}_f)$
\begin{equation*}
B_{V,f}(\mathrm{x},\mathrm{y})\coloneqq \mathrm{x}_f
+l_f.\left(B_{V,f}-B_{G,f}.\mathrm{y}_f\right) \, .
\end{equation*}
Then the cross-correlation is equal to the following counting formula
\begin{align*}
\Cor(\mu,\nu)[f]&=\RO_f(0)+
\vol\left(\left[\mathbf{T}(\mathbb{A})(l,\mathrm{x})\right]\right)^{-1}
\meas_{\mathbf{G}(\mathbb{A}_f)} \left(\Ad_{(e,\mathrm{y}_f)}  B_f \right)\\
&\cdot\sum_{[0]\neq [\mathrm{v}]\in\mathbf{T}(\mathbb{Q}) \backslash \mathbf{V}(\mathbb{Q})}
\RO_f(\mathrm{v})_\infty\cdot
\#\left\{[t]\in \mathbf{T}(\mathbb{A}_f)\slash \Lambda_f^{(1)}(\mathrm{x})
\mid t^{-1}.\mathrm{v}\in B_{V,f}(\mathrm{x},\mathrm{y})\right\} \, .
\end{align*}
\end{prop}
\begin{proof}
The function $t\mapsto f([t(l_f,\mathrm{x}_f)]^{-1}(e,\mathrm{v})g(e,\mathrm{y}_f))$ is $\Lambda_f^{(1)}(\mathrm{x})$-invariant for all $g\in\mathbf{G}(\mathbb{A}_f)$ hence by Fubini
\begin{align*}
\RO_f(\mathrm{v})_f&=\meas_{\mathbf{T}(\mathbb{A}_f)} (\Lambda_f^{(1)}(\mathrm{x}))
\sum_{[t]\in \mathbf{T}(\mathbb{A}_f)\slash \Lambda_f^{(1)}(\mathrm{x})}
\RO(\mathrm{v},t)
\, ,\\
\RO(\mathrm{v},t)&\coloneqq\int_{\mathbf{G}(\mathbb{A}_f)} f([t(l_f,\mathrm{x}_f)]^{-1}(e,\mathrm{v})g(e,\mathrm{y}_f)) \, ,
\end{align*}
where in the sum on the right we can pick an arbitrary representative of each class $[t]$.

Consider a single summand $\RO(\mathrm{v},t)$ and denote $L=t(l_f,\mathrm{x}_f)$, $R=(e,\mathrm{y}_f)$. The integrand is the characteristic function of $L B_f R^{-1}$.
For the integral not to vanish there must exist some $g_0\in L B_f R^{-1}$. This condition is equivalent to $t^{-1}.\mathrm{v}\in B_{V,f}(\mathrm{x},\mathrm{y})$.
Assume this condition holds then because $B_f$ is a group the change of variable $g\mapsto g_0^{-1} g$ implies
\begin{equation*}
\RO(\mathrm{v},t)=\meas_{\mathbf{G}(\mathbb{A}_f)}(R  B_{f}  R^{-1}) \, .
\end{equation*}
\end{proof}

\section{Orbit Space for the Action of a Torus on the Unipotent Radical}
\label{sec:orbit-space}
In this section we continue to consider a fixed $K_\infty$-invariant $\mathbf{T}$-homogeneous set $\left[\mathbf{T}(\mathbb{A})(l,\mathrm{x})\right]\subset\left[\mathbf{P}^1(\mathbb{A})\right]$ for $\mathbf{T}<\mathbf{G}$ defined and anisotropic over $\mathbb{Q}$. Our aim is to describe the the orbit space $\lfaktor{\Lambda_f^{(1)}(\mathrm{x})}{\mathbf{V}(\mathbb{A}_f)}$ using fractional $\Lambda$-ideals with level structure associated to $\mathrm{x}$.

Recall from section \S\ref{sec:torsion-pts} that we have constructed a linear isomorphism $\jmath_{\mathbb{A}}\colon \mathbf{V}(\mathbb{A})\to \mathbb{A}_E$ that maps $\mathbf{V}(\mathbb{Q})$ to $E$ and $l_v.\mathbb{Z}_v$ to $\Lambda_v$ for all $v<\infty$. Moreover, this isomorphism is equivariant with respect to the action of $\mathbf{T}(\mathbb{A}_E)=\mathbb{A}_E^{(1)}$ on both sides.

\begin{defi}\label{defi:ideals+level}
\hfill
\begin{enumerate}
\item Define
\begin{equation*}
\Ideals(\Lambda,\mathrm{x})\coloneqq\faktor{\mathbb{A}_{E,f}^\times}{\Lambda^\times_f(\mathrm{x})} \, .
\end{equation*}
Notice that $\Ideals(\Lambda,\mathrm{x})$ depend only on the class of $\jmath_{/\Lambda}(\mathrm{x})$ modulo $\Lambda$.

\item
For $\mathrm{x}=0$ we denote $\Ideals(\Lambda)\coloneqq \Ideals(\Lambda,0)$ which we identify with the set of invertible fractional $\Lambda$-ideals through the map
\begin{equation*}
\Ideals(\Lambda) \ni \left(\alpha_v \Lambda_v^\times\right)_{v<\infty}\mapsto \bigcap_{v<\infty} \alpha_v \Lambda_v \subset E \, .
\end{equation*}

\item For any rational place $v<\infty$ we define $\|\bullet\|_v\colon E_v\to \mathbb{Z}$ by writing $E_v\simeq\prod_{w\mid v} E_w$ and setting
\begin{equation*}
\|(\alpha_w)_{w\mid v}\|_v =\prod_{w\mid v} |\alpha_w|_w \, ,
\end{equation*}
where we normalize $|\bullet|_w$ so that the absolute value of a uniformizer in $\mathcal{O}_{E_w}$ is reciprocal to the size of the residue field of $E_w$.

\item
For each $\mathrm{x}\in \mathbb{A}_E$ there is a finite to one quotient map $\Ideals(\Lambda,\mathrm{x})\to\Ideals(\Lambda)$. It is natural to consider $\Ideals(\Lambda,\mathrm{x})$ as invertible fractional $\Lambda$-ideals with level structure associated to $\jmath_{/\Lambda}(\mathrm{x}) \in E/\Lambda$. We define for each $\mathfrak{a}\in \Ideals(\Lambda,\mathrm{x})$ the norm $\Nr(\mathfrak{a})$ to be the norm of the fractional $\Lambda$-ideal associated to $\mathfrak{a}$ via the map  $\Ideals(\Lambda,\mathrm{x})\to\Ideals(\Lambda)$. This norm coincides with the ad\`elic character
\begin{equation*}
\Nr\left(\left(\alpha_v \Lambda_v^\times(\mathrm{x}) \right)_{v<\infty}\right)=\prod_{v<\infty} \|\alpha_v\|_v^{-1} \, .
\end{equation*}

\item We use the notation $\mathfrak{a}\mapsto [\mathfrak{a}]$ for the quotient map
\begin{equation*}
\Ideals(\Lambda,\mathrm{x})=\faktor{\mathbb{A}_{E,f}^\times}{\Lambda^\times_f(\mathrm{x})}
\to \dfaktor{E^\times}{\mathbb{A}_{E,f}^\times}{\Lambda^\times_f(\mathrm{x})}=\Pic(\Lambda,\mathrm{x}) \, .
\end{equation*}
\end{enumerate}

\end{defi}

\begin{defi}\label{def:Vaccessible} \hfill
\begin{enumerate}
\item Define
\begin{align*}
\mathbf{V}(\mathbb{A}_f)_\mathrm{accessible}&\coloneqq \mathbf{T}(\mathbb{A}_f).\mathbf{V}(\mathbb{Q})\subset \mathbf{V}(\mathbb{A}_f) \, ,\\
\mathbf{V}(\mathbb{A}_f)_\mathrm{accessible}^\times&\coloneqq \mathbf{V}(\mathbb{A}_f)_\mathrm{accessible}\setminus \{0\} \, .
\end{align*}
Notice that
\begin{align*}
\jmath_{\mathbb{A}}\left(\mathbf{V}(\mathbb{A}_f)_\mathrm{accessible}\right)&=E\cdot \mathbb{A}_{E,f}^{(1)} \, ,\\
\jmath_{\mathbb{A}}\left(\mathbf{V}(\mathbb{A}_f)_\mathrm{accessible}^\times\right)&=E^\times\cdot \mathbb{A}_{E,f}^{(1)} \, .
\end{align*}
In particular, if an element of $\mathbf{V}(\mathbb{A}_f)_\mathrm{accessible}$ is zero at some place $v$ then it vanishes globally.
\item
Define the map $\inv\colon \mathbf{V}(\mathbb{A}_f)_\mathrm{accessible}^\times\to \Ideals(\Lambda,\mathrm{x})$
\begin{equation*}
\inv\left(\mathrm{w}_f\right)\coloneqq \jmath_{\mathbb{A}}(\mathrm{w}_f)\bmod \Lambda^\times_f(\mathrm{x}) \, .
\end{equation*}
\end{enumerate}
\end{defi}

Recall that the subgroup $\Pic^\mathrm{pg}(\Lambda,\mathrm{x})<\Pic(\Lambda,\mathrm{x})$ was defined in \ref{defi:Lambda_v(x)+Pic} as the image of $\mathbb{A}_{E,f}^{(1)}$ in $\Pic(\Lambda,\mathrm{x})$.
\begin{lem}\label{lem:invariant-basics}
All $\mathrm{w}_f=(\mathrm{w}_v)_{v<\infty}\in\mathbf{V}(\mathbb{A}_f)_\mathrm{accessible}^\times$ satisfy
$[\inv(\mathrm{w}_f)]\in \Pic^\mathrm{pg}(\Lambda,\mathrm{x})$. Moreover,
for any place $v<\infty$
\begin{equation*}
\Nr(\inv(\mathrm{w}_f))=\Nr\jmath_v(\mathrm{w}_v) \, ,
\end{equation*}
where the norm map on the right hand side is the local norm $E_v\to\mathbb{Q}_v$. In particular, the claim implies $\Nr\jmath_v(\mathrm{w}_v)\in \mathbb{Q}$.
\end{lem}
\begin{proof}
The first claim follows immediately from Definition \ref{def:Vaccessible}. For the second part write $\jmath_{\mathbb{A}}(\mathrm{w}_f)=c(\alpha_v)_v$ with $c\in E^\times$ and $\alpha_v\in E_v^{(1)}$ for all $v<\infty$. The product formula for $E$ and the fact $\Nr\alpha_v=1$ imply
\begin{equation*}
\Nr(\inv(\mathrm{w}_f))=\prod_{v<\infty} \|c\alpha_v\|_v^{-1}=\prod_{v<\infty} \|c\|_v^{-1}=\Nr c
=\Nr(c\alpha_v)=\Nr(\jmath_v(\mathrm{w}_v))
\end{equation*}
as claimed.
\end{proof}

\begin{prop}\label{prop:invariant-separates}
The invariant map $\inv\colon\mathbf{V}(\mathbb{A}_f)_\mathrm{accessible}^\times\to\Ideals(\Lambda,\mathrm{x})$ separates $\Lambda_f^{(1)}(\mathrm{x})$-orbits.
\end{prop}
\begin{proof}
Applying the isomorphism $\jmath_{\mathbb{A}}$ the claim is equivalent to the following assertion. For any $c,c'\in E^\times$,  $(\alpha_v)_{v<\infty},(\alpha'_v)_{v<\infty}\in \mathbb{A}_{E,f}^{(1)}$ if $c_f=c(\alpha_v)_{v<\infty}$, $c'_f=c'(\alpha'_v)_{v<\infty}$ satisfy $c'_f\in c_f \Lambda^\times_f(\mathrm{x})$ then  $c'_f\in c_f \Lambda^{(1)}_f(\mathrm{x})$.

Assume now $c'_f\in c_f \Lambda^\times_f(\mathrm{x})$ and write $c'\alpha'_v =c\alpha_v k_v$ for some $k_v\in\Lambda_v^\times(\mathrm{x})$ for all $v<\infty$.
Because the reciprocal of the norm of an element in $E^\times$ coincides with the restriction of the ad\`elic character on $\mathbb{A}_{E,f}^\times$ to $E^\times$ we have $\Nr(c)=\Nr(c')$. Hence
\begin{equation*}
k_v=\frac{c'}{c}\frac{\alpha'_v}{\alpha_v}
\end{equation*}
is a norm $1$ element for all $v<\infty$ as claimed.
\end{proof}
\begin{remark}
The image in $\Ideals(\Lambda,\mathrm{x})$ of the invariant map is exactly all elements $\mathfrak{a}\in \Ideals(\Lambda,\mathrm{x})$ such that $[\mathfrak{a}]\in \Pic^\mathrm{pg}(\Lambda,\mathrm{x})$.
\end{remark}

The following proposition shows how to compute the norm of am element in $\Ideals(\Lambda,\mathrm{x})$ using the archimedean place.
\begin{prop}\label{prop:psi-archimedean}
Let $\mathrm{w}_f=t_f\mathrm{v}\in\mathbf{V}(\mathbb{A}_f)_\mathrm{accessible}^\times$ where $t_f\in\mathbf{T}(\mathbb{A}_f)$ and $\mathrm{v}\in\mathbf{V}(\mathbb{Q})$ then
\begin{equation*}
\Nr(\inv(\mathrm{w}_f))=\sqrt{|D|}\left\| l_\infty^{-1}.\mathrm{v}\right\|_2^2 \, ,
\end{equation*}
where the norm on the right hand side is the usual Euclidean norm on $\mathbf{V}(\mathbb{R})=\mathbb{R}^2$.
\end{prop}
\begin{proof}
Write $\jmath_{\mathbb{A}_f}(\mathrm{w})=c(\alpha_v)_v$ where $c=\jmath(\mathrm{v})\in E$ and $\alpha_v=\jmath_v(t_v)\in E_v^{(1)}$ for all $v<\infty$. The proof of Lemma \ref{lem:invariant-basics} implies that $\Nr(\inv(\mathrm{w}_f))=\Nr c=\Nr \jmath_\infty (\mathrm{v})$.
Consider the composite  map
\begin{equation*}
S\colon\mathbb{R}^2=\mathbf{V}(\mathbb{R})\xrightarrow{l_\infty}\mathbf{V}(\mathbb{R})\xrightarrow{\jmath_\infty} \, .
\mathbb{C}
\end{equation*}
This is a linear isomorphism over $\mathbb{R}$ that intertwines the $K_\infty$-action on the left-hand side with multiplication by $\jmath_{\infty}(E_\infty^{(1)})=\mathbb{C}^{(1)}$ on the right-hand side. In particular, $S$ is a Euclidean similitude and we deduce that
\begin{equation*}
\Nr(\inv(\mathrm{w}_f))=\Nr(\jmath_{\infty}(\mathrm{v}))=\Nr S(l_\infty^{-1}\mathrm{v})=S\tensor[^*]{S}{} \left\| l_\infty^{-1}.\mathrm{v}\right\|_2^2
\, ,
\end{equation*}
where $S\tensor[^*]{S}{}$ is a scalar because $S$ is a similitude. This scalar is non-negative because both the field norm of an imaginary quadratic field and the Euclidean norm on $\mathbb{R}^2$ are positive functions.

To show that $S\tensor[^*]{S}{}=\sqrt{|D|}$ we use the equality
\begin{equation*}
(S\tensor[^*]{S}{})^2=\det S\tensor[^*]{S}{}=|\det S|^2\Rightarrow S\tensor[^*]{S}{}=|\det S| \, .
\end{equation*}
The map $\jmath_{\infty}$ maps the lattice $\cap_{v<\infty} l_v.\mathbb{Z}_v$ to $\Lambda$. Because $l_v\in\mathbf{SL}_2(\mathbb{Q}_v)$ for all $v<\infty$ we can check locally to see that the lattice $\cap_{v<\infty} l_v.\mathbb{Z}_v\subset \mathbf{V}(\mathbb{Q})\subset\mathbf{V}(\mathbb{R})$ is unimodular. As $l_\infty\in\mathbf{SL}_2(\mathbb{R})$ the map $S$ sends the unimodular lattice $l_\infty^{-1}.\left(\cap_{v<\infty} l_v.\mathbb{Z}_v\right)$ to the lattice $\Lambda$ of covolume $\sqrt{|D|}$. Hence $|\det S|=\sqrt{|D|}$ as required.
\end{proof}

\section{The Subconvex Bound}\label{sec:subconvex}
In this section we tie the different threads of the proof together. Using the previous two sections we rewrite the cross-correlation as a sum over $\Lambda$-ideals with level structure that are integral in an appropriate sense. This sum is then controlled using the subconvex bound of Duke, Friedlander and Iwaniec \cite{DFI}.

As in \S\ref{sec:rt-assumptions}
we fix a $K_\infty$-invariant $\mathbf{T}$-homogeneous set $\mathcal{H}=\left[\mathbf{T}(\mathbb{A})(l,\mathrm{x})\right]$ such that $\mathbf{T}<\mathbf{G}$ is a maximal torus.
We denote by $\mu$ the periodic measure on $\mathcal{H}$. In addition, we fix a $\mathbf{G}$-homogeneous set $[\mathbf{G}(\mathbb{A})(e,\mathrm{y})]$ and denote by $\nu$ the periodic measure supported on it.

We are now in position to rewrite the geometric expansion as presented in Proposition \ref{prop:geom-non-archimedean-sum} using the orbit space developed in \S\ref{sec:orbit-space}.
\begin{defi}\label{defi:Bowen-ball-test-function}
Let $\tau\in\mathbb{Z}_{\geq 0}$, $R_V,R_G>0  $.
A Bowen ball test function $f=f_{\tau, R_G, R_V}\colon\mathbf{P}^1(\mathbb{A})\to\mathbb{R}_{>0}$ is a function of the form $f=\prod_v f_v$ where
\begin{align*}
\forall v\neq \infty,p_1\colon f_v&=\mathds{1}_{\mathbf{P}^1(\mathbb{Z}_v)} \, ,\\
f_{p_1}&= \mathds{1}_{\mathbf{P}^1(\mathbb{Z}_{p_1})^{(-\tau,\tau)}} \, ,\\
f_\infty&= \mathds{1}_{B_{G,\infty}(R_G)}\cdot \mathds{1}_{B_{V,\infty}(R_V)} \, .
\end{align*}
Where $\mathbf{P}^1(\mathbb{Z}_{p_1})^{(-\tau,\tau)}$ is the Bowen ball as defined in Definition \ref{defi:bowen-ball} and $B_{G,\infty}(R_G)$, $B_{V,\infty}(R_V)$ are the open identity balls from Definition \ref{defi:test-function-archimedean}.
\end{defi}

\begin{prop}\label{prop:geom-expansion-ideals}
Let $f=f_{\tau, R_G, R_V}=\mathds{1}_B$ be a Bowen ball test function.
If $\tau>0$ then assume in addition $\mathrm{y}_{p_1}=0$ and
\begin{equation*}
(l_{p_1},\mathrm{x}_{p_1})^{-1} \mathbf{T}(\mathbb{Q}_{p_1})(l_{p_1},\mathrm{x}_{p_1})=\mathbf{A}(\mathbb{Q}_{p_1}) \, .
\end{equation*}
In particular, $\mathrm{x}_{p_1}=0$ if $\tau>0$.

Set
\begin{align*}
X_{\min}&=\max\left\{\exp\left(-\frac{R_G}{2}\right)\|\mathrm{y}_\infty\|_2-R_V ,0\right\} \, ,\\
X_{\max}&=\exp\left(\frac{R_G}{2}\right)\|\mathrm{y}_\infty\|_2+R_V \, .
\end{align*}
Then the cross-correlation is bounded above by
\begin{align*}
\Cor(\mu,\nu)[f]&\leq\RO_f(0)+
\vol\left(\left[\mathbf{T}(\mathbb{A})(l,\mathrm{x})\right]\right)^{-1}
\meas_{\mathbf{G}(\mathbb{A})} \left(\Ad_{(e,\mathrm{y})} (B^{-1}B) \right)\\
&
\sum_{N=\sqrt{|D|}X_{\min}^2}
^{\sqrt{|D|}X_{\max}^2}
\left|\left\{\Ideals(\Lambda,\mathrm{x})\ni \mathfrak{a}\subset \Lambda_f(\mathrm{x},\mathrm{y}) \mid \substack{\Nr\mathfrak{a}=N\\
[\mathfrak{a}]\in \Pic^\mathrm{pg}(\Lambda,\mathrm{x})}\right\}\right| \, ,
\end{align*}
where
\begin{align*}
\Lambda_f(\mathrm{x},\mathrm{y})&\coloneqq\prod_{v<\infty}\Lambda_v(\mathrm{x},\mathrm{y}) \, ,\\
\forall v\neq\infty,p_1 \colon \Lambda_v(\mathrm{x},\mathrm{y})&= \jmath_v (\mathrm{x}_v)+\Lambda_v-\Aut^1(\Lambda_v).\jmath_v(l_v.\mathrm{y}_v) \, ,\\
\Lambda_{p_1}(\mathrm{x},\mathrm{y})&= \begin{cases}
\jmath_{p_1} (\mathrm{x}_{p_1})+\Lambda_{p_1}-\Aut^1(\Lambda_{p_1}).\jmath_{p_1}(l_{p_1}.\mathrm{y}_{p_1}) & \tau=0\\
\bigcap_{k=-\tau}^\tau \frac{\pi^k}{\tensor[^\sigma]{\pi}{^k}}\Lambda_{p_1} & \tau>0
\end{cases} \, ,
\end{align*}
where $\pi$ is a uniformizer of $E_\mathfrak{p}$ for $p_1=\mathfrak{p}\tensor[^\sigma]{\mathfrak{p}}{}$.
\end{prop}
\begin{proof}
Recall from Proposition \ref{prop:geom-non-archimedean-sum} the definition $B_{V,f}(\mathrm{x},\mathrm{y})=\mathrm{x}_f+l_f.(B_{V,f}-B_{G,f}.\mathrm{y}_f)\subset\mathbf{V}(\mathbb{A}_f)$ where $B_{G,f}\ltimes B_{V_f}$ is the support of $\prod_{v<\infty} f_v$. By definition $B_{V,f}=\prod_{v\neq\infty,p_1} \mathbf{V}(\mathbb{Z}_v)\times \mathbf{V}(\mathbb{Z}_{p_1})^{(-\tau,\tau)}$. Using Lemma \ref{lem:Bowen-jmath}
we deduce that $\jmath_{\mathbb{A}}(B_{V,f}(\mathrm{x},\mathrm{y}))=\Lambda_f(\mathrm{x},\mathrm{y})$.

The claim then follows by rewriting the summation in Proposition \ref{prop:geom-non-archimedean-sum} as a sum over $\Lambda_f^{(1)}(\mathrm{x})\backslash\mathbf{V}(\mathbb{A}_f)_\mathrm{accessible}^\times$ and translating it to a sum over $\Ideals(\Lambda,\mathrm{x})$ using the invariant map of Definition \ref{def:Vaccessible} in conjunction with Lemma \ref{lem:invariant-basics}, Proposition \ref{prop:invariant-separates},
Proposition \ref{prop:RO-archimedean}, Corollary \ref{cor:archimedean-ball-norm} and Proposition \ref{prop:psi-archimedean}.
\end{proof}

\begin{defi}\label{defi:L-functions}
For any unitary character $\chi\colon \Pic(\Lambda,\mathrm{x})\to\mathbb{C}^\times$ define the meromorphic function
\begin{equation*}
L_{\Lambda(\mathrm{x},\mathrm{y})}(s,\chi)=\sum_{\Ideals(\Lambda,\mathrm{x})\ni\mathfrak{a}\subset \Lambda_f(\mathrm{x},\mathrm{y})}
\frac{\chi(\mathfrak{a})}{\Nr(\mathfrak{a})^s} \, .
\end{equation*}
Because $\chi$ is multiplicative and $\Ideals(\Lambda,\mathrm{x})$ and $\Lambda_f(\mathrm{x},\mathrm{y})$ split into a product of local factors the function $L_{\Lambda(\mathrm{x},\mathrm{y})}(s,\chi)$ has a formal Euler product
\begin{equation*}
L_{\Lambda(\mathrm{x},\mathrm{y})}(s,\chi)=\prod_{v<\infty} \sum_{E_v^\times/\Lambda_v^\times(\mathrm{x})\ni\alpha_v\subset \Lambda_v(\mathrm{x},\mathrm{y})} \chi(\alpha_v)\|\alpha_v\|_v^s \, .
\end{equation*}
For almost all places $v<\infty$ the Euler factor coincides with the Euler factor at $v$ of the Hecke $L$-function $L(s,\chi)$ with Grossencharakter $\chi$. This happens in particular when $v\neq p_1$, $\Lambda_v=\mathcal{O}_{E_v}$ and $\mathrm{x}_v,\mathrm{y}_v\in\mathcal{O}_{E_v}$. The other Euler factors are non-vanishing holomorphic functions for $\Re s>0$ as seen in Appendix \ref{app:L-functions}. Hence $L_{\Lambda(\mathrm{x},\mathrm{y})}(s,\chi)$ is meromorphic for $\Re s>0$. It is holomorphic if $\chi\neq 1$ and has a single simple pole at $s=1$ otherwise.
\end{defi}
\begin{remark}
Proposition \ref{prop:L-local-denominators} implies that
the $L$-function $L_{\Lambda(\mathrm{x,y})}$ is an $L$-series of the form
\begin{equation*}
\sum_{n\in\frac{1}{(\ord_\mathcal{H}(\mathrm{x})\ord_\mathcal{H}(l.\mathrm{y}))^2}\mathbb{Z}} \frac{a_n}{n^s} \, .
\end{equation*}
The non integral denominators arise due to the fact that whenever either $\mathrm{x}_v\not\in\mathcal{O}_{E_v}$ or $\mathrm{y}_v\not\in\mathcal{O}_{E_v}$ then  $\Lambda_v(\mathrm{x},\mathrm{y})\not\subset \mathcal{O}_{E_v}$. One could easily convert this to a standard $L$-series with an integral summation range by multiplying by $(\ord_\mathcal{H}(\mathrm{x})\ord_\mathcal{H}(l.\mathrm{y}))^{2s}$. Yet this transformation is unnecessary as all the argument we employ using the Perron formula are evidently valid also for $L$-series with non-integral summands.
\end{remark}

\begin{defi}
Fix a  smooth function $\varphi(x,\alpha)\colon \mathbb{R}\times[0,1) \to [0,1]$ such that for all $\alpha\in[0,1)$
\begin{enumerate}
\item $\varphi(x,\alpha)$ is a compactly supported function of $x$,
\item $\varphi(x,\alpha)\geq \mathds{1}_{[\alpha,1]}(x)$,
\item $\int_\mathbb{R} \varphi(x,\alpha) \dif x \ll 1-\alpha$.
\end{enumerate}
For any $0\leq \alpha< 1$ we denote the Mellin transform of $\varphi_\alpha(x,\alpha)$ in the $x$ variable by $\mathcal{M}\varphi(s,\alpha)$.
\end{defi}
\begin{remark}
An explicit construction
is
\begin{equation*}
\varphi(x,\alpha)=\eta\left(\frac{x-\alpha}{1-\alpha}\right) \, ,
\end{equation*}
where $\eta\colon \mathbb{R}\to[0,1]$ is any smooth compactly supported function satisfying $\eta\geq \mathds{1}_{[0,1]}$.
\end{remark}

\begin{remark}
For any $0\leq\alpha<1$ because $\varphi(\cdot,\alpha)$ is smooth and compactly supported the Mellin transform $\mathcal{M}\varphi(s,\alpha)$ decays faster then any polynomial in the vertical direction uniformly in any vertical strip $a\leq \Re s \leq b$. Moreover as $\varphi(x,\alpha)$ is a smooth function of $2$-variables the decay rate depends continuously on $\alpha$.
\end{remark}

\begin{prop}\label{prop:geom-expansion-perron}
In the setting of Proposition \ref{prop:geom-expansion-ideals}
the following inequality hold for any $c>1$
\begin{align*}
&\Cor(\mu,\nu)[f]\leq\RO_f(0)+
\frac{\vol\left(\left[\mathbf{T}(\mathbb{A})(l,\mathrm{x})\right]\right)^{-1}
\meas_{\mathbf{G}(\mathbb{A})} \left(\Ad_{(e,\mathrm{y})} (B^{-1}B) \right)}
{\left[\Pic(\Lambda,\mathrm{x}):\Pic^\mathrm{pg}(\Lambda,\mathrm{x})\right]}\\
&\cdot\sum_{\chi\in \Pic^\mathrm{pg}(\Lambda,\mathrm{x})^\perp}
\frac{1}{2\pi i}\int_{c-i\infty}^{c+i\infty} L_{\Lambda(\mathrm{x},\mathrm{y})}(s,\chi)
\mathcal{M}\varphi\left(s,\left(\frac{X_{\min}}{X_{\max}}\right)^2\right)
\left(\sqrt{|D|}X_{\max}^2\right)^s
 \dif s \, .
\end{align*}
\end{prop}
\begin{proof}
Denote $\alpha=\left(\frac{X_{\min}}{X_{\max}}\right)^2$.
We apply first an elementary transformation
\begin{align}
\nonumber
\sum_{\substack{\Ideals(\Lambda,\mathrm{x})\ni \mathfrak{a}\subset \Lambda_f(\mathrm{x},\mathrm{y})
\\ \sqrt{|D|}X_{\min}^2\leq \Nr\mathfrak{a}\leq \sqrt{|D|}X_{\max}^2
\\ [\mathfrak{a}]\in\Pic^\mathrm{pg}(\Lambda,\mathrm{x})}}
1
&\leq\sum_{\substack{\Ideals(\Lambda,\mathrm{x})\ni \mathfrak{a}\subset \Lambda_f(\mathrm{x},\mathrm{y})
\\ [\mathfrak{a}]\in\Pic^\mathrm{pg}(\Lambda,\mathrm{x})}}
\varphi\left(\frac{\Nr\mathfrak{a}}{ \sqrt{|D|}X_{\max}^2},\alpha\right)
\\
=
\left[\Pic(\Lambda,\mathrm{x}):\Pic^\mathrm{pg}(\Lambda,\mathrm{x})\right]^{-1}
&\sum_{\chi\in \Pic^\mathrm{pg}(\Lambda,\mathrm{x})^\perp}
\sum_{\Ideals(\Lambda,\mathrm{x})\ni \mathfrak{a}\subset \Lambda_f(\mathrm{x},\mathrm{y})}
\chi(\mathfrak{a})\varphi\left(\frac{\Nr\mathfrak{a}}{\sqrt{|D|}X_{\max}^2},\alpha\right) \, .
\label{eq:pre-Perron}
\end{align}
The following is a smoothed version of Perron's formula which holds in our case because $L_{\Lambda(\mathrm{x},\mathrm{y})}(s,\chi)$ is meromorphic with at most a single simple pole at $s=1$.
\begin{equation*}
\forall Y>0\colon \
\sum_{\Ideals(\Lambda,\mathrm{x})\ni \mathfrak{a}\subset \Lambda_f(\mathrm{x},\mathrm{y})}
\chi(\mathfrak{a}) \varphi\left(\frac{\Nr(\mathfrak{a})}{Y},\alpha\right)
=\frac{1}{2\pi i}
\int_{c-i\infty}^{c+i\infty} L_{\Lambda(\mathrm{x},\mathrm{y})}(s,\chi)
\mathcal{M}\varphi(s,\alpha) Y^s\dif s \, .
\end{equation*}
The claim follows by applying Perron's formula to each summand in \eqref{eq:pre-Perron}.
\end{proof}

We are now prepared to establish the main bound from which all our results follow. The most important aspect is that the main term of the correlation, which we denote by $\mathrm{MT}$, decays as $p_1^{-2\tau}$ for large $\tau$. This is exactly the behavior we expect if we replace $\mu$ by the Haar measure $m$, because then the correlation is comparable to the volume of a $p_1$-adic tube of width $p_1^{-\tau}$ around the fixed periodic $\mathbf{G}$-orbit. Because $\mathbf{G}$ is of codimension $2$ in $\mathbf{P}^1$ the the Haar volume of such a tube should be proportional to $p_1^{-2\tau}$. Most of the work goes into showing that the main term dominates the error term which we denote by $\mathrm{ST}$. This is where the subconvex bound plays a crucial role. There is a technical difficulty because the $L$-functions that appears in the pertinent counting problem, $L_{\Lambda(\mathrm{x,y})}(s,\chi)$, differ from the canonical class group $L$-functions at finitely many Euler terms. Because the number of differing terms is not uniformly bounded we need also to show that this local contribution to the error term is negligible. This calculation is done in the appendix.

\begin{thm}\label{thm:Cor-main-inequality}
Let $f=f_{\tau, R_G, R_V}=\mathds{1}_B$ be a Bowen ball test function.
If $\tau>0$ then assume in addition $\mathrm{y}_{p_1}=0$ and
\begin{equation*}
(l_{p_1},\mathrm{x}_{p_1})^{-1} \mathbf{T}(\mathbb{Q}_{p_1})(l_{p_1},\mathrm{x}_{p_1})=A_{p_1} \, .
\end{equation*}
In particular, if $\tau>0$ then $\mathrm{x}_{p_1}=0$ and $\val_{p_1}(f)=1$.

Set
\begin{align*}
X_{\min}&=\max\left\{\exp\left(-\frac{R_G}{2}\right)\|\mathrm{y}_\infty\|_2-R_V ,0\right\} \, ,\\
X_{\max}&=\exp\left(\frac{R_G}{2}\right)\|\mathrm{y}_\infty\|_2+R_V \, .
\end{align*}
Then there is an explicit computable constant $\delta>0$ and a continuous function $\Phi\colon [0,1)\to \mathbb{R}$ such that
\begin{align*}
\Cor(\mu,\nu)[f]-\RO_f(0)&\ll
\meas_{\mathbf{G}(\mathbb{A})} \left(\Ad_{(e,\mathrm{y})} (B^{-1}B) \right)
\ord_{\mathcal{H}}(l.\mathrm{y})^2 (\mathrm{MT}+\mathrm{ST}) \, ,\\
\mathrm{MT}&\coloneqq \left(X_{\max}^2-X_{\min}^2 \right)p_1^{-2\tau} \, ,\\
\left|\mathrm{ST}\right|\ &\ll_\varepsilon \Phi\left(\left(\frac{X_{\min}}{X_{\max}}\right)^2\right)
 (f\ord_{\mathcal{H}}(l.\mathrm{y}))^\varepsilon
|D|^{-\delta+\varepsilon}  X_{\max}\,p_1^{-\tau}\\
&\cdot\left(\frac{\ord_{\mathcal{H}}(\mathrm{x})}{\gcd(\ord_{\mathcal{H}}(\mathrm{x}),f)}\right)^{-1/4-\delta+\varepsilon} \, .
\end{align*}
\end{thm}
\begin{remark}\label{rem:torsion-order-translation}
Because $\ord_{\mathcal{H}}(l.\mathrm{y})$ is the order of $l_f.\mathrm{y}_f$ in $\mathbf{V}(\mathbb{A}_f)/l_f.\mathbf{V}(\hat{\mathbb{Z}})$  it is also the torsion order of $\mathrm{y}_f$ in  $\mathbf{V}(\mathbb{A}_f)/\mathbf{V}(\hat{\mathbb{Z}})$. In particular, $\ord_{\mathcal{H}}(l.\mathrm{y})$ depends only on $\mathrm{y}$ an not on the homogeneous toral set $\mathcal{H}$ and its datum $l$.
\end{remark}
\begin{remark}
The dependence on $\varepsilon$ is ineffective due to the application of Siegel's lower bound for $L(1,\chi_E)$.
\end{remark}
\begin{remark}
The constant $\delta>0$ is the best known constant for subconvexity of $\mathbf{GL}_2$ $L$-functions in the level aspect. A positive value of $\delta$ has been originally established by \cite{DFI} and $\delta=1/4$ would follow from the Lindel\"of Hypothesis for these $L$-functions. The classical convexity bound provides $\delta=0$ and would suffice for our needs as long as there is some $\eta>0$ such that
\begin{equation*}
\left(\frac{\ord_{\mathcal{H}}(\mathrm{x})}{\gcd(\ord_{\mathcal{H}}(\mathrm{x}),f)}\right)\gg |D|^\eta \, .
\end{equation*}
This is the case, in particular, for the joint equidistribution problem studied by Aka, Einsiedler and Shapira \cite{AES} whenever $D$ is fundamental.
\end{remark}

For the proof we will need the following upper bound from principal genus theory.
\begin{lem}\label{lem:principal-genus}
\begin{equation*}
\left[\Pic(\Lambda,\mathrm{x}):\Pic^\mathrm{pg}(\Lambda,\mathrm{x})\right]\ll 4^{\omega(f\ord_{\mathcal{H}}(\mathrm{x}))} \, .
\end{equation*}
\end{lem}
\begin{proof}
We shall compute the index using Pontryagin duality
\begin{equation*}
\left[\Pic(\Lambda,\mathrm{x}):\Pic^\mathrm{pg}(\Lambda,\mathrm{x})\right]= \left|\Pic^\mathrm{pg}(\Lambda,\mathrm{x})^\perp\right| \, .
\end{equation*}
The following commutative diagram
\begin{center}
\begin{tikzcd}
\lfaktor{E^{(1)}}{\mathbb{A}_{E}^{(1)}} \arrow[r] \arrow[d] & \lfaktor{E^\times}{\mathbb{A}_{E}^\times} \arrow[d] \\
\Pic^\mathrm{pg}(\Lambda,\mathrm{x}) \arrow[r] &  \Pic(\Lambda,\mathrm{x})  \\
\end{tikzcd}
\end{center}
implies that any character in $\Pic^\mathrm{pg}(\Lambda,\mathrm{x})^\perp=\ker\left[\widehat{\Pic(\Lambda,\mathrm{x})}\to \widehat{\Pic^\mathrm{pg}(\Lambda,\mathrm{x})}\right]$ defines a character of $\lfaktor{E^\times}{\mathbb{A}_{E}^\times}$ vanishing on $\lfaktor{E^{(1)}}{\mathbb{A}_{E}^{(1)}}$. Hilbert's Satz 90 implies that any such character is real valued, hence $\Pic^\mathrm{pg}(\Lambda,\mathrm{x})^\perp$ is a $2$-torsion group.

Global class field theory and the Hasse norm theorem provide an exact sequence
\begin{equation*}
1\to
\lfaktor{E^{(1)}}{\mathbb{A}_E^{(1)}} \to
\lfaktor{E^\times}{\mathbb{A}_E^\times}\xrightarrow{\Nr}\lfaktor{\mathbb{Q}^\times}{\mathbb{A}^\times}\xrightarrow{\chi_E} \{\pm 1\}\to 1 \, .
\end{equation*}
This sequence descends to an exact sequence
\begin{equation*}
1\to
\Pic^\mathrm{pg}(\Lambda,\mathrm{x}) \to
\Pic(\Lambda,\mathrm{x})\xrightarrow{\Nr}
\dfaktor{\mathbb{Q}^\times}{\mathbb{A}^\times}{\mathbb{R}_{>0}\prod_{v<\infty} \Nr \Lambda_v^\times(\mathrm{x})}
\xrightarrow{\chi_E} \{\pm 1\}\to 1
\end{equation*}
and a dual exact sequence
\begin{equation*}
1\leftarrow
\widehat{\Pic^\mathrm{pg}(\Lambda,\mathrm{x})} \leftarrow
\widehat{\Pic(\Lambda,\mathrm{x})}
\xleftarrow{\widehat{\Nr}}\widehat{\dfaktor{\mathbb{Q}^\times}{\mathbb{A}^\times}
{\mathbb{R}_{>0}\prod_{v<\infty} \Nr \Lambda_v^\times(\mathrm{x})}}\xleftarrow{\widehat{\chi_E}} \{\pm 1\}\leftarrow 1 \, .
\end{equation*}
The exactness of the latter sequence implies that the following sequence is also exact
\begin{equation}\label{eq:Delta-exact-seq}
1\to\langle\chi_E\rangle \to \Delta \to \Pic^\mathrm{pg}(\Lambda,\mathrm{x})^\perp\to 1 \, ,
\end{equation}
where $\Delta$ is the Pontryagin dual of $\dfaktor{\mathbb{Q}^\times}{\mathbb{A}^\times}{\mathbb{R}_{>0}\prod_{v<\infty}\Nr \Lambda_v^\times(\mathrm{x})}$.
Because $\Pic^\mathrm{pg}(\Lambda,\mathrm{x})^\perp$ is $2$-torsion and $\ord \chi_E=2$ we deduce from \eqref{eq:Delta-exact-seq} that $\Delta$ is $4$-torsion. Hence $\Delta$ is contained in the group of characters $\chi\colon \lfaktor{\mathbb{Q}^\times}{\mathbb{A}^\times}\to S^1$ with $\ord\chi\mid 4$ and conductor contained in $\mathbb{R}_{>0}\prod_{v<\infty}\Nr \Lambda_v^\times(\mathrm{x})$.

For any $v\nmid f\ord_{\mathcal{H}}(\mathrm{x})$ we have $\Lambda_v^\times(\mathrm{x})=\mathcal{O}_{E_v}^\times$ and $E_v/\mathbb{Q}_v$ is an unramified extension of \'etale-algebras. Local class field theory implies $\Nr\Lambda_v^\times(\mathrm{x})=\Nr\mathcal{O}_{E_v}^\times=\mathbb{Z}_v^\times$ for these $v$. We deduce that all $\chi\in \Delta$ are unramified outside $f\ord_{\mathcal{H}}(\mathrm{x})$. The multiplicative structure of id\`ele class group characters implies
\begin{equation*}
\Delta\subset \prod_{p\mid f\ord_{\mathcal{H}}(\mathrm{x})} \widehat{\mathbb{Z}_p^\times}[4] \, .
\end{equation*}
The isomorphism $\mathbb{Z}_p^\times\simeq \mu_{p-1}\times \mathbb{Z}_p$ implies $\widehat{\mathbb{Z}_p^\times}\simeq \mu_{p-1}\times \mathbb{Q}_p/\mathbb{Z}_p$
and $\widehat{\mathbb{Z}_p^\times}[4]\simeq \mu_{p-1}[4]\times \mathbb{Q}_p/\mathbb{Z}_p[4]$. Using the equalities
\begin{align*}
\mu_{p-1}[4]&\simeq\begin{cases}
1 & p=2\\
\mu_2 & p\equiv 3\mod 4\\
\mu_4 & p\equiv 1\mod 4
\end{cases} \, , &
\mathbb{Q}_p/\mathbb{Z}_p[4]&\simeq\begin{cases}
1 & p\neq 2\\
\cyclic{4} & p=2
\end{cases} \, ,
\end{align*}
and \eqref{eq:Delta-exact-seq} we conclude
\begin{equation*}
\left|\Pic^\mathrm{pg}(\Lambda,\mathrm{x})^\perp\right|=\frac{1}{2}|\Delta|\leq  \frac{1}{2} \prod_{p\mid f\ord_{\mathcal{H}}(\mathrm{x})} |\widehat{\mathbb{Z}_p^\times}[4]|
\ll 4^{\omega(f\ord_{\mathcal{H}}(\mathrm{x}))} \, .
\end{equation*}
\end{proof}

\begin{proof}[Proof of Theorem \ref{thm:Cor-main-inequality}]
Denote $\alpha=\left(\frac{X_{\min}}{X_{\max}}\right)^2$ and fix $c>1$. Apply Proposition \ref{prop:geom-expansion-perron}. We need to evaluate for each $\chi\in\Pic^\mathrm{pg}(\Lambda,\mathrm{x})$ the integral $\int_{c-i\infty}^{c+i\infty} F_\chi(s)\dif s$. Where the  integrand is
\begin{equation*}
F_\chi(s)\coloneqq L_{\Lambda(\mathrm{x},\mathrm{y})}(s,\chi)
\mathcal{M}\varphi(s,\alpha)\left(\sqrt{|D|}X_{\max}^2\right)^s \, .
\end{equation*}
The function $F_\chi(s)$ is meromorphic in the strip $\Re s>0$ with at most a simple pole at $s=1$ with residue
\begin{align*}
\Res_{s=1} F_\chi(s)&=\mathcal{M}\varphi(1,\alpha) \sqrt{|D|}X_{\max}^2\Res_{s=1}L_{\Lambda(\mathrm{x},\mathrm{y})}(s,\chi)\\
&\ll (1-\alpha)\sqrt{|D|}X_{\max}^2\Res_{s=1}L_{\Lambda(\mathrm{x},\mathrm{y})}(s,\chi)\\
&=\sqrt{|D|}\left(X_{\max}^2-X_{\min}^2\right)\Res_{s=1}L_{\Lambda(\mathrm{x},\mathrm{y})}(s,\chi) \, ,
\end{align*}
where we have used the property $\mathcal{M}\varphi(1,\alpha)=\int \varphi(x,\alpha)\dif x \ll 1-\alpha$.
This residue vanishes unless $\chi$ is trivial.
We would like to shift the contour of integration of $F_\chi(s)$ to the vertical line $\int_{1/2-i\infty}^{1/2+i\infty}$ using the residue theorem. In the process we collect the potential residue at $s=1$ and the contribution from the horizontal lines $\lim_{T\to\infty} \left(\int_{1/2- iT}^{c- iT}-\int_{1/2+iT}^{c+iT}\right)$. We argue that the horizontal contribution vanishes.

Because $\varphi$ is smooth and compactly supported its Mellin transform decays faster then any polynomial in the vertical direction uniformly for $1/2\leq \Re s \leq c$.
The convexity bound for quadratic Hecke $L$-functions, Lemma \ref{lem:local-vertical} and the trivial bound $|1-p^{-s}|^{-1}\geq \left(1+p^{-\Re s}\right)^{-1}$ imply that
\begin{equation}\label{eq:L-decay}
|L_{\Lambda(\mathrm{x},\mathrm{y})}(s,\chi)|\ll_{\Lambda,\mathrm{x},\mathrm{y}} 1+|\Im s|
\end{equation}
for any $\Re s\geq1/2$. Hence in the strip $1/2\leq \Re s \leq c$ the function $|F_\chi(s)|$ decays uniformly to $0$ when $|\Im s|  \to\infty$. This implies the vanishing of the horizontal contribution.

Next we need to evaluate each integral $\int_{1/2-i\infty}^{1/2+i\infty}F_\chi(s)\dif s$. We can bound $L_{\Lambda(\mathrm{x},\mathrm{y})}$ on the critical line $\Re s=1/2$ using the subconvexity bound of Duke, Friedlander and Iwaniec for Hecke $L$-function. If $L(s,\chi)$ is the Hecke $L$-function of the Grossencharakter $\chi$ then $L_{\Lambda(\mathrm{x},\mathrm{y})}$ differs from $L(s,\chi)$ at only finitely many terms of the Euler product which are evaluated explicitly in Appendix \ref{app:L-functions}. Specifically, Proposition \ref{prop:L1/2-bound} implies
\begin{align*}
\Bigg|\int_{1/2-i\infty}^{1/2+i\infty} F_\chi(s) \dif s\Bigg|
\ll_{\varepsilon} \int_{1/2-i\infty}^{1/2+i\infty}
&|L(s,\chi) \mathcal{M}\varphi(s,\alpha)| \dif s |D|^{1/4} X_{\max}\,p_1^{-\tau}\\
&\cdot\ord_{\mathcal{H}}(l.\mathrm{y})^2
(f\ord_{\mathcal{H}}(l.\mathrm{y}))^{\varepsilon}  12^{\omega(\ord_{\mathcal{H}}(\mathrm{x}))}\\
&\cdot \left(\Nr\left(\mathcal{O}_E:\jmath_{/\Lambda}(\mathrm{x})\right)\right)^{-1/2}
\prod_{v\mid\ord_{\mathcal{H}}(\mathrm{x})} \left[\Lambda_v^\times:\Lambda_v^\times(\mathrm{x})\right]\, .
\end{align*}

Because $\chi$ is a character of $\Pic(\Lambda,\mathrm{x})=\dfaktor{E^\times}{\mathbb{A}_{E,f}^\times}{\Lambda_f^\times(\mathrm{x})}$, the classical Grossencharakter associated to $\chi$  has conductor ideal dividing the ideal $\mathfrak{c}(\Lambda_f^\times(\mathrm{x}))$ from Lemma \ref{lem:conductor}.
The theta lift of $\chi$ is an $\mathbf{SL}_2$ modular form with level dividing $|D_E|\Nr \mathfrak{c}(\Lambda_f^\times(\mathrm{x}))$, weight $1$ and nebentypus $\chi_E$, cf.\ \cite[Proposition 12.5]{iwaniec1997topics}.
The subconvex bound \cite[Theorem 2.4]{DFI} whenever $\chi$ is non-trivial and the Burgess bound \cite{Burgess2} for trivial $\chi$ imply that there is some explicit $\delta>0$ such that
\begin{equation*}
\int_{1/2-i\infty}^{1/2+i\infty}
|L(s,\chi) \mathcal{M}\varphi(s,\alpha)| \dif s\ll \left(|D_E|\Nr \mathfrak{c}(\Lambda_f^\times(\mathrm{x}))\right)^{1/4-\delta}
\int_{1/2-i\infty}^{1/2+i\infty} (1+|s|)^{10} |\mathcal{M}\varphi(s,\alpha)| \dif s \, .
\end{equation*}
We now define
\begin{equation*}
\Phi(\alpha)\coloneqq \int_{1/2-i\infty}^{1/2+i\infty} (1+|s|)^{10} |\mathcal{M}\varphi(s,\alpha)| \dif s \, .
\end{equation*}
Because $\varphi(x,\alpha)$ is smooth and compactly supported in $x$ for each $\alpha$ the Mellin transform $\mathcal{M}\varphi(s,\alpha)$ decays faster then any polynomial in the vertical line $\Re s=1/2$ uniformly in $\alpha$ on compact sets $\subset [0,1)$. Hence the integral $\int_{1/2-i\infty}^{1/2+i\infty} (1+|s|)^{10} |\mathcal{M}\varphi(s)| \dif s$ converges to a finite positive constant that depends continuously on $\alpha\in[0,1)$. Applying Lemma \ref{lem:conductor} to bound $\Nr\mathfrak{c}(\Lambda^\times_f(\mathrm{x}))$ we deduce
\begin{align*}
\left|\int_{1/2-i\infty}^{1/2+i\infty} F_\chi(s) \dif s\right|\ll_\varepsilon& \Phi(\alpha)(f\ord_{\mathcal{H}}(l.\mathrm{y}))^\varepsilon
|D|^{1/2-\delta} X_{\max}\,p_1^{-\tau}\\
&12^{\omega(\ord_{\mathcal{H}}(\mathrm{x}))}\left(\Nr\left(\mathcal{O}_E:\jmath_{/\Lambda}(\mathrm{x})\right)\right)^{-1/4-\delta} \\
&\ord_{\mathcal{H}}(l.\mathrm{y})^2
\prod_{v\mid\ord_{\mathcal{H}}(\mathrm{x})} \left[\Lambda_v^\times:\Lambda_v^\times(\mathrm{x})\right] \, .
\end{align*}

Combining these results into Proposition \ref{prop:geom-expansion-perron} and using Propositions \ref{prop:toral-volume} and \ref{prop:residue} we deduce the required expression
\begin{align*}
\Cor(\mu,\nu)[f]-\RO_f(0)&\ll
\meas_{\mathbf{G}(\mathbb{A})} \left(\Ad_{(e,\mathrm{y})} (B^{-1}B) \right)
\ord_{\mathcal{H}}(l.\mathrm{y})^2 (\mathrm{MT}+\mathrm{ST}) \, ,\\
\mathrm{MT}&\coloneqq \left(X_{\max}^2-X_{\min}^2 \right)p_1^{-2\tau} \, ,\\
\left|\mathrm{ST}\right|\ &\ll_\varepsilon
\Phi(\alpha)
(f\ord_{\mathcal{H}}(l.\mathrm{y}))^\varepsilon  \left[\Pic(\Lambda,\mathrm{x}):\Pic^\mathrm{pg}(\Lambda,\mathrm{x})\right]
L(1,\chi_E)^{-1}\\
&\cdot|D|^{-\delta}  X_{\max}\,p_1^{-\tau}
 12^{\omega(\ord_{\mathcal{H}}(\mathrm{x}))}\left(\Nr\left(\mathcal{O}_E:\jmath_{/\Lambda}(\mathrm{x})\right)\right)^{-1/4-\delta} \, .
\end{align*}
To conclude the proof we need only show the correct upper bound to $|\mathrm{ST}|$.
Notice that up till now all dependence on the parameter $\varepsilon$ is effective and can be made completely explicit. The second order term $\mathrm{ST}$ should be negligible compared to the main order term $\mathrm{MT}$ whenever $|D|\to\infty$ and $\ord_{\mathcal{H}}(\mathrm{x})\to\infty$. To see this we apply Siegel's bound to deduce $L(1,\chi_E)\gg |D|^{-\varepsilon}$ \emph{ineffectively}. This and Lemma \ref{lem:principal-genus} imply
\begin{equation}\label{eq:ST}
\left|\mathrm{ST}\right|\ \ll_\varepsilon \Phi(\alpha)
(f\ord_{\mathcal{H}}(l.\mathrm{y}))^\varepsilon
|D|^{-\delta+\varepsilon}  X_{\max}\,p_1^{-\tau}\\
\cdot 48^{\omega(\ord_{\mathcal{H}}(\mathrm{x}))}\left(\Nr\left(\mathcal{O}_E:\jmath_{/\Lambda}(\mathrm{x})\right)\right)^{-1/4-\delta} \, .
\end{equation}
We are left with bounding the dependence of the second order term on $\mathrm{x}$. Recall from Lemma \ref{lem:conductor} that
\begin{equation*}
\left(\Nr\left(\mathcal{O}_E:\jmath_{/\Lambda}(\mathrm{x})\right)\right)^{-1}=\prod_{v<\infty}\prod_{w\mid v} \min\{|\mathrm{x}_w|_w^{-1},1\} \, .
\end{equation*}
Let $v\mid\ord_{\mathcal{H}}(\mathrm{x})$.
Notice that if $\mathrm{x}_v\not\in\mathcal{O}_{E_v}$, which is always the case if $\Lambda_v=\mathcal{O}_{E,v}$, then $\prod_{w\mid v} \min\{|\mathrm{x}_w|_w^{-1},1\}\leq \ord_{\mathcal{H}}(\mathrm{x}_v)^{-1}$. If $\Lambda_v$ is non-maximal and $\mathrm{x}_v\in\mathcal{O}_{E_v}$ then necessarily $f_v\mathrm{x}_v\in\Lambda_v$ and $\ord_{\mathcal{H}}(\mathrm{x}_v)| f_v$. Hence
$$\prod_{w\mid v} \min\{|\mathrm{x}_w|_w^{-1},1\}\leq \left(\frac{\ord_{\mathcal{H}}(\mathrm{x}_v)}{\gcd(\ord_{\mathcal{H}}(\mathrm{x}_v),f_v)}\right)^{-1}$$
for  any $v\mid\ord_{\mathcal{H}}(\mathrm{x})$. Thus
\begin{equation*}
\left(\Nr\left(\mathcal{O}_E:\jmath_{/\Lambda}(\mathrm{x})\right)\right)^{-1}\leq
\left(\frac{\ord_{\mathcal{H}}(\mathrm{x})}{\gcd(\ord_{\mathcal{H}}(\mathrm{x}),f)}\right)^{-1}
\end{equation*}
and
\begin{align*}
48^{\omega(\ord_{\mathcal{H}}(\mathrm{x}))}\left(\Nr\left(\mathcal{O}_E:\jmath_{/\Lambda}(\mathrm{x})\right)\right)^{-1/4-\delta}
&\ll_\varepsilon f^\varepsilon 48^{\omega\left(\frac{\ord_{\mathcal{H}}(\mathrm{x})}{\gcd(\ord_{\mathcal{H}}(\mathrm{x}),f)}\right)} \left(\frac{\ord_{\mathcal{H}}(\mathrm{x})}{\gcd(\ord_{\mathcal{H}}(\mathrm{x}),f)}\right)^{-1/4-\delta}\\
&\ll_\varepsilon f^\varepsilon \left(\frac{\ord_{\mathcal{H}}(\mathrm{x})}{\gcd(\ord_{\mathcal{H}}(\mathrm{x}),f)}\right)^{-1/4-\delta+\varepsilon} \, .
\end{align*}
The claim follows by substituting this inequality into \eqref{eq:ST}.
\end{proof}

\section{Equidistribution of Genus Orbits}\label{sec:equidistribution}
\subsection{Equidistribution of Torus Orbits}
\begin{defi}
For any element $\mathrm{y}\in\mathbf{V}(\mathbb{A})$ denote $\ord_{\mathbf{V}(\mathbb{A}_f)}(\mathrm{y})$ to be the torsion order of the non-archimedean part $\mathrm{y}_f$ in the torsion group $\mathbf{V}(\mathbb{A}_f)/\mathbf{V}(\hat{\mathbb{Z}})$. Similarly, we let $\ord_{\mathbf{V}(\mathbb{Q}_p)}(\mathrm{y}_p)$ to be the torsion order of $\mathrm{y}_p$ in $\mathbf{V}(\mathbb{Q}_p)/\mathbf{V}(\mathbb{Z}_p)$. If $\mathrm{y}=\begin{pmatrix} a/b \\ c/d \end{pmatrix}$ is a rational element in lowest terms then $\ord_{\mathbf{V}(\mathbb{A}_f)}=\operatorname{lcm}(b,d)$.

Recall from Remark \ref{rem:torsion-order-translation} that for a homogeneous toral set $\mathcal{H}=\left[\mathbf{T}(\mathbb{A})(l,\mathrm{x})\right]$ we have $\ord_{\mathcal{H}}(l.\mathrm{y})=\ord_{\mathbf{V}(\mathbb{A}_f)}(\mathrm{y})$ for all
$\mathrm{y}\in\mathbf{V}(\mathbb{A})$.
\end{defi}
\begin{lem}\label{lem:AdYvol-padic}
For any prime $p$ and $\mathrm{y}_p\in\mathbf{V}(\mathbb{Q}_p)$ with $\mathrm{y}_p\not\in \mathbf{V}(\mathbb{Z}_p)$
\begin{equation*}
\meas_{\mathbf{G}(\mathbb{Q}_p)}\left(\Ad_{(e,\mathrm{y}_p)}
\left(\mathbf{P}^1(\mathbb{Z}_p)^{-1} \mathbf{P}^1(\mathbb{Z}_p)\right)\right)=
\frac{1}{\ord_{\mathbf{V}(\mathbb{Q}_p)}(\mathrm{y}_p)^2(1-p^{-2})} \meas_{\mathbf{G}(\mathbb{Q}_p)}\left(\mathbf{P}^1(\mathbb{Z}_p)\right) \, .
\end{equation*}
If $\mathrm{y}_p\in \mathbf{V}(\mathbb{Z}_p)$ then
\begin{equation*}
\meas_{\mathbf{G}(\mathbb{Q}_p)}\left(\Ad_{(e,\mathrm{y}_p)}
\left(\mathbf{P}^1(\mathbb{Z}_p)^{-1} \mathbf{P}^1(\mathbb{Z}_p)\right)\right)=
\meas_{\mathbf{G}(\mathbb{Q}_p)}\left(\mathbf{P}^1(\mathbb{Z}_p)\right) \, .
\end{equation*}
\end{lem}
\begin{proof}
The second statement is trivial while the first is a standard computation that we include for the sake of completeness. Write $\ord_{\mathbf{V}(\mathbb{Q}_p)}(\mathrm{y}_p)=p^m\geq p$. There is some $k\in\mathbf{G}(\mathbb{Z}_p)$ such that $k.p^m\mathrm{y}_p=\begin{pmatrix}1 \\ 0\end{pmatrix}$. Computing we see
\begin{align*}
\meas_{\mathbf{G}(\mathbb{Q}_p)}&\left(\Ad_{(e,\mathrm{y}_p)}
\left(\mathbf{P}^1(\mathbb{Z}_p)^{-1} \mathbf{P}^1(\mathbb{Z}_p)\right)\right)=
\meas_{\mathbf{G}(\mathbb{Q}_p)}\left(\left\{b\in\mathbf{SL}_2(\mathbb{Z}_p) \mid b.\mathrm{y}_p
-\mathrm{y}_p\in \mathbb{Z}_p^2
\right\}\right)\\
&=
\meas_{\mathbf{G}(\mathbb{Q}_p)}\left(\left\{b\in\mathbf{SL}_2(\mathbb{Z}_p)
\mid bk. \begin{pmatrix}1 \\ 0\end{pmatrix}
-k.\begin{pmatrix}1 \\ 0\end{pmatrix}\in p^m\mathbb{Z}_p^2
\right\}\right)\\
&=
\meas_{\mathbf{G}(\mathbb{Q}_p)}\left(\left\{b\in\mathbf{SL}_2(\mathbb{Z}_p)
\mid b. \begin{pmatrix}1 \\ 0\end{pmatrix}
-\begin{pmatrix}1 \\ 0\end{pmatrix}\in p^m\mathbb{Z}_p^2
\right\}\right) \, .
\end{align*}
The last measure is inverse proportional to the index of the upper triangular unipotent subgroup in $\mathbf{SL}_2\left(\cyclic{p^m}\right)$.
\end{proof}

\begin{lem}\label{lem:AdYvol-real}
Fix $R_G,R_V>0$.
Let $B_\infty=B_{G_\infty}(R_G)\times B_{V,\infty}(R_V)\subset \mathbf{P}^1(\mathbb{R})$ as in Definition \ref{defi:test-function-archimedean}. Then for any $\mathrm{y}_\infty\in\mathbf{V}(\mathbb{R})$
\begin{equation*}
\meas_{\mathbf{G}(\mathbb{R})}\left(\Ad_{(e,\mathrm{y}_\infty)}
\left(B_\infty^{-1}B_\infty\right)\right)\leq
\meas_{\mathbf{G}(\mathbb{R})}\left(B_{G_\infty}\left(2R_G\right)\right) \, .
\end{equation*}
\end{lem}
\begin{remark}
With a slightly more delicate analysis it is possible to establish the better estimate
\begin{align*}
\meas_{\mathbf{G}(\mathbb{R})}&\left(\Ad_{(e,\mathrm{y}_\infty)}
\left(B_\infty^{-1}B_\infty\right)\right)\leq
\meas_{\mathbf{G}(\mathbb{R})}\left(B_{G_\infty}\left(R_G'\right)\right) \, ,\\
R_G'&=\min\left\{R_G,
\log\left(1+\frac{2\exp\left(\frac{R_G}{2}\right)R_V}{\|\mathrm{y}_\infty\|_2}\right)
\right\} \, .
\end{align*}
Notice that for large $\|\mathrm{y}_\infty\|_2$ the value of $\meas_{\mathbf{G}(\mathbb{R})}\left(B_{G_\infty}\left(R_G'\right)\right)$ is proportional to $\|\mathrm{y}_\infty\|_2^{-2}$ -- similarly to the $p$-adic case.
\end{remark}
\begin{proof}
From Definition \ref{defi:test-function-archimedean} it follows that $B_{G_\infty}(R_G)^{-1}=B_{G_\infty}(R_G)$ and $-B_{V,\infty}(R_V)=B_{V,\infty}(R_V)$.
The triangle inequality on the hyperbolic plane implies that $B_{G_\infty}(R_G)B_{G_\infty}(R_G)=B_{G_\infty}(2R_G)$ and from the Euclidean triangle inequality we deduce $B_{V,\infty}(R_V)+B_{V,\infty}(R_V)=B_{V,\infty}(2R_V)$. Moreover, $B_{G_\infty}(R_G).B_{V,\infty}(2R_V)=B_{V,\infty}(2\exp\left(\frac{R_G}{2}\right)R_V)$ and
\begin{align*}
B_\infty^{-1}B_\infty&= B_{G_\infty}(R_G)^{-1}\cdot \left[B_{G_\infty}(R_G)\times (B_{V_\infty}(R_V)-B_{V_\infty}(R_V))\right]\\
&\subset B_{G_\infty}(2R_G)\times B_{V,\infty}(2\exp\left(\frac{R_G}{2}\right)R_V) \, .
\end{align*}
We can then write
\begin{align*}
\Ad_{(e,\mathrm{y}_\infty)}\left(B_\infty^{-1}B_\infty\right)&\subset
\left\{(b,\mathrm{w}+\mathrm{y}_\infty-b.\mathrm{y}_\infty) \mid
b\in  B_{G_\infty}(2R_G), \mathrm{w}\in B_{V,\infty}(2\exp\left(\frac{R_G}{2}\right)R_V)
\right\}\\
&\Rightarrow
\meas_{\mathbf{G}(\mathbb{R})}\left(\Ad_{(e,\mathrm{y}_\infty)}\left(B_\infty^{-1}B_\infty\right)\right)\leq  \meas_{\mathbf{G}(\mathbb{R})}(B(2R_G,2\exp\left(\frac{R_G}{2}\right)R_V,\mathrm{y}_\infty)) \, ,\\
B(R_1,R_2,\mathrm{r})&\coloneqq \left\{b\in B_{G_\infty}(R_1) \mid b.\mathrm{r}-\mathrm{r}
\in B_{V,\infty}(R_2) \right\} \subset \mathbf{G}(\mathbb{R}) \, .
\end{align*}
The proof concludes by noticing that $B(R_1,R_2,\mathrm{r})\subset B_{G,\infty}(R_1)$ for all $R_1,R_2>0$ and $\mathrm{r}\in\mathbb{R}^2$. The more delicate estimate in the remark can be proven by evaluating the integral
\begin{equation*}
\int_{B_{G,\infty}(R_1)} \mathds{1}_{[0,R_2]}\left(\|b.\mathrm{r}-\mathrm{r}\|_2\right)
\dif \meas_{\mathbf{G}(\mathbb{R})}(b)
\end{equation*}
using the Cartan decomposition formula for the Haar measure.
\end{proof}
\begin{cor}\label{cor:AdYB}
Let $f=f_{\tau, R_G, R_V}=\mathds{1}_B$ be a Bowen ball test function. If $\tau>0$ then assume in addition $\mathrm{y}_{p_1}=0$.
Then for any $\mathrm{y}\in\mathbf{V}(\mathbb{A})$
\begin{align*}
\meas_{\mathbf{G}(\mathbb{A})}\left(\Ad_{(e,\mathrm{y})} (B^{-1}B) \right)
\ll  \ord_{\mathbf{V}(\mathbb{A}_f)}(\mathrm{y})^{-2}
\meas_{\mathbf{G}(\mathbb{R})}\left(B_{G_\infty}\left(2R_G\right)\right) p_1^{-4\tau} \, .
\end{align*}
\end{cor}
\begin{proof}
From Definition \ref{defi:Bowen-ball-test-function} and Lemmata \ref{lem:AdYvol-padic},  \ref{lem:AdYvol-real} we deduce
\begin{align*}
\meas_{\mathbf{G}(\mathbb{A})}\left(\Ad_{(e,\mathrm{y})} (B^{-1}B) \right)
&\ll  \meas_{\mathbf{G}(\mathbb{R})}\left(B_{G_\infty}\left(R_G'\right)\right) \meas_{\mathbf{G}(\mathbb{Q}_{p_1})}\left(\mathbf{P}^1(\mathbb{Z}_{p_1})^{(-\tau,\tau)}\right)\\
&\cdot
\prod_{p\mid \ord_{\mathbf{V}(\mathbb{A}_f)}(\mathrm{y})} \ord_{\mathbf{V}(\mathbb{Q}_p)}(\mathrm{y}_p)^{-2}(1-p^{-2})^{-1}\\
&\leq  \meas_{\mathbf{G}(\mathbb{R})}\left(B_{G_\infty}\left(R_G'\right)\right)
\meas_{\mathbf{G}(\mathbb{Q}_{p_1})}\left(\mathbf{G}(\mathbb{Z}_{p_1})^{(-\tau,\tau)}\right)
\ord_{\mathbf{V}(\mathbb{A}_f)}(\mathrm{y})^{-2} \zeta(2) \, .
\end{align*}
We need only show $\meas_{\mathbf{G}(\mathbb{Q}_{p_1})}\left(\mathbf{G}(\mathbb{Z}_{p_1})^{(-\tau,\tau)}\right)\ll p_1^{-4\tau}$.
This last inequality is easy to prove using $\mathscr{B}$ -- the Bruhat-Tits tree of $\mathbf{SL}_2(\mathbb{Q}_{p_1})$. Observe that if $\mathbf{SL}_2(\mathbb{Z}_{p_1})$ is the stabilizer of the vertex $x_0$ then $\mathbf{SL}_2(\mathbb{Z}_{p_1})^{(-\tau,\tau)}$ is the stabilizer of a path of length $4\tau$ centered at $x_0$. Using the strong transitivity of the action of $\mathbf{SL}_2(\mathbb{Q}_{p_1})$ on $\mathscr{B}$ it is easy to compute the index
\begin{equation*}
\left[\mathbf{SL}_2(\mathbb{Z}_{p_1}):\mathbf{SL}_2(\mathbb{Z}_{p_1})^{(-\tau,\tau)}\right]=
(p_1+1)p_1^{4\tau-1} \, .
\end{equation*}
\end{proof}

\begin{prop}\label{prop:final-strict-sequence}
Let $f=f_{\tau, R_G, R_V}=\mathds{1}_B$ be a Bowen ball test function and let $\{\mathcal{H}_i\subset\left[\mathbf{P}^1(\mathbb{A})\right]\}$ be a strict sequence of $K_\infty$-invariant homogeneous toral sets. If $\tau>0$ then assume in addition that all $\mathcal{H}_i$ are $A_{p_1}$-invariant. Denote by $\mu_i$ the periodic measure supported on $\mathcal{H}_i$.
Fix $\mathcal{Q}$ a \emph{compactly supported} probability measure on $\mathbf{V}(\mathbb{A})$ and if $\tau>0$ assume that $\mathrm{y}_{p_1}=0$ for $\mathcal{Q}$ almost every $\mathrm{y}$.  Denote by $\nu_\mathrm{y}$ the periodic measure on $\left[\mathbf{G}(\mathbb{A})(e,\mathrm{y})\right]\subset\left[\mathbf{P}^1(\mathbb{A})\right]$.

Set
\begin{align*}
X_{\min}&=\max\left\{\exp\left(-\frac{R_G}{2}\right)\|\mathrm{y}_\infty\|_2-R_V ,0\right\} \, ,\\
X_{\max}&=\exp\left(\frac{R_G}{2}\right)\|\mathrm{y}_\infty\|_2+R_V \, .
\end{align*}
Then
\begin{align*}
\limsup_{i\to\infty} \Cor\left(\mu_i,\int \nu_\mathrm{y} \dif\mathcal{Q}(\mathrm{y})\right)[f]&\ll
\meas_{\mathbf{G}(\mathbb{R})}\left(B_{G_\infty}\left(2R_G\right)\right)
\left(X_{\max}^2-X_{\min}^2\right)p_1^{-6\tau} \, .
\end{align*}
\end{prop}
\begin{proof}
Using the linearity of cross-correlation observe that for all $i$
\begin{equation*}
\Cor\left(\mu_i,\int \nu_\mathrm{y} \dif\mathcal{Q}(\mathrm{y})\right)[f]=
\int \Cor\left(\mu_i, \nu_\mathrm{y} \right)[f] \dif\mathcal{Q}(\mathrm{y}) \, .
\end{equation*}
We write $\mathcal{H}_i=\left[\mathbf{T}_i(\mathbb{A})(l_i,\mathrm{x}_i)\right]$ and bound the integral over cross-correlation using Theorem \ref{thm:Cor-main-inequality}.
\begin{align}\label{eq:integral-correlation-inequaltity}
&\int \Cor\left(\mu_i, \nu_\mathrm{y} \right)[f] \dif\mathcal{Q}(\mathrm{y})-
\int \RO^i_{f}(0)\dif\mathcal{Q}(\mathrm{y})\\\nonumber
\ll
&\int \meas_{\mathbf{G}(\mathbb{A})} \left(\Ad_{(e,\mathrm{y})} (B^{-1}B) \right)
\ord_{\mathbf{V}(\mathbb{A}_f)}(\mathrm{y})^2 (\mathrm{MT}+\mathrm{ST})\dif\mathcal{Q}(\mathrm{y}) \, ,\\\nonumber
\mathrm{MT}&\coloneqq \left(X_{\max}^2-X_{\min}^2 \right)p_1^{-2\tau} \, ,\\\nonumber
\left|\mathrm{ST}\right|\ &\ll_\varepsilon \Phi\left(\alpha\right)
 (f_i\ord_{\mathbf{V}(\mathbb{A}_f)}(\mathrm{y}))^\varepsilon
|D_i|^{-\delta+\varepsilon}  X_{\max}\,p_1^{-\tau}
\left(\frac{\ord_{\mathcal{H}_i}({\mathrm{x}}^i)}{\gcd(\ord_{\mathcal{H}_i}({\mathrm{x}}^i),f_i)}\right)^{-1/4-\delta+\varepsilon} \, .
\end{align}
where $D_i=D_i^\mathrm{fund}f_i^2$ is the discriminant of $\mathcal{H}_i$
and $\alpha=X_{\min}^2/X_{\max}^2$. Write $\supp f =B=B_G\times B_V$ where $B_G\subset\mathbf{G}(\mathbb{A})$ and $B_V\subset\mathbf{V}(\mathbb{A})$. From Proposition \ref{prop:trivial-contribution-vanishes} we deduce that $\RO^i_{f}(0)=0$ for all $\mathrm{y}\in \supp\mathcal{Q}$ whenever $l_i.\mathrm{x}^i\not\in -B_V+B_G.\supp\mathcal{Q}$. Thus because the sequence $\mathcal{H}_i$ is strict $\int \RO^i_{f}(0) \dif\mathcal{Q}(\mathrm{y})=0$ for $i\gg_{\supp \mathcal{Q}, R_G, R_V}1$.

The strictness assumption also implies that $|D_i|\to_{i\to\infty}\infty$. Moreover, $$\Phi(\alpha)\left(\ord_{\mathbf{V}(\mathbb{A}_f)}(\mathrm{y})\right)^\varepsilon X_{\max}$$ is a continuous function of $\mathrm{y}$ hence it is bounded on $\supp\mathcal{Q}$. Thus $|\mathrm{ST}|\to_{i\to\infty}0$ uniformly in $\supp \mathcal{Q}$. These facts in conjunction with \eqref{eq:integral-correlation-inequaltity} imply that
\begin{equation*}
\limsup_{i\to\infty}\Cor\left(\mu_i,\int \nu_\mathrm{y} \dif\mathcal{Q}(\mathrm{y})\right)[f]
\ll \mathrm{MT}\int \meas_{\mathbf{G}(\mathbb{A})} \left(\Ad_{(e,\mathrm{y})} (B^{-1}B) \right)
\ord_{\mathbf{V}(\mathbb{A}_f)}(\mathrm{y})^2 \dif\mathcal{Q}(\mathrm{y}) \, .
\end{equation*}
The claim follows by substituting Corollary \ref{cor:AdYB} into the inequality above.
\end{proof}

\begin{proof}[Proof of Theorem \ref{thm:main}]
Due to Corollary \ref{cor:reduction-to-r=1} it is enough to consider the case $r=1$. Let $\left\{\mathcal{H}_i\subset \left[\mathbf{P}^1(\mathbb{A})\right]\right\}$ be a strict sequence of homogeneous toral sets such that the congruence conditions \eqref{eq:congruence-1} and \eqref{eq:congruence-2} of Theorem \ref{thm:main} are satisfied. Denote by $\mu_i$ the periodic measure supported on $\mathcal{H}_i$. By Theorem \ref{thm:rigidity-r=1} there is a pre-compact sequence $\{\xi_i\}_i\subset \mathbf{P}^1(\mathbb{A})$ such that the periodic measure $\xi_{i.*}\mu_i$ on the homogeneous toral set $\mathcal{H}_i\xi_i$ is $A_{p_1}\times A_{p_2}$ invariant for all $i$. Let $\mu$ be any weak-$*$ limit point of $\{\mu_i\}_i$. Corollary \ref{cor:structure-of-limit-r=1} implies that there is some $\xi\in\overline{\{\xi_i\}_i}$ and $c\geq 0$ such that
\begin{equation}\label{eq:limit-meas-decompose}
\xi_*.\mu=(1-c)\meas_{\mathbf{P}^1}+c\int \nu_\mathrm{y} \dif\mathcal{P}(y) \, ,
\end{equation}
where $\nu_\mathrm{y}$ is the periodic measure supported on $\left[\mathbf{G}(\mathbb{A})(e,\mathrm{y})\right]$ and $\mathcal{P}$ is a probability measure on $\mathbf{V}(\mathbb{A})$ such that $\mathrm{y}_{p_1}=\mathrm{y}_{p_2}=0$ for $\mathcal{P}$ almost all $\mathrm{y}$. The claim would follow if we show that $c=0$.

For any $\tau\in\mathbb{Z}_{\geq 0}$, $R_V,R_G>0$ let $f=f_{\tau,R_G,R_V}=\mathds{1}_{B(\tau)}$ be the Bowen ball test function from Definition \ref{defi:Bowen-ball-test-function}. Fix  a compact set $C_V\subset\mathbf{V}(\mathbb{A})$ large enough so that $\mathcal{P}(C_V)>0$ and define a new compactly supported probability measure on $\mathbf{V}(\mathbb{A})$ by conditioning on $C_V$
\begin{equation*}
\mathcal{Q}(A)\coloneqq \frac{\mathcal{P}(C_V\cap A)}{\mathcal{P}(C_V)} \, .
\end{equation*}
The function $f$ is the characteristic function of an open-set, hence weak-$*$ convergence and \eqref{eq:limit-meas-decompose} imply
\begin{align}\nonumber
c \mathcal{P}(C_V)\Cor\left(\int \nu_\mathrm{y}\dif\mathcal{Q}(\mathrm{y}),\int \nu_\mathrm{y}\dif\mathcal{Q}(\mathrm{y})\right)\left[f\right]
&\leq\Cor\left(\xi_*.\mu,\int
 \nu_\mathrm{y}\dif\mathcal{Q}(\mathrm{y})\right)\left[f\right]\\
&\leq \limsup_{i\to\infty} \Cor\left(\xi_{i,*}.\mu_i,\int \nu_\mathrm{y}\dif\mathcal{Q}(\mathrm{y})\right) \, .
\label{eq:integral-correlation}
\end{align}
We would like to apply Proposition \ref{prop:final-strict-sequence} to the $\limsup$ above. We need to show that the the sequence of $K_\infty$-invariant homogeneous toral sets $\mathcal{H}_i\xi_i$ is strict. It is already known to be $A_{p_1}$-invariant.
Write $\mathcal{H}_i\xi_i=\left[\mathbf{T}_i(\mathbb{A})(l_i',{\mathrm{x}'}^i)\right]$ and let $D_i'={D'}_i^\mathrm{fund}f_i'^2$ be the discriminant of the homogeneous toral set $\mathcal{H}_i\xi_i$.
We compute the discriminant and torsion order of
$\mathcal{H}_i\xi_i$ in terms of the discriminant and torsion order of $\mathcal{H}_i$. Recall from the proof of Theorem \ref{thm:rigidity-r=1} that $\xi_i$ is non-trivial only at the places $p_1$ and $p_2$. This immediately implies that $D_i$ and $D_i'$ agree at all primes except $p_1$, $p_2$; the same holds for the torsion order.

At the primes $p=p_1,p_2$ we know that $\mathcal{H}_i\xi_i$ is invariant under $A_p$. Because $A_p$ intersect $\mathbf{G}(\mathbb{Z}_p)$ at a maximal compact-open subgroup of $A_p$ we see from Definition \ref{defi:order} that the local order of $\mathcal{H}_i\xi_i$ at $p=p_1,p_2$ is maximal. Hence
\begin{equation*}
D_i'=D_i\prod_{p\in\{p_1,p_2\}}p^{-2\val_p(f_i)} \, ,
\end{equation*}
where $f_i$ is the conductor of the order attached to $\disc(\mathcal{H}_i)$. The congruence assumption \eqref{eq:congruence-1} and the strictness condition for $\left\{\mathcal{H}_i\right\}_i$ now imply that $D_i'\to_{i\to\infty}\infty$.
The fact that $\mathcal{H}_i\mathrm{\xi}_i$ is $A_{p_1}\times A_{p_2}$-invariant implies that the torsion order of  $\mathcal{H}_i\mathrm{\xi}_i$ is $1$ at the primes $p_1$, $p_2$ and we have
\begin{equation*}
\ord_{\mathcal{H}_i\xi_i}({\mathrm{x}'}^i)=N_i\prod_{p\in\{p_1,p_2\}}p^{-\val_p(N_i)} \, .
\end{equation*}
Again congruence condition \eqref{eq:congruence-2} now implies $\ord_{\mathcal{H}_i\xi_i}({\mathrm{x}'}^i)\to_{i\to\infty}\infty$, hence $l_i.{\mathrm{x}'}^i$ escapes all compact set in $\mathbf{V}(\mathbb{A})$ when $i\to\infty$ and the sequence is strict as claimed. We conclude using Proposition \ref{prop:final-strict-sequence}
\begin{equation*}
\limsup_{i\to\infty} \Cor\left(\xi_*.\mu,\int
 \nu_\mathrm{y}\dif\mathcal{Q}(\mathrm{y})\right) \ll_{C_V,R_G,R_V} p_1^{-6\tau} \, .
\end{equation*}
Assume in contradiction $c>0$ then the inequality above and \eqref{eq:integral-correlation} imply
\begin{equation*}
\Cor\left(\int \nu_\mathrm{y}\dif\mathcal{Q}(\mathrm{y}),\int \nu_\mathrm{y}\dif\mathcal{Q}(\mathrm{y})\right)\left[f\right]
\ll_{C_V,R_G,R_V} c^{-1} p_1^{-6\tau} \, .
\end{equation*}
The definitions of the cross-correlation and the Bowen ball test function \ref{defi:cross-correlation}, \ref{defi:Bowen-ball-test-function} imply immediately that
\begin{align}\label{eq:Bowen-balls-upper-bound}
\int \nu_\mathrm{y}\dif\mathcal{Q}(\mathrm{y}) &\times \int \nu_\mathrm{y}\dif\mathcal{Q}(\mathrm{y}) \left(\left\{x,y\in \left[\mathbf{P}^1(\mathbb{A})\right]^2 \mid y\in x B(\tau)
\right\}\right)\\
& \leq \Cor\left(\int \nu_\mathrm{y}\dif\mathcal{Q}(\mathrm{y}),\int \nu_\mathrm{y}\dif\mathcal{Q}(\mathrm{y})\right)\left[f\right]
\ll_{C_V,R_G,R_V} c^{-1} p_1^{-6\tau} \, ,\nonumber
\end{align}
where $B(\tau)$ is the support of $f$.
Recall from Definition \ref{defi:bowen-ball} that $B(\tau)\subset\mathbf{P}^1(\mathbb{A})$ is a Bowen ball of level $2\tau$ for $a\in A_{p_1}$; where $a$ generates $A_{p_1}/A_{p_1}^\circ\simeq \mathbb{Z}$. There is a direct relation between the decay rate of Bowen balls for $a$ as proven in \eqref{eq:Bowen-balls-upper-bound} and the Kolmogorov-Sinai entropy $h_{\int \nu_\mathrm{y}\dif\mathcal{Q}}(a)$. A standard adaption of \cite[Proposition 3.2]{ELMVPeriodic} to the $S$-arithmetic\footnote{This statement is closely related to the Brin-Katok theorem \cite{BrinKatok}.} setting says that \eqref{eq:Bowen-balls-upper-bound} implies
\begin{equation}\label{eq:P'-entropy-lower-bound}
h_{\int \nu_\mathrm{y}\dif\mathcal{Q}(\mathrm{y})}(a)\geq 3\log p_1 \, .
\end{equation}
On the other hand, entropy is a linear function of the measure hence
\begin{equation}\label{eq:P'-entropy-upper-bound}
h_{\int \nu_\mathrm{y}\dif\mathcal{Q}(\mathrm{y})}(a)
=\int h_{\nu_\mathrm{y}}(a) \dif\mathcal{Q}(\mathrm{y})\leq 2 \log p_1 \, .
\end{equation}
To prove the last inequality we use the fact that for every $\mathrm{y}\in\mathbf{V}(\mathbb{A})$ with $\mathrm{y}_{p_1}=0$ the periodic measure is a measure of maximal entropy on  the measurable dynamical system $\left[\mathbf{G}(\mathbb{A})(e,\mathrm{y})\right]$ with the left action of $a$; and that the entropy of the periodic measure is $2\log p_1$. Notice that this dynamical system is isomorphic to $\left[\mathbf{G}(\mathbb{A})\right]$ with the left action of $a$. This isomorphism sends the periodic measure to the Haar measure. The fact that the Haar measure has maximal entropy and its value is $2 \log p_1$ is a corollary of the relation between entropy and leaf-wise measures in the stable direction, cf.\ \cite[Theorem 7.9]{Pisa}. The inconsistent inequalities \eqref{eq:P'-entropy-lower-bound} and \eqref{eq:P'-entropy-upper-bound} contradict the assumption $c>0$.
\end{proof}
\begin{proof}[Proof of Theorem \ref{thm:main2}]
Let $f=f_{\tau=0, R_G, R_V}=\mathds{1}_{B(R_G,R_V)}$ be a Bowen ball test function with $\tau=0$. Set $\mathcal{Q}$ to be the point mass at $\mathrm{y}$
then Proposition \ref{prop:final-strict-sequence} implies
\begin{align*}
\Cor&\left(\mu,\nu_\mathrm{y} \right)[f]\leq
\limsup_{i\to\infty} \Cor\left(\mu_i,\nu_\mathrm{y} \right)[f]\ll
\meas_{\mathbf{G}(\mathbb{R})}\left(B_{G_\infty}\left(2R_G\right)\right)
\left(X_{\max}^2-X_{\min}^2 \right)
\\
&\ll \meas_{\mathbf{G}(\mathbb{R})}\left(B_{G_\infty}\left(2R_G\right)\right)
\begin{cases}
2\|\mathrm{y}_\infty\|_2^2\sinh(R_G)+4\|\mathrm{y}_\infty\|_2 \cosh(R_G/2)R_V &
\|\mathrm{y}_\infty\|_2\geq \exp\left(\frac{R_G}{2}\right)R_V\\
\left(\exp\left(\frac{R_G}{2}\right)\|\mathrm{y}_\infty\|_2+R_V\right)^2 &\|\mathrm{y}_\infty\|_2<  \exp\left(\frac{R_G}{2}\right)R_V
\end{cases} \, .
\end{align*}
In case $\mathrm{y}_\infty=0$ we are in the second case. This proves the second claimed statement.

To see that $\mu\left(\left[\mathbf{G}(\mathbb{A})(e,\mathrm{y})\right]\right)=0$ we observe first
\begin{equation}\label{eq:Cor(m,y)-asymp}
\lim_{(R_G,R_V)\to 0} \frac{\Cor\left(\mu,\nu_\mathrm{y} \right)[f_{0,R_G, R_V}]}{ \meas_{\mathbf{G}(\mathbb{R})}\left(B_{G_\infty}\left(2R_G\right)\right)}=0 \, .
\end{equation}
Assume $\mu\left(\left[\mathbf{G}(\mathbb{A})(e,\mathrm{y})\right]\right)>0$.
We will establish a contradiction with the fact that on $\left[\mathbf{G}(\mathbb{A})(e,\mathrm{y})\right]$ the uniform $\nu_\mathrm{y}$ mass of \emph{any} archimedean ball decays with the same rate when the center is restricted to a compact set. The latter restriction is required to avoid non-injectivity problems at the cusp.

Let $C\subset \left[\mathbf{G}(\mathbb{A})(e,\mathrm{y})\right]$ be a compact subset such that $\mu(C)>0$. For $R_G, R_V \ll_{C,\mathrm{y}} 1$ the quotient map\footnote{A restriction of the quotient map $\mathbf{P}^1(\mathbb{A})\to \left[\mathbf{P}^1(\mathbb{A})\right]$.} $\mathbf{G}(\mathbb{A})(e,\mathrm{y})\to\left[\mathbf{G}(\mathbb{A})(e,\mathrm{y})\right]$ is injective on $\mathbf{G}(\mathbb{A})(e,\mathrm{y})\cap x B(R_G,R_V)$ for all $x\in C$. The latter set contains an archimedean ball around $x\in\mathbf{G}(\mathbb{A})(e,\mathrm{y})$ of radius $\gg_{\mathrm{y}} R_G$. In particular, if $R_G, R_V \ll_{C,\mathrm{y}} 1$ then there is some $K(\mathrm{y})\geq 0$ such that $\nu_\mathrm{y}\left(x B(R_G,R_V)\right) \geq \meas_{\mathbf{G}(\mathbb{R})}(B_{G,\infty}(K(\mathrm{y})R_G))$ for all $x\in C$. This implies for $R_G, R_V \ll_{C,\mathrm{y}} 1$
\begin{align*}
\Cor\left(\mu,\nu_\mathrm{y} \right)[f_{0,R_G, R_V}]&\geq
\int_{C} \nu_\mathrm{y} \left(x\left(B_{G,\infty}(R_G)\times B_{V,\infty}(R_V)\right)\cdot \mathbf{P}^1(\hat{\mathbb{Z}})\right) \dif \mu(x)\\
&\geq \mu(C) \meas_{\mathbf{G}(\mathbb{R})}(B_{G,\infty}(K(\mathrm{y})R_G)) \, .
\end{align*}
This contradicts \eqref{eq:Cor(m,y)-asymp}.
\end{proof}
\subsection{Equidistribution of Genus Galois Orbits}
\begin{proof}[Proof of Theorem \ref{thm:genus-equidistribution}]
For an imaginary quadratic extension $E/\mathbb{Q}$ we need to compute
the preimage of $\Gal(E^\mathrm{ab}/\mathbb{Q}^\mathrm{ab})$ under the Artin reciprocity map $\lfaktor{E^\times}{\mathbb{A}_E^\times}\to \Gal(E^\mathrm{ab}/E)$. The argument uses the following commutative diagram of class field theory
\begin{center}
\begin{tikzcd}
\lfaktor{E^\times}{\mathbb{A}_E^\times} \arrow[r]\arrow[dd,"\Nr"] & \Gal(E^\mathrm{ab}/E)\arrow[d] \\
& \Gal(\mathbb{Q}^\mathrm{ab}/E)\arrow[d, hook] \\
\lfaktor{\mathbb{Q}^\times}{\mathbb{A}^\times} \arrow[r]&  \Gal(\mathbb{Q}^\mathrm{ab}/\mathbb{Q})
\end{tikzcd}
\end{center}
where the horizontal maps are reciprocity maps, the left vertical map is the norm map, the top right vertical map is restriction map and the bottom right vertical map is an inclusion. The Hasse norm theorem implies that $\lfaktor{E^{(1)}}{\mathbb{A}_E^{(1)}}$ is the kernel of the left vertical map while Galois theory says that $\Gal(E^\mathrm{ab}/\mathbb{Q}^\mathrm{ab})$ is the kernel of the right vertical map. As the diagram commutes the preimage
of $\Gal(E^\mathrm{ab}/\mathbb{Q}^\mathrm{ab})$ is $\lfaktor{E^{(1)}}{\mathbb{A}_E^{(1)}}$.

Combining this fact with the main theorem of complex multiplication for elliptic curves, cf.\ \cite[Theorem 5.4]{Shimura}, we see that a genus orbit of a special point $\Gal(\bar{\mathbb{Q}}/\mathbb{Q}^\mathrm{ab}).x_i$ is a $K_\infty$-invariant homogeneous toral set $\mathcal{H}_i\subset\left[\mathbf{P}^r(\mathbb{A})\right]$. Proposition \ref{prop:complex-torus-cm} implies that the endomorphism ring of a CM elliptic curve in the genus orbit is exactly the quadratic order from \S\ref{sec:order}. In particular, the strictness assumption from Definition \ref{defi:strict-galois} implies $|\disc(\mathcal{H}_i)|\to\infty$ and congruence condition \eqref{eq:congruence-1} follows from the congruence assumptions in Theorem \ref{thm:genus-equidistribution}. The second part of the strictness assumption of Definition \ref{defi:strict-homogeneous} follows from the strictness assumption \ref{defi:strict-galois} using the map $\jmath_{/\Lambda}$ from \S \ref{sec:torsion-pts} and the characterization of strictness using torsion order in Definition \ref{defi:strict-homogeneous}. The congruence condition \eqref{eq:congruence-2} follows as well from the congruence assumptions in Theorem \ref{thm:genus-equidistribution} using the map $\jmath_{/\Lambda}$.

Theorem \ref{thm:main} for the sequence $\left\{\mathcal{H}_i\right\}_i$ implies the claim.
\end{proof}
\begin{proof}[Proof of Theorem \ref{thm:genus-non-split}]
Torsion sections of $\mathcal{E}$ are quotient of homogeneous sets of the form $\left[\mathbf{G}(\mathbb{A})(e,\mathrm{y})\right]$ where $\mathrm{y}_\infty=0$ and $\mathrm{y}_f\in\mathbf{V}(\mathbb{A}_f)\cap \mathbf{V}(\mathbb{Q})$. The claim now
follows from Theorem \ref{thm:main2} by the argument in the proof of Theorem \ref{thm:genus-equidistribution}.
\end{proof}

\section{Joint Equidistribution of Points on Spheres and Orthogonal Grids}\label{sec:joint}
We begin by describing the orthogonal grid construction using an intersection of two periodic orbits in $\lfaktor{\mathbf{SL}_3(\mathbb{Z})}{\mathbf{SL}_3(\mathbb{R})}$. The benefit of this description is that it introduces the toral action naturally.

We carry the notations from the introduction \S \ref{intro:grids}. Recall that
\begin{equation*}
\mathcal{H}_D\coloneqq \left\{y\in\mathbb{Z}_\mathrm{primitive}^3 \mid \langle y, y\rangle=D \right\}\, .
\end{equation*}
For any $y\in\mathbb{Z}^3$ we denote by $y^\perp$ the affine variety over $\mathbb{Q}$ representing the $2$-dimensional rational linear subspace orthogonal to $y$. Moreover, $y^\perp(\mathbb{Z})\coloneqq \left\{z\in \mathbb{Z}^3 \mid \langle z, y\rangle=0 \right\}$.
For any $y\in\mathscr{H}_D$ we fix an integral point $y^1\in\mathbb{Z}^3$ such that $\langle y^1, y \rangle=1$. This point is not unique, but the coset $y^1+y^\perp(\mathbb{Z})$ is uniquely defined.

\begin{defi}
Denote by $\Theta\colon \mathbf{SL}_3\to\mathbf{SL}_3$ the Cartan involution $g\mapsto \tensor[^t]{g}{^{-1}}$, the fixed points of which are $\mathbf{SO}_3$. We use the notation $\tensor[^\Theta]{g}{}=\tensor[^t]{g}{^{-1}}$ for the action of $\Theta$ on $\mathbf{SL}_3$.
Fix $x\in\mathbb{Z}^3_\mathrm{primitive}$ with $\langle x,x\rangle =D\in\mathbb{N}$ and define
\begin{align*}
\mathbf{T}&\coloneqq \Stab_{\mathbf{SO}_3}(x) \, ,\\
\mathbf{H}&\coloneqq \Stab_{\mathbf{SL}_3}(x) \, ,\\
\tensor[^\Theta]{\mathbf{H}}{} &\coloneqq \Theta(\mathbf{H}) \, .
\end{align*}
Notice that $\Theta$ acts trivially on $\mathbf{T}$ and $\mathbf{T}<\mathbf{H}\cap \tensor[^\Theta]{\mathbf{H}}{}$. The group $\mathbf{T}$ is a rank-$1$ torus defined over $\mathbb{Q}$ and anisotropic over $\mathbb{R}$.
\end{defi}

We will study the following two periodic orbits and see how their intersections gives rise to the orthogonal grid construction.
\begin{equation*}
\left[\mathbf{SO}_3(\mathbb{R})\right],\left[\tensor[^\Theta]{\mathbf{H}}{}(\mathbb{R})\right]\subset \lfaktor{\mathbf{SL}_3(\mathbb{Z})}{\mathbf{SL}_3(\mathbb{R})} \, .
\end{equation*}

We begin by discussing the orbit of the group $\tensor[^\Theta]{\mathbf{H}}{}$ and its relation to lattice cosets in the orthogonal plane.

\subsection{The Periodic Orbit of the Special Affine Group}
\begin{prop}\label{prop:Htheta-ASL2}\hfill
\begin{enumerate}
\item
For any $y\in\mathbb{Q}^3$ the group $\tensor[^\Theta]{\mathbf{H}}{}$ stabilizes the affine plane $y+x^\perp=\langle y,x \rangle x+x^\perp$. If $\langle y,x\rangle\neq 0$ then any affine $\mathbb{Q}$-isomorphism of $y+x^\perp$ with the standard affine $2$-plane induces an isomorphism of $\tensor[^\Theta]{\mathbf{H}}{}$ with $\mathbf{ASL}_2$.

If $\langle y,x \rangle=0$ then the unipotent radical of $\tensor[^\Theta]{\mathbf{H}}{}$ acts trivially on $y+x^\perp=x^\perp$ and a linear $\mathbb{Q}$-isomorphism of $x^\perp$ with the affine $2$-plane induces and isomorphism of $\tensor[^\Theta]{\mathbf{H}}{}/\operatorname{\mathcal{R}_u}\left(\tensor[^\Theta]{\mathbf{H}}{}\right)$ with $\mathbf{SL}_2$

\item
Define
$$\tensor[^\Theta]{\mathbf{H}}{}(\mathbb{Z}) \coloneqq \mathbf{SL}_3(\mathbb{Z})\cap \tensor[^\Theta]{\mathbf{H}}{}(\mathbb{R})$$ then $\tensor[^\Theta]{\mathbf{H}}{}(\mathbb{Z})$ is a lattice in $\tensor[^\Theta]{\mathbf{H}}{}(\mathbb{R})$. In particular, the set
\begin{equation*}
\left[\tensor[^\Theta]{\mathbf{H}}{}(\mathbb{R})\right]\subset \lfaktor{\mathbf{SL}_3(\mathbb{Z})}{\mathbf{SL}_3(\mathbb{R})}
\end{equation*}
is a closed periodic orbit of the group $\tensor[^\Theta]{\mathbf{H}}{}(\mathbb{R})$.
\end{enumerate}
\end{prop}
\begin{proof}
Notice first that for all geometric points $y,z$ of the affine $3$-space if $\langle y-z,x\rangle=0$ then $\langle h.y-h.z,x\rangle=\langle y-z,\tensor[^t]{h}{}.x\rangle=\langle y-z,x\rangle=0$ for any point of $\tensor[^\Theta]{\mathbf{H}}{}$; hence the plane $y+x^\perp$ is $\tensor[^\Theta]{\mathbf{H}}{}$-stable. Assume $\langle y,x\rangle\neq 0$; we reduce to the case $y=\hat{e}$. There is $\gamma\in \mathbf{SL}_3(\mathbb{Q})$ such that $\gamma.\langle y,x\rangle x=\hat{e}$. This $\gamma$ then satisfies $\tensor[^\Theta]{\gamma}{}.x^\perp=\hat{e}^\perp$ and $\gamma.(y+x^\perp)=\hat{e}+\hat{e}^\perp$. We have $\gamma \mathbf{H} \gamma^{-1}=\Stab_{\mathbf{SL}_3} (\hat{e})$ and $\tensor[^\Theta]{\gamma}{} \tensor[^\Theta]{\mathbf{H}}{} \tensor[^\Theta]{\gamma}{^{-1}}=\tensor[^\Theta]{\Stab}{_{\mathbf{SL}_3}} (\hat{e})$. Conjugation by $\tensor[^\Theta]{\gamma}{}$ intertwines the action of $\tensor[^\Theta]{\Stab}{_{\mathbf{SL}_3}} (\hat{e})$ on $\hat{e}+\hat{e}^\perp$ with the action of $\tensor[^\Theta]{\mathbf{H}}{}$ on $y+x^\perp$.

We need now to show that a rational affine isomorphism of $\hat{e}+\hat{e}^\perp$ with the affine $2$-plane induces an isomorphism of $\tensor[^\Theta]{\Stab}{_{\mathbf{SL}_3}} (\hat{e})$ with $\mathbf{ASL}_2$. This is obvious when
writing $\tensor[^\Theta]{\Stab}{_{\mathbf{SL}_3}} (\hat{e})$ in matrix form
\begin{equation*}
\tensor[^\Theta]{\Stab}{_{\mathbf{SL}_3}} (\hat{e})=\left\{\begin{pmatrix}
* & * & *\\
* & * & * \\
0 & 0 & 1
\end{pmatrix}\right\} \, .
\end{equation*}
The case of $\langle y,x \rangle=0$ is proven similarly by taking any $\gamma\in\mathbf{SL}_2(\mathbb{Q})$ such that $\gamma.x=\hat{e}$.

For $y=x$ the affine isomorphism can be chosen to send the lattice $\mathbf{ASL}_2(\mathbb{Z})$ to $\tensor[^\Theta]{\Stab}{_{\mathbf{SL}_3}} (\hat{e})(\mathbb{Z})\coloneqq \mathbf{SL}_3(\mathbb{Z})\cap \tensor[^\Theta]{\Stab}{_{\mathbf{SL}_3}} (\hat{e})(\mathbb{R})$. Hence the latter is also a lattice and $\lfaktor{\tensor[^\Theta]{\Stab}{_{\mathbf{SL}_3}} (\hat{e})(\mathbb{Z})}{\tensor[^\Theta]{\Stab}{_{\mathbf{SL}_3}} (\hat{e})(\mathbb{R})}$
supports a finite Haar measure.

We can write
\begin{equation*}
\left[\tensor[^\Theta]{\mathbf{H}}{}(\mathbb{R})\right]=
\left[\tensor[^\Theta]{\gamma}{^{-1}}\tensor[^\Theta]{\Stab}{_{\mathbf{SL}_3}} (\hat{e})(\mathbb{R})\tensor[^\Theta]{\gamma}{}\right]
=
\left[\tensor[^\Theta]{\Stab}{_{\mathbf{SL}_3}} (\hat{e})(\mathbb{R})\tensor[^\Theta]{\gamma}{}\right]
\subset \lfaktor{\mathbf{SL}_3(\mathbb{Z})}{\mathbf{SL}_3(\mathbb{R})} \, .
\end{equation*}
This implies $\left[\tensor[^\Theta]{\mathbf{H}}{}(\mathbb{R})\right]
=\lfaktor{\tensor[^\Theta]{\mathbf{H}}{}(\mathbb{Z})}{\tensor[^\Theta]{\mathbf{H}}{}(\mathbb{R})}$ supports a finite $\tensor[^\Theta]{\gamma}{^{-1}}\tensor[^\Theta]{\Stab}{_{\mathbf{SL}_3}} (\hat{e})(\mathbb{R})\tensor[^\Theta]{\gamma}{}$-invariant measure. This is  a finite Haar measure on $\lfaktor{\tensor[^\Theta]{\mathbf{H}}{}(\mathbb{Z})}{\tensor[^\Theta]{\mathbf{H}}{}(\mathbb{R})}$ proving that the discrete subgroup $\tensor[^\Theta]{\mathbf{H}}{}(\mathbb{Z})$ is a lattice.
\end{proof}

\begin{lem} \label{lem:Htheta-integral}
For any $y,y'\in\mathscr{H}_D$ and $g\in\mathbf{SL}_3(\mathbb{R})$ if
\begin{equation*}
g.\left(y^1+y^\perp(\mathbb{Z})\right)=(y')^1+y'^\perp(\mathbb{Z})
\end{equation*}
then $g\in\mathbf{SL}_3(\mathbb{Z})$.
\end{lem}
\begin{proof}
Because $y^\perp(\mathbb{Z})=\left(y^1+y^\perp(\mathbb{Z})\right)-\left(y^1+y^\perp(\mathbb{Z})\right)$, and the same for $y'$, the group element $g$ also satisfies $g.y^\perp(\mathbb{Z})=y'^\perp(\mathbb{Z})$. Fix $v_1,v_2$ a basis for $y^\perp(\mathbb{Z})$ over $\mathbb{Z}$, then $y^1,v_1,v_2$ span $\mathbb{Z}^3$ over $\mathbb{Z}$. The vectors $g.v_1,g.v_2$ span $y'^\perp(\mathbb{Z})$ over $\mathbb{Z}$ and $g.y^1\in g.(y')^1+y'^\perp(\mathbb{Z})$. Hence $g.y^1,g.v_1,g.v_2$ also form a $\mathbb{Z}$-basis of $\mathbb{Z}^3$. This implies $g.\mathbb{Z}^3=\mathbb{Z}^3\Rightarrow g\in\mathbf{SL}_3(\mathbb{Z})$.
\end{proof}
\begin{lem} \label{lem:Htheta-Z-stab}
\begin{equation*}
\tensor[^\Theta]{\mathbf{H}}{}(\mathbb{Z})=
\Stab_{\tensor[^\Theta]{\mathbf{H}}{}(\mathbb{R})}\left(x^1+x^\perp(\mathbb{Z})\right) \, .
\end{equation*}
Moreover, for any $y\in\mathbb{Z}^3$
\begin{equation*}
\tensor[^\Theta]{\mathbf{H}}{}(\mathbb{Z})
\subset
\Stab_{\tensor[^\Theta]{\mathbf{H}}{}(\mathbb{R})}\left(y+x^\perp(\mathbb{Z})\right) \, .
\end{equation*}
\end{lem}
\begin{proof}
If $h\in \tensor[^\Theta]{\mathbf{H}}{}(\mathbb{Z})=\tensor[^\Theta]{\mathbf{H}}{}(\mathbb{R})\cap \mathbf{SL}_3(\mathbb{Z})$ then $h.\mathbb{Z}^3=\mathbb{Z}^3$ and $h.x^\perp=x^\perp$; hence $h.x^\perp(\mathbb{Z})=x^\perp(\mathbb{Z})$. Moreover, if $y\in\mathbb{Z}^3$ then $\langle h.y,x\rangle=\langle y, \tensor[^t]{h}{}.x \rangle=\langle y,x \rangle$ and $h.y\in\mathbb{Z}^3$ hence $h.y-y\in x^\perp(\mathbb{Z})$. This implies that $h.(y+x^\perp(\mathbb{Z}))=y+x^\perp(\mathbb{Z})$. Hence $\tensor[^\Theta]{\mathbf{H}}{}(\mathbb{Z})
\subset
\Stab_{\tensor[^\Theta]{\mathbf{H}}{}(\mathbb{R})}\left(y+x^\perp(\mathbb{Z})\right)$.

The reverse inclusion for $y=x_1$ follows from Lemma \ref{lem:Htheta-integral} with $y=y'=x$.
\end{proof}

\begin{defi}\label{defi:Htheta-orbit-to-Grids}
\hfill
\begin{enumerate}
\item
For any $y,z\in\mathbb{Q}^3$ and $d>0$ denote by $\mathcal{L}^d_\bullet (y+z^\perp)$ the space of cosets of rank $2$ $\mathbb{Z}$-lattices of covolume $\sqrt{d}$ contained in the affine space $y+z^\perp(\mathbb{R})$.

\item
Proposition \ref{prop:Htheta-ASL2} implies that the group $\tensor[^\Theta]{\mathbf{H}}{}(\mathbb{R})$ acts transitively on the space $\mathcal{L}^D_\bullet (x^1+x^\perp)$. Fixing $x_1+x^\perp(\mathbb{Z})\in\mathcal{L}^D_\bullet (x^1+x^\perp)$ as a base point we use Lemma \ref{lem:Htheta-Z-stab} to construct an $\tensor[^\Theta]{\mathbf{H}}{}(\mathbb{R})$-equivariant isomorphism
\begin{equation*}
\lfaktor{\tensor[^\Theta]{\mathbf{H}}{}(\mathbb{Z})}{\tensor[^\Theta]{\mathbf{H}}{}(\mathbb{R})}
\to \mathcal{L}^D_\bullet (x^1+x^\perp) \, .
\end{equation*}

\item The map $P+L\mapsto (P-\frac{x}{D})+L$ is a bijection $\mathcal{L}^D_\bullet (x^1+x^\perp)\to \mathcal{L}^D_\bullet (x^\perp)$. Using this bijection defined a sequence $\tensor[^\Theta]{\mathbf{H}}{}(\mathbb{R})$-equivariant isomorphisms
\begin{equation*}
\lfaktor{\tensor[^\Theta]{\mathbf{H}}{}(\mathbb{Z})}{\tensor[^\Theta]{\mathbf{H}}{}(\mathbb{R})}
\to \mathcal{L}^D_\bullet (x^1+x^\perp)\to \mathcal{L}^D_\bullet (x^\perp) \, ,
\end{equation*}
where the action of $\tensor[^\Theta]{\mathbf{H}}{}(\mathbb{R})$ on $\mathcal{L}^D_\bullet (x^\perp)$ is the twisted action defined by $h.(P+L)=(h.\frac{x}{D}-\frac{x}{D})+h.P+h.L$. In what follows we will not be using the twisted action.
\end{enumerate}
\end{defi}

\subsection{The Intersection of Periodic Orbits}
We turn to discuss the intersection between $\left[\tensor[^\Theta]{\mathbf{H}}{}(\mathbb{R})\right]$ and $\left[\mathbf{SO}_3(\mathbb{R})\right]$.
\begin{prop}\label{prop:Htheta-SO3-intersect}
\begin{equation*}
\lfaktor{\mathbf{SL}_3(\mathbb{Z})}{\mathbf{SL}_3(\mathbb{R})}
\supset
\left[\tensor[^\Theta]{\mathbf{H}}{}(\mathbb{R})\right] \cap \left[\mathbf{SO}_3(\mathbb{R})\right]
=\bigsqcup_{y\in \mathbf{SO}_3(\mathbb{Z})\backslash\mathscr{H}_D} \left[g_{x\to y}\mathbf{T}(\mathbb{R})\right] \, ,
\end{equation*}
where for each $y\in\mathscr{H}_D$ the coset $g_{x\to y}\mathbf{T}(\mathbb{R})$ is the set of all elements in $\mathbf{SO}_3(\mathbb{R})$ mapping $x$ to $y$.
\end{prop}
\begin{proof}
Notice first that the union on the right hand side is indeed disjoint. Two $\mathbf{T}(\mathbb{R})$ orbit $\left[g_{x\to y}\mathbf{T}(\mathbb{R})\right]$ and $\left[g_{x\to y'}\mathbf{T}(\mathbb{R})\right]$ intersect if there is $\gamma\in\mathbf{SL}_3(\mathbb{Z})$ and $t,t'\in\mathbf{T}(\mathbb{R})$ such that
\begin{equation*}
\gamma= g_{x\to y}tt'^{-1}g_{x\to y'}^{-1}\in\mathbf{SO}_3(\mathbb{R}) \, .
\end{equation*}
Then $\gamma\in\mathbf{SO}_3(\mathbb{Z})$ and $\gamma.y'=y$, hence $\mathbf{SO}_3(\mathbb{Z})y=\mathbf{SO}_3(\mathbb{Z})y'$.

We now establish that the right hand side is contained in the left hand side. Fix $y\in \mathscr{H}_D$. It is obvious that $\left[g_{x\to y}\mathbf{T}(\mathbb{R})\right]\subset \left[\mathbf{SO}_3(\mathbb{R})\right]$ because $g_{x\to y}\mathbf{T}(\mathbb{R})\subset \mathbf{SO}_3(\mathbb{R})$. Fix $\gamma\in\mathbf{SL}_3(\mathbb{Z})$ such that $\tensor[^\Theta]{\gamma}{}.x=y$, this is possible because $x,y\in\mathbb{Z}^3_\mathrm{primitive}$ and $\tensor[^\Theta]{\mathbf{SL}}{_3}(\mathbb{Z})=\mathbf{SL}_3(\mathbb{Z})$. For any $g\in g_{x\to y}\mathbf{T}(\mathbb{R})$ note that $\tensor[^\Theta]{\left(g^{-1}\gamma\right)}{}.x=g^{-1}\tensor[^\Theta]{\gamma}{}.x=x$. Thus $\tensor[^\Theta]{\left(g^{-1}\gamma\right)}{}\in\mathbf{H}(\mathbb{R})$ and $g^{-1}\gamma\in \tensor[^\Theta]{\mathbf{H}}{}(\mathbb{R})$. We conclude that $g\in\left[\tensor[^\Theta]{\mathbf{H}}{}(\mathbb{R})\right]$ proving that the right hand side is included in the left one.

In the other direction we argue as follows.
For any point $\mathbf{SL}_3(\mathbb{Z})g\in \left[\mathbf{SO}_3(\mathbb{R})\right]$ that belongs also to $\left[\tensor[^\Theta]{\mathbf{H}}{}(\mathbb{R})\right]$ there is some $\gamma\in\mathbf{SL}_3(\mathbb{Z})$ and $h\in \tensor[^\Theta]{\mathbf{H}}{}(\mathbb{R})$ such that
$\gamma h =g$. In particular $g.x=\tensor[^\Theta]{g}{}.x= \tensor[^\Theta]{\gamma}{}.x$. Notice that $\tensor[^\Theta]{\gamma}{}.x\in\mathbb{Z}^3_\mathrm{primitive}$ because $x\in\mathbb{Z}^3_\mathrm{primitive}$ and $\tensor[^\Theta]{\gamma}{}\in \mathbf{SL}_3(\mathbb{Z})$. On the other hand $\langle g.x, g.x \rangle =D$, hence $g.x=\tensor[^\Theta]{\gamma}{}.x\in \mathscr{H}_D$.

Any other intersection point $\mathbf{SL}_3(\mathbb{Z})g'$, $g'\in\mathbf{SO}_3(\mathbb{R})$, such that $g'.x=g.x$ satisfies $g'\in \mathbf{T}(\mathbb{R})g$.
\end{proof}
\begin{remark}
A slightly more conceptual presentation of the proof is by using the fact that the ring of regular functions on $\mathbf{SL}_3$ invariant under the right action of $\mathbf{H}$ and the left action of $\mathbf{SO}_3$ is generated by the single polynomial $g\mapsto \langle g.x, g.x \rangle$.
\end{remark}

\begin{prop}\label{prop:intersection-grids}
This image of $\left[\tensor[^\Theta]{\mathbf{H}}{}(\mathbb{R})\right] \cap \left[\mathbf{SO}_3(\mathbb{R})\right]$ under the isomorphism from Definition \ref{defi:Htheta-orbit-to-Grids}
\begin{equation*}
\lfaktor{\tensor[^\Theta]{\mathbf{H}}{}(\mathbb{Z})}{\tensor[^\Theta]{\mathbf{H}}{}(\mathbb{R})}\to \mathcal{L}^D_\bullet(x^\perp)
\end{equation*}
is
\begin{equation*}
\bigsqcup_{y\in \mathbf{SO}_3(\mathbb{Z})\backslash\mathscr{H}_D} \mathbf{T}(\mathbb{R})g_{x\to y}^{-1}.\left(y^\mathrm{tors}+y^\perp(\mathbb{Z})\right)
\, .
\end{equation*}
where $y^\mathrm{tors}=y^1-\frac{y}{D}$ as in \S \ref{intro:grids} and the action is the standard action of $\mathbf{SO}_3(\mathbb{R})$ on $\mathbb{R}^3$.

Equivalently, the image of the intersection under the isomorphism
\begin{equation*}
\lfaktor{\tensor[^\Theta]{\mathbf{H}}{}(\mathbb{Z})}{\tensor[^\Theta]{\mathbf{H}}{}(\mathbb{R})}\to \mathcal{L}^D_\bullet(x^1+x^\perp)
\end{equation*}
is
\begin{equation}\label{eq:orthogonal-lattice-intersection}
\bigsqcup_{y\in \mathbf{SO}_3(\mathbb{Z})\backslash\mathscr{H}_D} \mathbf{T}(\mathbb{R})g_{x\to y}^{-1}.\left(y^1+y^\perp(\mathbb{Z})\right) \, .
\end{equation}
\end{prop}
\begin{proof}
We prove the second version of the claim.
That the union in \eqref{eq:orthogonal-lattice-intersection} is disjoint follows by a computation from Lemma \ref{lem:Htheta-integral}.

Notice that for every $y\in\mathscr{H}_D$ the orbit $\mathbf{T}(\mathbb{R})g_{x\to y}^{-1}.\left(y^1+y^\perp(\mathbb{Z})\right)$ is indeed contained in $ \mathcal{L}^D_\bullet(x^1+x^\perp)$ because
\begin{equation*}
\left\langle \mathbf{T}(\mathbb{R})g_{x\to y}^{-1}.\left(y^1+y^\perp(\mathbb{Z})\right), x\right\rangle
=\left\langle y^1+y^\perp(\mathbb{Z}), g_{x\to y}\mathbf{T}(\mathbb{R}).x\right\rangle
=\left\langle y^1+y^\perp(\mathbb{Z}), y\right\rangle=1 \, .
\end{equation*}
Lemma \ref{lem:Htheta-integral} can be used to show that the orbit $\mathbf{T}(\mathbb{R})g_{x\to y}^{-1}.\left(y^1+y^\perp(\mathbb{Z})\right)$ is contained in the image of
\begin{equation*}
\left[\tensor[^\Theta]{\mathbf{H}}{}(\mathbb{R})\right] \cap \left[\mathbf{SO}_3(\mathbb{R})\right]\hookrightarrow \lfaktor{\tensor[^\Theta]{\mathbf{H}}{}(\mathbb{Z})}{\tensor[^\Theta]{\mathbf{H}}{}(\mathbb{R})} \, .
\end{equation*}
We have shown that \eqref{eq:orthogonal-lattice-intersection} is contained in the intersection. To prove that it is the whole intersection we notice from Proposition \ref{prop:Htheta-SO3-intersect} that the intersection is a collection of $\left|\mathbf{SO}_3(\mathbb{Z})\backslash\mathscr{H}_D\right|$-disjoint orbits of $\mathbf{T}(\mathbb{R})$ and the same holds for \eqref{eq:orthogonal-lattice-intersection}, hence they are equal.
\end{proof}

\begin{defi}
Let $g_{x\to\hat{e}}\in\mathbf{SO}_3(\mathbb{R})$ satisfy $g_{x\to\hat{e}}.x=D^{1/2}\hat{e}$.
Any element in $g\in\mathbf{SO}_2(\mathbb{R})g_{x\to\hat{e}}$ defines a bijection
\begin{equation*}
\mathcal{L}^D_\bullet(x^\perp)\xrightarrow{g} \mathcal{L}^1_\bullet(\hat{e}^\perp)
\end{equation*}
by $L+P\mapsto D^{-1/4}g.\left(L+P\right)$. This map depends on the specific representative $g\in \mathbf{SO}_2(\mathbb{R})g_{x\to\hat{e}}$ but the induced bijection of quotients
\begin{equation*}
\lfaktor{\mathbf{T}(\mathbb{R})}{\mathcal{L}^D_\bullet(x^\perp)}\xrightarrow{g} \lfaktor{\mathbf{SO}_2(\mathbb{R})}{\mathcal{L}^1_\bullet(\hat{e}^\perp)}
\end{equation*}
is uniquely-defined.

We denote by $\mathscr{G}_D\subset \mathcal{L}^1_\bullet(\hat{e}^\perp)$ the image of the intersection
$\left[\tensor[^\Theta]{\mathbf{H}}{}(\mathbb{R})\right] \cap \left[\mathbf{SO}_3(\mathbb{R})\right]$ under the composite map
\begin{equation*}
\left[\tensor[^\Theta]{\mathbf{H}}{}(\mathbb{R})\right]\to \mathcal{L}^D_\bullet(x^1+x^\perp)
\to \mathcal{L}^D_\bullet(x^\perp)\to \mathcal{L}^1_\bullet(\hat{e}^\perp) \, .
\end{equation*}
It is a finite collection of $\mathbf{SO}_2(\mathbb{R})$ orbits of lattice cosets in $\mathcal{L}^1_\bullet(\hat{e}^\perp)$. Each  $\mathbf{SO}_2(\mathbb{R})$-orbit in $\mathscr{G}_D$ is equal to $\operatorname{Grid}(y)$ for some $y\in\mathscr{H}_D$.
\end{defi}

The picture emerging till far is rather elegant. The image of the intersection $\left[\tensor[^\Theta]{\mathbf{H}}{}(\mathbb{R})\right] \cap \left[\mathbf{SO}_3(\mathbb{R})\right]$ in $\lfaktor{\mathbf{SO}_3(\mathbb{Z})}{\mathbf{S}^2(\mathbb{R})}\simeq\dfaktor{\mathbf{SO}_3(\mathbb{Z})}{\mathbf{SO}_3(\mathbb{R})}{\mathbf{SO}_2 (\mathbb{R})}$ is $D^{-1/2}\mathbf{SO}^3(\mathbb{Z})\backslash\mathscr{H}_D$. Each point in $\mathscr{H}_D$ we can lift to a $\mathbf{T}(\mathbb{R})$ orbit in the intersection, the image of this orbit in $\lfaktor{\mathbf{SO}_2(\mathbb{R})}{\mathcal{L}^1_\bullet(\hat{e}^\perp)}$ is exactly $\operatorname{Grid}(x)$ from \S \ref{intro:grids}.

\subsection{The Ad\`elic Description of the Intersection}
The real advantage of the intersection picture is that the joint action of the Picard group, equivalently the ad\`elic torus $\mathbf{T}(\mathbb{A})$, on the correspondence between lattice points on the sphere and their orthogonal grids is evident from the next proposition. In particular, it establishes that $\Phi\left(\left[\tensor[^\Theta]{\mathbf{H}}{}(\mathbb{R})\right] \cap \left[\mathbf{SO}_3(\mathbb{R})\right]\right)g_\infty$ is the projection of a $K_\infty$-invariant homogeneous toral set.
\begin{prop}\label{prop:adelic-intersection}
The image of
$\left[\mathbf{T}(\mathbb{A})\right]\subset \lfaktor{\mathbf{SL}_3(\mathbb{Q})}{\mathbf{SL}_3(\mathbb{A})}$ under the quotient map
\begin{equation*}
\lfaktor{\mathbf{SL}_3(\mathbb{Q})}{\mathbf{SL}_3(\mathbb{A})}
\xrightarrow{\slash\mathbf{SL}_3(\hat{\mathbb{Z}})}
\lfaktor{\mathbf{SL}_3(\mathbb{Z})}{\mathbf{SL}_3(\mathbb{R})}
\end{equation*}
is $\left[\tensor[^\Theta]{\mathbf{H}}{}(\mathbb{R})\right] \cap \left[\mathbf{SO}_3(\mathbb{R})\right]$.
\end{prop}
This proposition is the well-known statement that the primitive integral points on the $2$-sphere of radius $\sqrt{D}$ form a single toral packet, cf.\ \cite{EMV} for the case of $D$ square-free.  For the proof we will use the Hamilton quaternions and a few preparatory lemmata. This proof is different from the one in \cite{EMV}; I have preferred arguments which may generalize better to higher rank torus orbits.

\begin{defi}\hfill
\begin{enumerate}
\item
Denote by $\mathbf{B}$ the Hamilton quaternion algebra $\mathbf{B}$ defined over $\mathbb{Q}$. Denote by $\mathbf{Z}\simeq \Gm$ the center of $\mathbf{B}^\times$. The group $\mathbf{PB}^\times\coloneqq \mathbf{Z}\backslash \mathbf{B}^\times$ acts faithfully by conjugation on the traceless quaternions $\mathbf{B}^0$. The group $\mathbf{PB}^\times$ is exactly the group of linear automorphisms of  $\mathbf{B}^0$ preserving the reduced norm and its polarization which is proportional to the trace form.

The quaternion algebra $\mathbf{B}$ is ramified exactly at $\infty$ and $2$. For $p>2$ there is a group isomorphism $\mathbf{PB}^\times(\mathbb{Q}_p) \simeq \mathbf{PGL}_2(\mathbb{Q}_p)$ where the action of $\mathbf{PB}^\times(\mathbb{Q}_p)$ on $\mathbf{B}^0(\mathbb{Q}_p)$ is intertwined with the adjoint action\footnote{For this we identify the trace zero $2 \times 2$ matrices with the Lie algebra of $\mathbf{PGL}_2$.} of $\mathbf{PGL}_2(\mathbb{Q}_p)$ on $\mathfrak{pgl}_2(\mathbb{Q}_p)$.
\item
Fix a $\mathbb{Q}$-linear isomorphism of the traceless quaternions with the $3$-dimensional Euclidean space sending the quaternion norm squared to the Euclidean norm squared, we henceforth identify these two spaces. This induces an isomorphism $\mathbf{PB}^\times\simeq \mathbf{SO}_3$ over $\mathbb{Q}$ and a closed embedding $\mathbf{PB}^\times\hookrightarrow \mathbf{SL}_3$. We identify henceforth  $\mathbf{PB}^\times$ with $\mathbf{SO}_3$ and consider it as a closed subgroup of $\mathbf{SL}_3$.
\item
For each prime $p$ define $K_{\mathbf{B},p}=\mathbf{PB}^\times(\mathbb{Q}_p)\cap \mathbf{SL}_3(\mathbb{Z}_p)$. Using the identification of $\mathbf{PB}^\times$ and $\mathbf{SO}_3$ the group $K_{\mathbf{B},p}$ is identified with $\mathbf{SO}_3(\mathbb{Z}_p)$.
\item
Let $\mathbf{B}^{0,D}$ be the affine variety of quaternions of norm $D$ and trace $0$. It is a homogeneous space for the group $\mathbf{PB}^\times$. Denote the stabilizer of $x\in\mathbf{B}^{0,D}(\mathbb{Q})$ by $\mathbf{PB}^\times_x\simeq \mathbf{T}$. Then $\mathbf{B}^{0,D}\simeq\faktor {\mathbf{PB}^\times} {\mathbf{PB}^\times_x}$.
\end{enumerate}
\end{defi}

\begin{lem}\label{lem:SO3(Q)-transitive}
For any $y\in\mathscr{H}_D$ there is $g\in\mathbf{SO}_3(\mathbb{Q})$ such that $g.x=y$.
\end{lem}
We present two proofs. The first uses quadratic spaces and the second Galois cohomology of tori.
\begin{proof}[First Proof]
Consider $x^\perp(\mathbb{Q})$ and $y^\perp(\mathbb{Q})$ as rational quadratic space equipped with the restriction of the Euclidean inner product $\langle,\rangle$. The restriction of the norm $\|\bullet \|^2$ to $x^\perp(\mathbb{Z})$ and $y^\perp(\mathbb{Z})$ is in both cases a primitive integral binary quadratic form of discriminant $-D$ or $-4D$, cf. \cite[\S 4.1.2]{AES}. Hence it is a norm form of a lattice in the quadratic field $\mathbb{Q}(\sqrt{D})$. This shows that both quadratic spaces $\left(x^\perp(\mathbb{Q}),\langle,\rangle\right)$ and $\left(y^\perp(\mathbb{Q}),\langle,\rangle\right)$ are isometric over $\mathbb{Q}$ to $\mathbb{Q}(\sqrt{D})$, equipped with the trace form. By Witt's extension theorem there is a rational isometry of $\mathbb{Q}^3$ sending $x^\perp(\mathbb{Q})$ to $y^\perp(\mathbb{Q})$. We have thus constructed an element in $g\in\mathbf{O}_3(\mathbb{Q})$ such that $g.x^\perp=y^\perp$, hence $g.x=\pm y$. By post-composing $g$ with the reflection through the vector $y$ we can assume $g.x=y$. If $\det g=-1$ then we post-compose it with an element $h\in\mathbf{O}_3(\mathbb{R})$ with $\det h=-1$ and $h.y=y$. Such an element $h$ can be constructed by first constructing an orientation inverting rational isometry of $y^\perp(\mathbb{Q})$ (reflection through a vector), extending it using Witt's extension theorem to an isometry $h$ of $\mathbb{Q}^3$ and if $h.y=-y$ then we compose it with the reflection through $y$.
\end{proof}
\begin{proof}[Second Proof]
The points $x$ and $y$ are in the same $\mathbf{SO}_3(\mathbb{Q})$-orbit if the kernel of the following map of pointed Galois cohomologies is trivial
\begin{equation*}
\ker\left[H^1(\mathbb{Q},\mathbf{T})\to H^1(\mathbb{Q},\mathbf{SO}_3)\right]=1 \, .
\end{equation*}
To prove that this kernel is trivial we use the Hamilton quaternion algebra $\mathbf{B}$. We need to show that $\ker\left[H^1(\mathbb{Q},\mathbf{PB}^\times_x)\to H^1(\mathbb{Q},\mathbf{PB}^\times)\right]=1$. We consider the non-faithful action of $\mathbf{B}^\times$ on $\mathbf{B}^0$ by conjugation and denote the stabilizer of $x$ by $\mathbf{B}^\times_x$.

We have the following commutative diagram with exact rows.
\begin{center}
\begin{tikzcd}
1 \arrow[r] & \mathbf{B}^{\times}_x \arrow[r]\arrow[d] & \mathbf{B}^{\times} \arrow[r]\arrow[d] & \mathbf{B}^{0,D} \arrow[r]\arrow[d] & 1 \\
1 \arrow[r] & \mathbf{PB}^{\times}_x \arrow[r] & \mathbf{PB}^{\times} \arrow[r]& \mathbf{B}^{0,D} \arrow[r] & 1 \\
\end{tikzcd}
\end{center}
It induces a commutative diagram of pointed Galois cohomology sets with exact rows
\begin{center}
\begin{tikzcd}
\cdots\arrow[r] & \mathbf{B}^{0,D}(\mathbb{Q}) \arrow[r]\arrow[d] & H^1(\mathbb{Q},\mathbf{B}^{\times}_x) \arrow[r]\arrow[d] & H^1(\mathbb{Q},\mathbf{B}^{\times}) \arrow[d]\\
\cdots\arrow[r]&\mathbf{B}^{0,D}(\mathbb{Q}) \arrow[r] & H^1(\mathbb{Q},\mathbf{PB}^{\times}_x) \arrow[r] & H^1(\mathbb{Q},\mathbf{PB}^{\times}) \\
\end{tikzcd}
\end{center}
Because the leftmost vertical map is the identity the vanishing of $$\ker\left[H^1(\mathbb{Q},\mathbf{PB}^{\times}_x) \to H^1(\mathbb{Q},\mathbf{PB}^{\times})\right]$$ is a consequence of $H^1(\mathbb{Q},\mathbf{B}^{\times}_x)=1$. This latter equality holds because $\mathbf{B}^\times_x\simeq \WR_{\mathbb{Q}(\sqrt{-D})/\mathbb{Q}} \Gm$ and the Galois cohomology of this torus vanishes as it is a quasi-split torus (its character group is a permutation module for the Galois group).
\end{proof}

\begin{lem}\label{lem:KBp-structure}
For all primes $p$ the group $K_{\mathbf{B},p}$ is a maximal compact-open subgroup of $\mathbf{PB}^\times(\mathbb{Q}_p)$. In particular, $K_{\mathbf{B},2}=\mathbf{PB}^\times(\mathbb{Q}_2)$.
\end{lem}
\begin{proof}
Denote by $\mathcal{O}\subset\mathbf{B}(\mathbb{Q})$ the Hurwitz quaternions -- all quaternions such that either all coordinates are integral or all coordinates are half-integral. This is a maximal order in $\mathbf{B}(\mathbb{Q})$ hence for any prime $p$ its $p$-adic completion $\mathcal{O}_p$ is a maximal order in $\mathbf{B}(\mathbb{Q}_p)$. In coordinate form
\begin{equation*}
\mathcal{O}_p=\begin{cases}
\left\{a+bi+cj+dk \mid a,b,c,d\in\mathbb{Z}_2 \right\}\sqcup \frac{1}{2} \left\{a+bi+cj+dk \mid a,b,c,d\in\mathbb{Z}_2^\times \right\}
 & p =2\\
 \left\{a+bi+cj+dk \mid a,b,c,d\in\mathbb{Z}_p \right\} & p>2
\end{cases} \, .
\end{equation*}
Denote $\mathcal{O}_p^0\coloneqq\mathbf{B}^0(\mathbb{Q}_p)\cap \mathcal{O}_p$, for all $p$ we see that $\mathcal{O}_p^0=\mathbb{Z}_p^3$.

For $p=2$ there is a unique maximal order invariant under conjugation. The claim for $p=2$ follows because the trace is also invariant for the conjugation action of $\mathbf{PB}^\times(\mathbb{Q}_2)$.

Assume $p>2$. Any element $M\in\mathbf{M}_{2\times 2}(\mathbb{Q}_p)$ can be written as $M=M-\frac{1}{2}\Tr M +\frac{1}{2}\Tr M$. Because $2$ is invertible in $\mathbb{Z}_p$ we see that any maximal order $\mathcal{O}'\subset\mathbf{M}_{2\times 2}(\mathbb{Q}_p)$ satisfies  $\mathcal{O}'={\mathcal{O}'}^0+\mathbb{Z}_p$; where ${\mathcal{O}'}^0$ is the set of traceless elements of $\mathcal{O}'$. Thus $\Stab_{\mathbf{PGL}_2}(\mathcal{O}')=\Stab_{\mathbf{PGL}_2}({\mathcal{O}'}^0)$. Recall that maximal orders in $\mathbf{M}_{2\times 2}(\mathbb{Q}_p)$ are in bijection with the vertices of the Bruhat-Tits tree of $\mathbf{PGL}_2(\mathbb{Q}_p)$ and the conjugation action of $\mathbf{PGL}_2(\mathbb{Q}_p)$ on maximal orders corresponds to the action of $\mathbf{PGL}_2(\mathbb{Q}_p)$ on the Bruhat-Tits tree. Hence $\Stab_{\mathbf{PGL}_2}({\mathcal{O}'}^0)$ is the stabilizer of a vertex in the Bruhat-Tits tree which is a maximal compact-open subgroup.

The group $K_{\mathbf{B},p}$ is by definition the stabilizer in $\mathbf{PB}^\times(\mathbb{Q}_p)$ of $\mathbb{Z}_p^3=\mathcal{O}_p^0\subseteq \mathbf{B}(\mathbb{Q}_p)$.
Because $\mathbf{B}$ is unramified at $p>2$ there is an isomorphism of central simple algebras $\mathbf{B}(\mathbb{Q}_p)$ with $\mathbf{M}_{2\times 2}(\mathbb{Q}_p)$. This isomorphism sends $\mathcal{O}_p$ to some maximal order $\mathcal{O}'\subset \mathbf{M}_{2\times 2}(\mathbb{Q}_p)$ and $K_{\mathbf{B},p}$ to  $\Stab_{\mathbf{PGL}_2}({\mathcal{O}'}^0)$ which is a maximal compact-open subgroup.
\end{proof}

\begin{lem}\label{lem:SO3-local-transitivity}
For all primes $p$ there is an element $k\in K_{\mathbf{B},p}=\mathbf{SO}_3(\mathbb{Z}_p)$ satisfying $k.x=y$.
\end{lem}
\begin{proof}
For $p=2$ this is immediate from Lemmata \ref{lem:KBp-structure} and \ref{lem:SO3(Q)-transitive}. Assume $p>2$ and fix an isomorphism $\mathbf{B}(\mathbb{Q}_p)\simeq \mathbf{M}_{2\times 2}(\mathbb{Q}_p)$ sending the Hurwitz maximal order $\mathcal{O}_p$ from the proof of Lemma \ref{lem:KBp-structure} to $\mathbf{M}_2(\mathbb{Z}_p)$. Use this isomorphism to identify the former spaces and identify $\mathbf{PB}^\times(\mathbb{Q}_p)$ with $\mathbf{PGL}_2(\mathbb{Q}_p)$. Then $K_{\mathbf{B},p}=\mathbf{PGL}_2(\mathbb{Z}_p)$ and $x$, $y$ are two points in $\mathcal{O}_p$ with trace $0$. The assumption that $x$ and $y$ are primitive implies that $x,y\not\in p \mathbf{M}_2(\mathbb{Z}_p)$.

We abuse the notation and define tentatively $\mathbf{PGL}_2$ as an affine scheme over $\mathbb{Z}_p$ using the adjoint representation.
Denote by $\Lambda_{y\to x}$ the closed affine subscheme of $\mathbf{PGL}_2$ of group elements $g$ such that $gyg^{-1}=x$. This scheme can be evidently defined over $\mathbb{Z}_p$ and
\begin{equation*}
\Lambda_{y\to x}(\mathbb{Z}_p)=\left\{k\in K_{\mathbf{B},p} \mid k^{-1}.x=y \right\} \, .
\end{equation*}
We will show next that the reduction of $\Lambda_{y\to x}$ modulo $p$ is a smooth variety with a point over $\mathbb{F}_p$.  Hensel's lemma then implies that there is a point in $\Lambda_{y\to x}(\mathbb{Z}_p)$ finishing the proof of the claim.

The scheme $\Lambda_{x\to y}$ is a torsor for the stabilizer $\mathbf{T}$ of $x$ which is a torus in $\mathbf{PGL}_2$. We denote the reduction modulo $p$ by an over-line.
The reduction of $\overline{\Lambda_{x\to y}}$ is a torsor for the reduction $\overline{\mathbf{T}}=\Stab_{\overline{\mathbf{PGL}_2}} \overline{x}$. Hence if $\overline{\Lambda_{x\to y}}(\mathbb{F}_p)$ is non-empty then $\overline{\Lambda_{x\to y}}$ is isomorphic to $\overline{\mathbf{T}}$ over $\mathbb{F}_p$. We need then to show that $\overline{x}$ and $\overline{y}$ are conjugate over $\mathbb{F}_p$ and that $\overline{\mathbf{T}}$ is smooth.

The element $x$, $y$ have the same norm and trace so the characteristic polynomial of $\overline{x}$ and $\overline{y}$ is the same. We distinguish between two cases. If $p\nmid D$ this is the case of \emph{multiplicative reduction} and if $p\mid D$ this is the case of \emph{additive reduction}.

If $p\mid D$ then $\det \overline{x}=\det\overline{y}=0$ but $\overline{x},\overline{y}\neq 0$ because $x,y\in\mathbb{Z}^3_\mathrm{primitive} \Rightarrow x,y\not\in p \mathbf{M}_2(\mathbb{Z}_p)$. If $p\nmid D$ then $\overline{x},\overline{y}$ are both regular semisimple elements in $\mathbf{GL}_2(\mathbb{F}_p)$. The centralizer $\overline{\mathbf{T}}$ of a regular semisimple element is a maximal torus -- hence smooth.

In both cases using the Jordan normal form over the algebraic closure $\overline{\mathbb{F}_p}$ we see that  $\overline{x},\overline{y}$ are conjugate over $\overline{\mathbb{F}_p}$. Assume for the moment $p\mid D$. There is a single non-trivial nilpotent conjugacy class over $\overline{\mathbb{F}_p}$ -- the class of $N\coloneqq\begin{pmatrix}
0 & 1 \\ 0 & 0
\end{pmatrix}\in \mathbb{F}_p$. The stabilizer $\Stab_{\overline{\mathbf{PGL}_2}} N$ is isomorphic to $\Ga$ as it is the group of upper-triangular unipotent matrices. This is also a smooth group. To show that there is a single non-trivial nilpotent conjugacy class over $\mathbb{F}_p$ we need to establish
$\ker\left[H^1(\mathbb{F}_p, \Stab_{\overline{\mathbf{PGL}_2}} N)\to H^1(\mathbb{F}_p, \overline{\mathbf{PGL}_2})\right]=1$. This follows from Lang's theorem as  $\Stab_{\overline{\mathbf{PGL}_2}} N\simeq \Ga$ is smooth and connected and $H^1(\mathbb{F}_p, \Stab_{\overline{\mathbf{PGL}_2}} N)=1$. This also implies for $p\mid D$ that $\overline{\mathbf{T}}$ is conjugate to $\Stab_{\overline{\mathbf{PGL}_2}} N$ over $\mathbb{F}_p$ and is smooth.

The claim that $\overline{x},\overline{y}$ are conjugate over $\mathbb{F}_p$ if $p\nmid D$ follows similarly as Lang's theorem implies the vanishing $H^1(\mathbb{F}_p, \overline{\mathbf{T}})$.
\end{proof}

\begin{proof}[Proof of Proposition \ref{prop:adelic-intersection}]
We first show the inclusion of the image of $\left[\mathbf{T}(\mathbb{A})\right]$ in the intersection of the real periodic orbits.
Because $\mathbf{SO}_3(\mathbb{A}_f)\cap \mathbf{SL}_3(\hat{\mathbb{Z}})=\mathbf{SO}_3(\hat{\mathbb{Z}})$ and the quadratic form $x^2+y^2+z^2$ has class number $1$ (it is the unique form in its genus) the image of $\left[\mathbf{SO}_3(\mathbb{A})\right]$ under the quotient map is exactly $\left[\mathbf{SO}_3(\mathbb{R})\right]$. The group $\tensor[^\Theta]{\mathbf{H}}{}
\simeq\mathbf{ASL}_2=\mathbf{P}^1$ has strong approximation, hence the image of $\left[\tensor[^\Theta]{\mathbf{H}}{}(\mathbb{A})\right]$ under the quotient map is $\left[\tensor[^\Theta]{\mathbf{H}}{}(\mathbb{R})\right]$. Obviously we have $\left[\mathbf{T}(\mathbb{A})\right]\subset \left[\tensor[^\Theta]{\mathbf{H}}{}(\mathbb{A})\right] \cap \left[\mathbf{SO}_3(\mathbb{A})\right]$ -- hence the quotient image of $\left[\mathbf{T}(\mathbb{A})\right]$ is contained in the intersection of the images of the latter two homogeneous sets, which is exactly $\left[\tensor[^\Theta]{\mathbf{H}}{}(\mathbb{R})\right] \cap \left[\mathbf{SO}_3(\mathbb{R})\right]$.

To establish the inverse inclusion it is enough to show for every $y\in\mathscr{H}_D$ that
\begin{equation*}
\mathbf{SL}_3(\mathbb{Q})\cdot g_{x\to y}\mathbf{T}(\mathbb{R})\subset \mathbf{SL}_3(\mathbb{Q})\cdot \mathbf{T}(\mathbb{A})\cdot \mathbf{SO}_3(\hat{\mathbb{Z}}) \, .
\end{equation*}
Lemma \ref{lem:SO3(Q)-transitive} furnishes the existence of $g_\mathbb{Q}\in\mathbf{SO}_3(\mathbb{Q})$ such that $g_\mathbb{Q}.x=y$. Using this element we write
\begin{equation*}
\mathbf{SL}_3(\mathbb{Q})\cdot g_{x\to y}\mathbf{T}(\mathbb{R})=\mathbf{SL}_3(\mathbb{Q})\cdot \mathbf{T}(\mathbb{R})\cdot g_{\mathbb{Q},f}^{-1} \, ,
\end{equation*}
where $g_{\mathbb{Q},f}\in\mathbf{SO}_3(\mathbb{A}_f)$ is the diagonal embedding of $g_\mathbb{Q}^{-1}$. Lemma \ref{lem:SO3-local-transitivity} implies that for any $p$ there is an element $k_p\in\mathbf{SO}_3(\mathbb{Z}_p)$ such that $k_p.x=y\Rightarrow g_\mathbb{Q}^{-1}\in \mathbf{T}(\mathbb{Q}_p)k_p^{-1}$. Hence $g_{\mathbb{Q},f}\in \mathbf{T}(\mathbb{A}_f)\cdot \mathbf{SO}_3(\hat{\mathbb{Z}})$.
\end{proof}

\begin{remark}\label{rem:squaring}
The proposition above implies that the finite abelian group $$C_D\coloneqq\dfaktor{\mathbf{T}(\mathbb{Q})}{\mathbf{T}(\mathbb{A})}{\mathbf{T}(\mathbb{R})\cdot\mathbf{T}(\mathbb{A}_f)\cap \mathbf{SO}_3(\hat{\mathbb{Z}})}$$ acts simply transitively on the correspondence $\mathscr{J}_D$ from Conjecture \ref{conj:joint-equidistirbution}. Recall tht $\mathbf{T}\simeq \lfaktor{\Gm}{\WR_{E/\mathbb{Q}} \Gm}$ for some quadratic imaginary extension $E/\mathbb{Q}$ and for all primes $p\neq 2$ the group $K_{B,p}$ is the projective group of units of a maximal order. This implies that $C_D$ is a quotient
of a Picard group of an order\footnote{It is not necessarily the Picard group itself because the compact group $K_{\mathbf{B},2}$ is all of $\mathbf{SO}_3(\mathbb{Q}_2)$, which is bigger then the image of integral elements $\mathcal{O}_2^\times$ in the projective group of units. Specifically, if $2$ ramifies in $E$ then $C_D$ is a quotient of $\Pic(\Lambda)$ by the order $2$ group generated by the prime above $2$. Otherwise, $C_D=\Pic(\Lambda)$.} $\Lambda\subset E$.

In \cite{EMV} it is shown that if $x\in\mathcal{H}_D$ then $(\lambda.x)^\perp(\mathbb{Z})=\lambda^2.x^\perp(\mathbb{Z})$ for each $\lambda\in C_D$.
This squaring of the action is evident in our description because the intersection is a homogeneous toral set in $\mathbf{ASL}_2$, rather then in $\mathbf{AGL}_2$. The group $\mathbf{SO}_3\simeq \mathbf{PB}^\times$ is of adjoint type while $\mathbf{SL}_2$ is simply-connected. The intersection construction provides a map $\mathbf{T}\hookrightarrow \tensor[^\Theta]{\mathbf{H}}{}\simeq \mathbf{ASL}(x^1+x^\perp)$. We claim that this map is the isomorphism $\lfaktor{\Gm}{\WR_{E/\mathbb{Q}} \Gm}\to \WR_{E/\mathbb{Q}}^1 \Gm$ defined by $\lambda\mapsto \frac{\lambda}{\tensor[^\sigma]{\lambda}{}}$ which descends to a square of $\lambda$ in $C_D$. To see this identify $\mathbf{B}^0= \Lie(\mathbf{PB}^\times)$ in the standard fashion, then the space $x^\perp$ is the spanned by the non-trivial roots of $\mathbf{T}$ in $\Lie(\mathbf{PB}^\times)$. In particular, $\mathbf{T}$ acts in $x^\perp$ with weights $\lambda\mapsto \frac{\lambda}{\tensor[^\sigma]{\lambda}{}}$ and $\lambda\mapsto \frac{\tensor[^\sigma]{\lambda}{}}{\lambda}$. The weights of $\mathbf{T}$ and the fact $\mathbf{T}.x=x$ characterize the embedding $\mathbf{T}\hookrightarrow \mathbf{ASL}\left(x^1+x^\perp\right)$. On the level of tori this embedding is seen to coincide with the map $\lfaktor{\Gm}{\WR_{E/\mathbb{Q}} \Gm}\to \WR_{E/\mathbb{Q}}^1 \Gm$ from above.
\end{remark}

\subsection{Ad\`elic Torus Action on Orthogonal Grids}
Our last step is to present the collection of orthogonal grids $\mathscr{G}_D\subset \mathcal{L}^1_\bullet(\hat{e}^\perp)$ as a projection of an ad\`elic homogeneous toral set.
Fix an orientation preserving rational linear isomorphism $\varphi$ of $\hat{e}^\perp$ with the affine $2$-space mapping the lattice $\hat{e}^\perp(\mathbb{Z})$ to $\mathbb{Z}^2$. This map is uniquely-defined up to post composition with an element of $\mathbf{SL}_2(\mathbb{Z})$. This induces a unique isomorphism
\begin{equation*}
\varphi\colon \mathcal{L}^1_\bullet(\hat{e}^\perp)\to \lfaktor{\mathbf{ASL}_2(\mathbb{Z})}{\mathbf{ASL}_2(\mathbb{R})}
\end{equation*}
mapping the lattice $\hat{e}^\perp(\mathbb{Z})$ to the identity coset on the right. To see the action of $\mathbf{ASL}_2(\mathbb{R})$ on $\mathcal{L}^1_\bullet(\hat{e}^\perp)$ as an action of a subgroup of $\mathbf{SL}_3(\mathbb{R})$ we define first an affine isomorphism $\hat{e}^\perp\to\hat{e}+\hat{e}^\perp$ by $P\mapsto P+\hat{e}$. This induces a bijection
\begin{equation*}
\mathcal{L}^1_\bullet(\hat{e}^\perp)\to \mathcal{L}^1_\bullet(\hat{e}+\hat{e}^\perp)
\end{equation*}
defined by $P+L\mapsto(\hat{e}+P)+L$ and. By composing with $\varphi$ we derive an affine isomorphism of $\hat{e}+\hat{e}^\perp$ with the affine $2$-space. This affine isomorphism intertwines the action  of the group $\tensor[^\Theta]{\Stab_{\mathbf{SL}_e}(\hat{e})}{}$ on $\hat{e}+\hat{e}^\perp$ with the action of $\mathbf{ASL}_2$ on the affine $2$-space. In particular, we henceforth identify $\mathbf{ASL}_2=\tensor[^\Theta]{\Stab_{\mathbf{SL}_3}(\hat{e})}{}$.

\begin{prop}\label{prop:grids-packet}
Let $\delta\in\mathbf{SL}_3(\mathbb{R})$ be any element satisfying $\delta.x=\hat{e}$ and fix $g_\infty\in\mathbf{ASL}_2(\mathbb{R})\subset \mathbf{SL}_3(\mathbb{R})$ such that $\Ad_{g_\infty^{-1} \tensor[^\Theta]{\delta}{}} \mathbf{T}(\mathbb{R})=\mathbf{SO}_2(\mathbb{R})$. Then the homogeneous toral set
\begin{equation*}
\mathcal{H}_D\coloneqq\left[\left(\Ad_{\tensor[^\Theta]{\delta}{}} \mathbf{T}\right)(\mathbb{A})g_\infty\right]\subset \left[\mathbf{ASL}_2(\mathbb{A})\right]
\end{equation*}
project to $\varphi(\mathscr{G}_D)\subset \lfaktor{\mathbf{ASL}_2(\mathbb{Z})}{\mathbf{ASL}_2(\mathbb{R})}$. The discriminant of $\mathcal{H}_D$ is $-4D$ if $D\equiv 1,2 \mod 4$ and $-D$ if $D\equiv 3 \mod 4$. The torsion order is $D$.
\end{prop}
\begin{remark}
Notice that there is no canonically defined homogeneous toral set projecting to $\varphi(\mathscr{G}_D)$ because only the coset $\tensor[^\Theta]{\mathbf{ASL}}{_2}(\mathbb{Z})\delta$ is uniquely defined. Nevertheless, all the possible homogeneous toral set have the same projection; essentially, because $\tensor[^\Theta]{\mathbf{ASL}}{_2}(\mathbb{Z})$ is contained in $\mathbf{SL}_3(\mathbb{Z}_p)$ for all primes $p$.
\end{remark}
\begin{proof}
First notice that $\Ad_\delta \mathbf{T}< \Stab_{\mathbf{SL}_3}(\hat{e})\Rightarrow
\Ad_{\tensor[^\Theta]{\delta}{}} \mathbf{T}<\mathbf{ASL}_2$. Because $\mathbf{ASL}_2$ has class number one to compute the projection of $\mathcal{H}_D$ to $\lfaktor{\mathbf{ASL}_2(\mathbb{Z})}{\mathbf{ASL}_2(\mathbb{R})}$ is is enough to compute the projection of
\begin{equation*}
\widetilde{\mathcal{H}_D}\coloneqq\left[\tensor[^\Theta]{\delta}{}\mathbf{T}(\mathbb{A})\tensor[^\Theta]{\delta}{^{-1}}g_\infty\right]
=\left[\mathbf{T}(\mathbb{A})\tensor[^\Theta]{\delta}{^{-1}}g_\infty\right] \subset\left[\mathbf{SL}_3(\mathbb{A})\right]
\end{equation*}
to $\lfaktor{\mathbf{SL}_3(\mathbb{Z})}{\mathbf{SL}_3(\mathbb{R})}$.
For all finite places $v<\infty$ we have $\tensor[^\Theta]{\delta}{^{-1}}\in\mathbf{SL}_3(\mathbb{Z})\subset \mathbf{SL}_3(\mathbb{Z}_v)$. Hence the real projection of $\widetilde{\mathcal{H}_D}$ is the same as of $\left[\mathbf{T}(\mathbb{A})\tensor[^\Theta]{\delta}{^{-1}_\infty}g_\infty\right]$ where $\tensor[^\Theta]{\delta}{^{-1}_\infty}$ is the image of $\tensor[^\Theta]{\delta}{^{-1}}$ in $\mathbf{SL}_3(\mathbb{R})$. Yet this projection is the right translation by   $$g=\tensor[^\Theta]{\delta}{^{-1}_\infty}g_\infty\in\mathbf{SL}_3(\mathbb{R})$$ of the real projection of $\left[\mathbf{T}(\mathbb{A})\right]$.

From Proposition \ref{prop:adelic-intersection} we deduce that $\mathcal{H}_D$ projects to
$\left(
\left[\tensor[^\Theta]{\mathbf{H}}{}(\mathbb{R})\right] \cap \left[\mathbf{SO}_3(\mathbb{R})\right]
\right)g$ in $\lfaktor{\mathbf{SL}_3(\mathbb{Z})}{\mathbf{SL}_3(\mathbb{R})}$. The element $g^{-1}\in\mathbf{SL}_3(\mathbb{R})$ conjugates $\mathbf{T}(\mathbb{R})$ to $\mathbf{SO}_2(\mathbb{R})$.
Let $g_{x\to\hat{e}}\in\mathbf{SO}_3(\mathbb{R})$ satisfy $g_{x\to\hat{e}}.x=D^{1/2}\hat{e}$, then $g_{x\to\hat{e}}$ also conjugates $\mathbf{T}(\mathbb{R})$ to $\mathbf{SO}_2(\mathbb{R})$. This implies that $g_{x\to\hat{e}}g$ normalizes $\mathbf{SO}_2(\mathbb{R})$.

To learn what are the possibilities for $g_{x\to\hat{e}}g$ we compute $\Nrml_{\mathbf{SL}_3(\mathbb{R})}(\mathbf{SO}_2(\mathbb{R}))$. Because $\mathbf{SO}_2(\mathbb{R})$ is a torus with non-singular elements $\Cent_{\mathbf{SL}_3(\mathbb{R})}\mathbf{SO}_2(\mathbb{R})$ is the maximal torus
\begin{align*}
S&\coloneqq
\exp\left(*H\right)\cdot \mathbf{SO}_2(\mathbb{R}) \, ,\\
H&\coloneqq \begin{pmatrix}
1 & 0 & 0 \\ 0 & 1 & 0 \\ 0 & 0 & -2
\end{pmatrix} \, .
\end{align*}
The normalizer of $\mathbf{SO}_2(\mathbb{R})$ is contained in the normalizer of $S$. The group
$\Nrml(S)/S$ is contained in the absolute Weyl group of $S$. The absolute Weyl group is the permutation group on the three absolute characters of $S$ corresponding to the three eigenvalues.
If $k\in\mathbf{SO}_2(\mathbb{R})$ has eigenvalues $\exp(\pm i\theta)$ then $\exp(tH)k$ has eigenvalues $ \exp(t\pm i\theta), \exp(-2t)$. The real group $\Nrml(S)/ S$ must keep the unique real eigenvalue $\exp(-2t)$ invariant. Hence the only non-trivial possibility for $\Nrml(S)/ S$ is the group $\cyclic{2}$ permuting $\exp(t\pm i\theta)$. We deduce that if there is a non-trivial element in $\Nrml(S)/ S$ it is represented by an element $g_0\in\mathbf{SL}_2(\mathbb{R})\hookrightarrow\mathbf{SL}_3(\mathbb{R})$ where $\mathbf{SL}_2(\mathbb{R})$ is embedded in the upper left block. Such an element need also normalize $\mathbf{SO}_2(\mathbb{R})$, but $\mathbf{SO}_2(\mathbb{R})$ is self-normalizing in $\mathbf{SL}_2(\mathbb{R})$ because it is a maximal compact. We conclude that $\Nrml_{\mathbf{SL}_3(\mathbb{R})}(\mathbf{SO}_2(\mathbb{R}))=S$.

At last we see that there is some $\lambda\in\mathbb{R}$ such that $g\in g_{x\to \hat{e}}^{-1}\mathbf{SO}_2(\mathbb{R})\exp(\lambda H)$. To compute $\lambda$ note that
\begin{align*}
\exp(\lambda H)&\in \mathbf{SO}_3(\mathbb{R})\tensor[^\Theta]{\delta}{^{-1}}\mathbf{ASL}_2(\mathbb{R})
\Rightarrow \exp(-\lambda H)=\tensor[^\Theta]{\exp(\lambda H)}{}
\in \mathbf{SO}_3(\mathbb{R}) \delta^{-1} \Stab_{\mathbf{SL}_3(\mathbb{R})}(\hat{e})\\
&\Rightarrow
\exp(4\lambda)=\langle \exp(-\lambda H).\hat{e}, \exp(-\lambda H).\hat{e} \rangle= \langle \delta^{-1}.\hat{e}, \delta^{-1}.\hat{e} \rangle= \langle x, x \rangle =D \, .
\end{align*}
Hence $\exp(\lambda)=D^{1/4}$.

The projection $\left(
\left[\tensor[^\Theta]{\mathbf{H}}{}(\mathbb{R})\right] \cap \left[\mathbf{SO}_3(\mathbb{R})\right]
\right)g$ is a collection of $\mathbf{SO}_2(\mathbb{R})$-orbits contained in $\left[\mathbf{ASL}_2(\mathbb{R})\right]$. To find the corresponding lattice cosets in $\mathcal{L}^1_\bullet(\hat{e}+\hat{e}^\perp)$ we use Proposition \ref{prop:intersection-grids} to write  the projection in terms of rank $2$ lattices in $\mathbb{R}^3$
\begin{align*}
\left(
\left[\tensor[^\Theta]{\mathbf{H}}{}(\mathbb{R})\right] \cap \left[\mathbf{SO}_3(\mathbb{R})\right]
\right)g
&\mapsto \bigsqcup_{y\in \mathbf{SO}_3(\mathbb{Z})\backslash\mathscr{H}_D} g^{-1}\mathbf{T}(\mathbb{R})g_{x\to y}^{-1}.\left(y^1+y^\perp(\mathbb{Z})\right)\\
&=\bigsqcup_{y\in \mathbf{SO}_3(\mathbb{Z})\backslash\mathscr{H}_D} \mathbf{SO}_2(\mathbb{R}) \left(g_{x\to y}g\right)^{-1}.\left(y^1+y^\perp(\mathbb{Z})\right)\\
&=\bigsqcup_{y\in \mathbf{SO}_3(\mathbb{Z})\backslash\mathscr{H}_D} \exp(-\lambda H)\mathbf{SO}_2(\mathbb{R}) g_{y\to \hat{e}}.\left(y^1+y^\perp(\mathbb{Z})\right) \, .
\end{align*}
We see that these lattice cosets are the rotation of the orthogonal cosets to some plane orthogonal to  $\hat{e}$. The element $\exp(-\lambda H)$ acts on a plane orthogonal to $\hat{e}$ as homothety by the scalar $\exp(-\lambda)=D^{-1/4}$. The plane $y^1+y^\perp$ is equal to $y/D+y^\perp$, thus all these lattice cosets are rotated to the plane $\hat{e}+e^\perp(\mathbb{R})$ as expected. Observing the bijection between $\mathcal{L}^D_\bullet(x^1+x^\perp)$ and $\mathcal{L}^D_\bullet(x^\perp)$ we deduce that the projection of the homogeneous toral set is exactly as claimed.

To compute the discriminant notice that it depends only on the projection of the homogeneous toral set to $\lfaktor{\mathbf{SL}_2(\mathbb{Z})}{\mathbf{SL}_2(\mathbb{R})}$ and this discriminant can be computed using Proposition \ref{prop:complex-torus-cm}. In particular, the discriminant coincides with the discriminant of a primitive integral representative of the quadratic form $\langle , \rangle$ restricted to $x^\perp(\mathbb{Z})$. This can be computed elementary and shown to be equal to the claimed value, cf.\cite[\S 4.1.2]{AES}. The torsion order is $D$ exactly because $x^\mathrm{tors}$ is an order $D$ torsion point.
\end{proof}

\subsection{Proof of Joint Equidistribution Theorem}
\begin{proof}[Proof of Theorem \ref{thm:joint}]
Using the methods of \cite{AES} it is enough to show that the normalized counting measures on $\varphi(\mathscr{G}_D)$ converge weak-$*$ to the Haar measure on $\lfaktor{\mathbf{ASL}_2(\mathbb{Z})}{\mathbf{ASL}_2(\mathbb{R})}$ when $D\to\infty$. The joint equidistribution then follows by the joining rigidity theorem of Einsiedler and Lindenstrauss \cite{ELJoinings}. The equidistribution of $\varphi(\mathscr{G}_D)$ is an immediate corollary of Theorem \ref{thm:main} and Proposition \ref{prop:grids-packet}.
\end{proof}
\appendix
\section{Modified Hecke L-functions}\label{app:L-functions}
Fix an imaginary quadratic field $E/\mathbb{Q}$ and an order $\Lambda<\mathcal{O}_E$ Denote by $D_E$ the discriminant of $E$ then $D=D_Ef^2$ where $f\in\mathbb{N}$ is the conductor.

This appendix is dedicated to studying the $L$-functions defined in \ref{defi:L-functions} which coincide with some Hecke $L$-functions of the field $E/\mathbb{Q}$ modified at finitely many places.

\subsection{Local Euler Factors}
Fix a non-archimedean place $v$ and denote the residue characteristic of $\mathbb{Q}_v$ by $p$. Let $\meas_v$ be the additive Haar measure on $E_v$ normalized so that $\meas(\mathcal{O}_{E_v})=1$. By abuse of notation we will write $\meas$ for $\meas_v$ when the rational place $v$ is understood from the context.

Exactly as in Definition \ref{defi:ideals+level} we extend the non-archimedean absolute value from local fields to a norm on local \'etale-algebras by taking the product of the absolute values of all coordinates
\begin{equation*}
\|\bullet\|_v \coloneqq \prod_{w\mid v} |\bullet|_w \, .
\end{equation*}
Whenever $v$ is fixed by the context we may drop the $v$ subscript in $\|\bullet\|_v$.
This definition coordinates well with the change of variable formula which reads $\meas(a B)=\|a\|_v\meas(B)$ for any $a\in E_v$ and all Borel sets $B\subseteq E_v$. The general change of variables formula is $g_*.m=|\det g|_v \cdot m$ for all $g\in \End_{\mathbb{Q}_v}(E_v)$.
If $E_v/\mathbb{Q}_v$ is a quadratic \'etale-algebra then $\|p^k\|_v=p^{-2k}$ for all $k\in\mathbb{Z}$.

We introduce a definition to be employed only in the appendix.
\begin{defi}
For any $\mathrm{w}\in\mathbb{A}_E$ denote by $\ord(\mathrm{w})$ the order of the non-archimedean part $\mathrm{w}_f\in \mathbb{A}_{E,f}$ in the torsion group $\mathbb{A}_{E,f}/\Lambda_f$.

Notice that if $E$ and $\Lambda$ are associated to a homogeneous toral set $\mathcal{H}=\left[\mathbf{T}(\mathbb{A})(l,\mathrm{x})\right]$ then $\ord(\jmath_\mathbb{A}(\mathrm{y}))=\ord_\mathcal{H}(\mathrm{y})$ for all $\mathrm{y}\in\mathbf{V}(\mathbb{A})$.
\end{defi}

Fix $\mathrm{x},\mathrm{y}\in\mathbb{A}_E$.
We are interested in counting elements of $\mathfrak{a}\in\Ideals(\Lambda,\mathrm{x})$ satisfying $\mathfrak{a}\in \mathrm{x}_v+\Lambda_v-\Aut^1(\Lambda_v).\mathrm{y}_v$ weighted by a character $\chi\colon \lfaktor{E^\times}{\mathbb{A}_E^\times}\to S^1$. The local factor at $v$ is described by
\begin{equation}\label{eq:L-order-local-integral}
L_{\Lambda_v(\mathrm{x},\mathrm{y})}(s,\chi)=\meas(\Lambda_v^\times(\mathrm{x}))^{-1}\int_{\mathrm{x}_v+\Lambda_v-\Aut^1(\Lambda_v).\mathrm{y}_v} \chi(z) \|z\|^{s-1}\dif\meas(z) \, .
\end{equation}
The function $\chi\colon E_v\to \mathbb{C}$ is defined using the composition $E_v^\times \hookrightarrow \mathbb{A}_E^\times\xrightarrow{\chi}S^1$ and extended to all of $E_v$ by letting it vanish on non-invertible elements.

If $\chi=1$ we shall omit it from the notation.
If $\mathrm{y}_v=0$ we denote $L_{\Lambda_v(\mathrm{x})}=L_{\Lambda_v(\mathrm{x},\mathrm{y})}$. Moreover, if $\mathrm{x}_v=\mathrm{y}_v=0$ then we write $L_{\Lambda_v(\mathrm{x},\mathrm{y})}=L_{\Lambda_v}$.

The following elementary lemma shows that each local factor is bounded in vertical strips. This is useful when applying Perron's formula.
\begin{lem}\label{lem:local-vertical}
For all $s\in\mathbb{C}$ and all characters $\chi$
\begin{equation*}
|L_{\Lambda_v(\mathrm{x},\mathrm{y})}(s,\chi)|\leq L_{\Lambda_v(\mathrm{x},\mathrm{y})}(\Re s) \, .
\end{equation*}
\end{lem}
\begin{proof}
\begin{equation*}
|L_{\Lambda_v(\mathrm{x},\mathrm{y})}(s)|\leq \meas(\Lambda_v^\times(\mathrm{x}))^{-1}\int_{\mathrm{x}_v+\Lambda_v-\Aut^1(\Lambda_v).\mathrm{y}_v} |\chi(z)| \cdot|\|z\|^{s-1}|\dif\meas(z)=L_{\Lambda_v(\mathrm{x},\mathrm{y})}(\Re s)
\end{equation*}
\end{proof}

\subsubsection{Structure of a non-Maximal Quadratic Order}
We assume $\Lambda_v<\mathcal{O}_{E_v}$ is a non-maximal order of conductor $f_v$.  To simplify the notation we assume without loss of generality that the conductor is a power of $p$ and write
\begin{equation*}
f_v=p^n
\end{equation*}

We fix the standard set of representatives in $\{0,1,\ldots,p^{n-1}\}\subset\mathbb{Z}$ for $\cyclic{p^n}$
and use them also as representatives for $\faktor{\mathbb{Z}_v}{f_v\mathbb{Z}_v}\simeq \cyclic{p^n}$.
\begin{lem}\label{lem:Lambda-reduction}
One has
\begin{align*}
\Lambda_v&=\mathbb{Z}_v+f_v\mathcal{O}_{E_v}=\mathbb{Z}+f_v\mathcal{O}_{E_v}=\bigsqcup_{a\in\cyclic{f_v}} a+f_v\mathcal{O}_{E_v} \, ,\\
\Lambda_v^\times&=\mathbb{Z}_v^\times+f_v\mathcal{O}_{E_v}=\mathbb{Z}^\times+f_v\mathcal{O}_{E_v}=\bigsqcup_{a\in\cyclic{f_v}^\times} a+f_v\mathcal{O}_{E_v} \, .
\end{align*}
\end{lem}
\begin{proof}
The second equality in both statements is a corollary of weak approximation.

The first statement is proven in a manner identical to orders in quadratic number fields.
For the second statement the inclusion $\Lambda_v^\times\subseteq\mathbb{Z}_v^\times+f_v\mathcal{O}_{E_v}$ is immediate. For the reverse inclusion it is enough to show that $1+f_v\mathcal{O}_{E_v}\subseteq \Lambda_v^\times$. This follows by writing $(1+f_v\alpha)^{-1}=\sum_{k=0}^\infty (-f_v\alpha)^k$ and noticing that $\Lambda_v$ is closed, hence the power series converges to a value in $\Lambda_v$.
\end{proof}
\begin{lem}\label{lem:local-volumes}
The following formulae hold
\begin{align*}
\meas(\mathcal{O}_{E_v}^\times)&=\left(1-p^{-1}\right)\left(1-\chi_E(p)p^{-1}\right)=L_{\mathcal{O}_{E_v}}(1)^{-1} \, ,\\
\meas(\Lambda_v^\times)&=f_v^{-1}\left(1-p^{-1}\right) \, ,\\
\meas(\Lambda_v)&=f_v^{-1} \, .
\end{align*}
\end{lem}
\begin{proof}
The formula for $\meas(\mathcal{O}_{E_v}^\times)$ follows by subtracting the measure of the maximal ideals using the inclusion-exclusion principle. There is $1$ maximal ideal in the inert and ramified cases and $2$ in the split one.

The second formula follows from Lemma \ref{lem:Lambda-reduction} in the following manner
\begin{align*}
\meas(\Lambda_v^\times)=&\meas(\mathbb{Z}^\times+f_v\mathcal{O}_{E_v})
=\sum_{a \in \left(\cyclic{f_v}\right)^\times} \meas(a+f_v\mathcal{O}_E)\\
&=\sum_{a \in \left(\cyclic{f_v}\right)^\times} \|f_v\| =f_v\left(1-\frac{1}{p}\right)\|f_v\| \, .
\end{align*}
The third formula is proved in a similar manner.
\end{proof}
\begin{cor}\label{cor:units-volume-ratio} The ratio of the volumes of group of units satisfies
\begin{equation*}
\left[\mathcal{O}_{E_v}^\times \colon \Lambda_v^\times \right]
=\frac{\meas(\mathcal{O}_{E_v}^\times)}{\meas(\Lambda_v^\times)}=f_v\left(1-\chi_E(p)p^{-1}\right) \, .
\end{equation*}
\end{cor}\qed

\subsubsection{First Properties of the Local Factor}
\begin{prop}\label{prop:L-local-denominators}
\begin{equation*}
\|\ord(\mathrm{x})\ord(\mathrm{y})\|^s
L_{\Lambda_v(\mathrm{x},\mathrm{y})}(s,\chi)\in \mathbb{C}[[p^{-s}]] \, .
\end{equation*}
\end{prop}
\begin{proof}
Denote $\Lambda_v(\mathrm{x},\mathrm{y})'\coloneqq \ord(\mathrm{x})\ord(\mathrm{y})\left(\mathrm{x}_v+\Lambda_v-\Aut^1(\Lambda_v).\mathrm{y}_v\right)$ then $\Lambda_v(\mathrm{x},\mathrm{y})\subset \Lambda_v$ and
\begin{align*}
\meas(\Lambda_v^\times(\mathrm{x}))L_{\Lambda_v(\mathrm{x},\mathrm{y})}(s,\chi)
&= \int_{\Lambda_v(\mathrm{x},\mathrm{y})'} \chi\left(\frac{z}{\ord(\mathrm{x})\ord(\mathrm{y})}\right)
\cdot \left\|\frac{z}{\ord(\mathrm{x})\ord(\mathrm{y})}\right\|^{s-1}
\dif\meas\left(\frac{z}{\ord(\mathrm{x})\ord(\mathrm{y})}\right)\\
&=\|\ord(\mathrm{x})\ord(\mathrm{y})\|^{-s} \chi(\ord(\mathrm{x})\ord(\mathrm{y}))^{-1}
\int_{\Lambda_v(\mathrm{x},\mathrm{y})'} \chi(x)\|z\|^{s-1}\dif\meas(z)\\
&\in \|\ord(\mathrm{x})\ord(\mathrm{y})\|^{-s} \mathbb{C}[[p^{-s}]] \, .
\end{align*}
\end{proof}

The following results reduce our further work to the case where we need only evaluate  $L_{\Lambda_v(\mathrm{x})}(s)$.

\begin{prop}\label{prop:Lx<=Lxy}
\begin{equation*}
|L_{\Lambda_v(\mathrm{x},\mathrm{y})}(s,\chi)|
\leq \|\ord(\mathrm{y})\|^{-\Re s} \frac{\meas\left(\Lambda_v^\times\left(\ord(\mathrm{y})\mathrm{x}\right)\right)}
{\meas(\Lambda_v^\times(\mathrm{x}))}
 L_{\Lambda_v(\ord(\mathrm{y})\mathrm{x})}(\Re s) \, .
\end{equation*}
\end{prop}
\begin{proof}
Notice first that $\ord(\mathrm{y})\cdot\mathrm{y}_v\in \Lambda_v$ thus $\mathrm{x}_v+\Lambda_v-\Aut^1(\Lambda_v).\mathrm{y}_v\subset \mathrm{x}_v+\ord(\mathrm{y})^{-1}\Lambda_v$. We deduce
\begin{align*}
\meas(\Lambda_v^\times(\mathrm{x}))|L_{\Lambda_v(\mathrm{x},\mathrm{y})}(s,\chi)|
&\leq \int_{\mathrm{x}_v+\ord(\mathrm{y})^{-1}\Lambda_v} |\chi(z)|\cdot | \|z\|^{s-1} |\dif\meas(z)\\
&=\int_{\ord(\mathrm{y})\mathrm{x}_v+\Lambda_v} \left\|\frac{z}{\ord(\mathrm{y})}\right\|^{\Re s-1}\dif\meas\left(\frac{z}{\ord(\mathrm{y})}\right)\\
&=\|\ord(\mathrm{y})\|^{-\Re s} \meas(\Lambda_v^\times(\ord(\mathrm{y})\mathrm{x}))L_{\Lambda_v(\ord(\mathrm{y})\mathrm{x})}(\Re s) \, .
\end{align*}
\end{proof}

We can now evaluate the local factor when $s=1$ and $\chi=1$.
\begin{prop}\label{prop:L-local-at-1}
For any $\mathrm{x}\in\mathbb{A}_E$
\begin{equation*}
L_{\Lambda_v(\mathrm{x})}(1)=
\left(\meas(\Lambda_v^\times(\mathrm{x}))f_v\right)^{-1}
=\left[\Lambda_v^\times: \Lambda_v^\times(\mathrm{x}) \right](1-p^{-1})^{-1}\cdot\begin{cases}
(1-\chi_E(p)p^{-1})^{-1} & f_v=1\\
1 & f_v>1
\end{cases} \, .
\end{equation*}
In particular,
\begin{align*}
L_{\Lambda_v(\mathrm{x},\mathrm{y})}(1)&\leq
\left(\meas(\Lambda_v^\times(\mathrm{x}))\|\ord(\mathrm{y})\|f_v\right)^{-1}\\
&= \|\ord(\mathrm{y})\|^{-1}\left[\Lambda_v^\times: \Lambda_v^\times(\mathrm{x}) \right](1-p^{-1})^{-1}\cdot\begin{cases}
(1-\chi_E(p)p^{-1})^{-1} & f_v=1\\
1 & f_v>1
\end{cases} \, .
\end{align*}
\end{prop}
\begin{proof}
The first claim follows directly from \eqref{eq:L-order-local-integral} and Lemma \ref{lem:local-volumes}. The second claim follows from the first part and Proposition \ref{prop:Lx<=Lxy}.
\end{proof}

\subsubsection{Local Factor When \texorpdfstring{$\mathrm{x}_v=0$}{x=0}}
In this section we compute $L_{\Lambda_v}$ which coincides with $L_{\Lambda_v(\mathrm{x})}$ whenever $\mathrm{x}_v\in\Lambda_v$.

\begin{lem}\label{lem:local-coset-integrals}
Let $a\in\mathbb{Z}_v$ then
\begin{equation*}
\int_{a+f_v\mathcal{O}_{E_v}} \|z\|^{s-1} \dif\meas(z)=\begin{cases}
\|f_v\|\|a\|^{s-1} & \|a\|>\|f_v\|\\
\|f_v\|^s \meas(\mathcal{O}_{E_v}^\times) L_{\mathcal{O}_{E_v}}(s) & a=0
\end{cases} \, .
\end{equation*}
\end{lem}
\begin{proof}
If $\|a\|>\|f_v\|$ then $\|a+x\|=\|a\|$ for all $x\in f_v\mathcal{O}_{E_v}$, thus
\begin{align*}
\int_{a+f_v\mathcal{O}_{E_v}} \|z\|^{s-1} \dif\meas(z)&=
\int_{f_v\mathcal{O}_{E_v}} |a+z|^{s-1} \dif\meas(z)=\\
\int_{f_v\mathcal{O}_{E_v}} \|a\|^{s-1} \dif\meas(z)&=\|a\|^{s-1}\meas(f_v\mathcal{O}_{E_v})
=\|a\|^{s-1}\|f_v\|
\end{align*}
The equality in the case $a=0$ is an immediate application of the change of variable $x\mapsto \frac{x}{f_v}$ and \eqref{eq:L-order-local-integral}.
\end{proof}

\begin{prop}\label{prop:local-l-factor-formula}
Assume $\Lambda_v\subsetneq \mathcal{O}_{E_v}$ is non-maximal.
The local factor of $L_\Lambda$ at $v$ is for all $s\neq 1/2$ equal to
\begin{align*}
L_{\Lambda_v}(s)&=\frac{1-f_v^{-(2s-1)}}{1-p^{-(2s-1)}}+f_v^{-(2s-1)}\left(1-\left(\frac{D_E}{p}\right)\frac{1}{p}\right)L_{\mathcal{O}_{E_v}}(s)\\
&=\frac{1-f_v^{-(2s-1)}}{1-p^{-(2s-1)}}+\frac{f_v^{-(2s-1)}}{1-p^{-1}}\frac{L_{\mathcal{O}_{E_v}}(s)}{L_{\mathcal{O}_{E_v}}(1)} \, ,
\end{align*}
while for $s=1/2$
\begin{equation*}
L_{\Lambda_v}(1/2)=n+\frac{1}{1-p^{-1}}\frac{L_{\mathcal{O}_{E_v}}(1/2)}{L_{\mathcal{O}_{E_v}}(1)} \, .
\end{equation*}
\end{prop}
\begin{proof}
We use the integral representation \eqref{eq:L-order-local-integral} and Lemma \ref{lem:local-coset-integrals} to write
\begin{equation}\label{eq:local-L-integral-decomposition}
L_{\Lambda_v}(s)= \meas(\Lambda_v^\times)^{-1} \int_{f_v\mathcal{O}_{E_v}} \|z\|^{s-1} \dif\meas(z)
+\meas(\Lambda_v^\times)^{-1}\sum_{0\neq a \in \cyclic{f_v}} \int_{a+f_v\mathcal{O}_{E_v}} \|z\|^{s-1} \dif\meas(z) \, .
\end{equation}
By Lemma \ref{lem:local-coset-integrals} and Corollary \ref{cor:units-volume-ratio} the first term on the right hand side of \eqref{eq:local-L-integral-decomposition} is equal to
\begin{equation*}
\frac{\meas(\mathcal{O}_{E_v}^\times)}{\meas(\Lambda_v^\times)} f_v^{-2s}L_{\mathcal{O}_{E_v}}(s)
=f_v^{-(2s-1)}\left(1-\chi_E(p)p^{-1}\right) L_{\mathcal{O}_{E_v}}(s)\, .
\end{equation*}

To evaluate the second term on the right hand side of \eqref{eq:local-L-integral-decomposition} we use Lemma \ref{lem:local-coset-integrals}
\begin{align*}
\sum_{0\neq a \in \cyclic{f_v}} &\int_{a+f_v\mathcal{O}_{E_v}} \|z\|^{s-1} \dif\meas(z)=
\sum_{0\neq a \in \cyclic{f_v}} \|f_v\| \|a\|^{s-1}\\
&=f_v^{-2} \sum_{k=0}^{n-1}p^{n-k}\left(1-\frac{1}{p}\right)p^{-2k(s-1)}
=f_v^{-1}\left(1-\frac{1}{p}\right) \sum_{k=0}^{n-1}p^{-k(2s-1)}\\
&=f_v^{-1}\left(1-\frac{1}{p}\right)\frac{1-f_v^{-(2s-1)}}{1-p^{-(2s-1)}}=\meas(\Lambda_v^\times)\frac{1-f_v^{-(2s-1)}}{1-p^{-(2s-1)}} \, .
\end{align*}
To pass to the sum over $k$ above we have collected all non-trivial element of $\cyclic{f_v}$ according to their $p$-valuation -- $p^k$. In the last equality we have used Lemma \ref{lem:local-volumes}. Notice that the formula for the geometric sums only holds for $s\neq 1/2$ while for $s=1/2$ this sum is equal to $n$.

The claim follows by combining the expressions for all the summands in \eqref{eq:local-L-integral-decomposition}.
\end{proof}

\begin{cor}\label{cor:L1/2-x=0}
\begin{equation*}
0<\meas(\Lambda_v^\times)f_v L_{\Lambda_v}(1/2)\leq n+3 \, .
\end{equation*}
\end{cor}
\begin{proof}
Assume first $\Lambda_v\subsetneq \mathcal{O}_{E_v}$.
Substitute the formula for $L_{\mathcal{O}_{E_v}}$ into Proposition \ref{prop:local-l-factor-formula}
\begin{equation*}
L_{\Lambda_v}(1/2)=n(1-p^{-1})
+\frac{(1-p^{-1})(1-\chi_E(p)p^{-1})}{(1-p^{-1/2})(1-\chi_E(p)p^{-1/2})} \, .
\end{equation*}
The same formula holds with $n=0$ when $\Lambda_v=\mathcal{O}_{E_v}$.
The claim follows by substituting for the three possible case $\chi_E(p)=0,1,-1$ and applying the inequality $p\geq 2$.
\end{proof}

\subsubsection{Local Factor when \texorpdfstring{$\mathrm{x}_v$}{Xv} not in the Order}
In this section we compute for the case $\mathrm{x}_v\not\in \Lambda_v$.
\begin{prop}\label{prop:L-local-xv}
Assume $\mathrm{x}_v\not\in\Lambda_v$ then
\begin{equation*}
\meas(\Lambda_v^\times(\mathrm{x})) f_v L_{\Lambda_v(\mathrm{x})}(s)=f_v^{-1}\sum_{a\in \cyclic{f_v}} \prod_{w\mid v}\begin{cases}
|\mathrm{x}_{w}+a|_w^{s-1} &\mathrm{x}_w+a\not\in f_v\mathcal{O}_{E_w}\\
|f_v|_w^{s-1}\frac{1-p^{-1}}{1-p^{-s}} &\mathrm{x}_w+a\in f_v\mathcal{O}_{E_w}
\end{cases} \, ,
\end{equation*}
where the second option can happen only when $v$ is split in $E$ and at most for a single $a\in\cyclic{f_v}$ for each $w\mid v$. Moreover, it is impossible for the second option to happen with the same $a$ for both $w\mid v$

If $\Lambda_v=\mathcal{O}_{E_v}$ then a simpler formula holds
\begin{equation*}
\meas(\Lambda_v^\times(\mathrm{x})) f_v L_{\Lambda_v(\mathrm{x})}(s)=
\prod_{w\mid v}\begin{cases}
|\mathrm{x}_w|_w^{s-1} &\mathrm{x}_w\not\in \mathcal{O}_{E_w}\\
\frac{1-p^{-1}}{1-p^{-s}} &\mathrm{x}_w\in \mathcal{O}_{E_w}
\end{cases} \, .
\end{equation*}
Here as well the second option can happen only if $v$ splits and at most for one $w\mid v$.

\end{prop}
\begin{proof}
Formula \eqref{eq:L-order-local-integral} implies
\begin{align*}
\meas(\Lambda_v^\times(\mathrm{x})) L_{\Lambda_v(\mathrm{x})}&=\sum_{a\in \cyclic{f}} \int_{\mathrm{x}_v+a+f_v\mathcal{O}_{E_v}} \|z\|^{s-1}\dif\meas(z)\\
&=\sum_{a\in \cyclic{f}} \prod_{w\mid v} \int_{f_v\mathcal{O}_{E_w}} |\mathrm{x}_w+a+z_w|_w^{s-1} \dif\meas_w(z_w) \, .
\end{align*}
The assumption $\mathrm{x}_v\not\in\Lambda_v$ implies that $\mathrm{x}_v+a\not\in f_v\mathcal{O}_{E_v}$ for all $a\in\mathbb{Z}_v$, hence if $\mathrm{x}_w+a\in f_v\mathcal{O}_{E_w}$ then $v$ must be split in $E$.
For each $a\in\cyclic{f}$ we consider two cases.
If $v$ is split and $\mathrm{x}_w+a\in f_v\mathcal{O}_{E_w}$ then
\begin{align*}
\int_{f_v\mathcal{O}_{E_w}} |\mathrm{x}_w+a+z_w|_w^{s-1} \dif\meas_w(z_w)&=
\int_{f_v\mathbb{Z}_v} |z_w|_w^{s-1} \dif\meas_w(z_w)
=\int_{\mathbb{Z}_v} |f_vz_w|_w^{s-1} \dif\meas_w(f_vz_w)\\
&=|f_v|_w^s \frac{1-p^{-1}}{1-p^{-s}}=|f_v|_w |f_v|_w^{s-1}\frac{1-p^{-1}}{1-p^{-s}} \, .
\end{align*}
On the other hand if $\mathrm{x}_w+a\not \in f_v\mathcal{O}_{E_w}$ then $|\mathrm{x}_w+a+z_w|_w=|\mathrm{x}_w+a|_w$ for all $z_w\in f_v\mathcal{O}_{E_w}$ and
\begin{equation*}
\int_{f_v\mathcal{O}_{E_v}} |\mathrm{x}_w+a+z_w|_w^{s-1} \dif\meas_w(z_w)=
|\mathrm{x}_w+a|_w^{s-1} \meas_w(f_v \mathcal{O}_{E_w})
= |f_v|_w|\mathrm{x}_w+a|_w^{s-1} \, .
\end{equation*}
This and Proposition \ref{prop:L-local-at-1} implies the first claimed formula. If $\mathrm{x}_w+a_i\in f_v\mathcal{O}_{E_w}$ for some $a_1,a_2\in\cyclic{f_v}$ then $a_1-a_2=0$, thus the second option in the first claimed formula can occur at most for one $a\in\cyclic{f_v}$ for each $w\mid v$. It also may not occur for the same $a\in\cyclic{f_v}$ for both $w\mid v$ because otherwise $\mathrm{x}_v\in -a+f_v\mathcal{O}_{E_v}\subset \Lambda_v$ contradicting the assumption.

The second claimed formula is a simple specialization of the first one.
\end{proof}

\begin{cor}\label{cor:L1/2-non-max}
If $v$ splits in $E$ then
\begin{equation*}
0<\meas(\Lambda_v^\times(\mathrm{x}))f_vL_{\Lambda_v(\mathrm{x})}(1/2)\leq  \prod_{w\mid v} \begin{cases}
|\mathrm{x}_w|_w^{-1/2} & \mathrm{x}_w\not\in \mathcal{O}_{E_w}\\
2\sqrt{n} & \mathrm{x}_w\in \mathcal{O}_{E_w}
\end{cases} \, .
\end{equation*}
Otherwise,
\begin{equation*}
0<\meas(\Lambda_v^\times(\mathrm{x}))f_vL_{\Lambda_v(\mathrm{x})}(1/2)\leq
\begin{cases}
|\mathrm{x}_v|_v^{-1/2} & \mathrm{x}_v\not\in \mathcal{O}_{E_v}\\
4n & \mathrm{x}_v\in \mathcal{O}_{E_v}
\end{cases} \, .
\end{equation*}

\end{cor}
\begin{proof}
Notice first that
\begin{equation*}
\frac{1-p^{-1}}{1-p^{-1/2}}=1+p^{-1/2}<2 \, .
\end{equation*}
We can deduce from Proposition \ref{prop:L-local-xv} above that
\begin{equation*}
\meas(\Lambda_v^\times(\mathrm{x}))f_v L_{\Lambda_v(\mathrm{x})}(1/2)\leq f_v^{-1}\sum_{a\in \cyclic{f_v}} \prod_{w\mid v} \max\{|\mathrm{x}_w+a|_w,|f_v|_w\}^{-1/2}(1+\delta_{\mathrm{x}_w\in\mathcal{O}_{E_w}}) \, ,
\end{equation*}
where we have used the fact that if $\mathrm{x}_w+a\in f_v\mathcal{O}_{E_w}$ then $\mathrm{x}_w\in\mathcal{O}_{E_w}$.

If $v$ is split write $\|\bullet\|_v=|\bullet|_{w_1}|\bullet|_{w_2}$ as usual.
For this proof only we introduce a non-standard definition $\|\bullet\|_v=|\bullet|_{w_1}|\bullet|_{w_2}$ even if $v$ is not split and denote accordingly $w_i\mid v$ for $i=1,2$. If $v$ is either inert or ramified we define $|\bullet|_{w_i}\coloneqq \sqrt{\|\bullet\|_v}$.
Applying Cauchy-Schwartz to the inequality above we arrive at
\begin{equation}\label{eq:Lxv(1/2)-CS}
\meas(\Lambda_v^\times(\mathrm{x}))f_v L_{\Lambda_v(\mathrm{x})}(1/2)\leq  \prod_{w\mid v}(1+\delta_{\mathrm{x}_w\in\mathcal{O}_{E_w}})\sqrt{f_v^{-1}\sum_{a\in \cyclic{f_v}} \max\{|\mathrm{x}_w+a|_w,|f_v|_w\}^{-1}} \, .
\end{equation}
We follow by bounding the sum for each $w\mid v$ in \eqref{eq:Lxv(1/2)-CS} independently. Fix $w\mid v$. If $\mathrm{x}_w\not\in \mathcal{O}_{E_w}$ then
\begin{equation*}
f_v^{-1}\sum_{a\in \cyclic{f_v}} \max\{|\mathrm{x}_w+a|_w,|f_v|_w\}^{-1}= |\mathrm{x}_w|_w^{-1} \, .
\end{equation*}
Otherwise, we bound how many element $a\in \cyclic{f_v}$ exist with a fixed value of $|\mathrm{x}_w+a|_w$. Let $a_1,a_2\in\cyclic{f_v}$.
If $|\mathrm{x}_w+a_i|_w\leq|f_v|_w$ for $i=1,2$ then $a_1\equiv a_2 \mod f_v\mathcal{O}_{E_w}\Rightarrow a_1=a_2$. Thus at most one element $a\in\cyclic{f_v}$ satisfies $|\mathrm{x}_w+a|_w\leq|f_v|_w$.

Assume $|\mathrm{x}_w+a_1|_w=|\mathrm{x}_w+a_2|_w=p^{-k/2}>|f_v|_w$ where $0\leq k<2n$. The non-archimedean triangle inequality then implies $|a_1-a_2|_w\leq p^{-k/2}$ and $a_1\equiv a_2 \mod p^{\lceil k/2 \rceil}$. In $\cyclic{f_v}$ there $p^{n-\lceil k/2 \rceil}$ elements that reduce to the same element modulo $p^{\lceil k/2 \rceil}$. Hence there are at most $p^{n-\lceil k/2 \rceil}$ summands for which $|\mathrm{x}_w+a|_w=p^{-k/2}$. We deduce
\begin{equation*}
\sum_{a\in \cyclic{f_v}} \max\{|\mathrm{x}_w+a|_w,|f_v|_w\}^{-1}\leq
\sum_{k=0}^{2n} p^{n-k/2}p^{k/2}=2n p^n=2n f_v \, .
\end{equation*}
The claim follows by substituting this bound in \eqref{eq:Lxv(1/2)-CS}.
\end{proof}
\qed

\subsubsection{Local Factor with Bowen Level Structure}
Recall that in Definition \ref{defi:L-functions} we have a fixed prime $p_1$ and an integer $\tau\geq 0$. Let $v$ the place corresponding to $p_1$ and assume $v$ splits in $E$. If $\tau>0$ then $\Lambda_{p_1}(\mathrm{x},\mathrm{y})=\cap_{k=-\tau}^\tau \pi^k \tensor[^\sigma]{\pi}{^{-k}} \Lambda_{p_1}$ where $\pi$ is a uniformizer of $E_w$ for some $w\mid v$.
\begin{lem}
Assume $\tau>0$ then
\begin{equation*}
\bigcap_{k=-\tau}^\tau \frac{\pi^k}{\tensor[^\sigma]{\pi}{^k}} \Lambda_v=p_1^\tau f_v \mathcal{O}_{E_v} \, .
\end{equation*}
\end{lem}
\begin{proof}
The definition does not depend on the choice of uniformizer so we may as well write $v=w_1 w_{-1}$ and $\pi=p\in E_{w_1}$, then $\tensor[^\sigma]{\pi}{}=p\in E_{w_{-1}}$.
Use Lemma \ref{lem:Lambda-reduction} to write
\begin{equation}\label{eq:Bowen-order}
\bigcap_{k=-\tau}^\tau \frac{\pi^k}{\tensor[^\sigma]{\pi}{^k}} \Lambda_v
=\bigcap_{k=-\tau}^\tau \bigsqcup_{a\in\cyclic{f_v}} \prod_{i\in\{\pm1\}} p^{ik}a+f_vp^{ik}\mathcal{O}_{E_{w_i}} \, .
\end{equation}
In the intersection above consider two cosets corresponding to $a\in\cyclic{f_v}$ and $-\tau\leq k \leq \tau$ and $b\in\cyclic{f_v}$ and $-\tau\leq l \leq \tau$; without loss of generality assume $k\geq l$. The corresponding cosets can intersect only if
\begin{equation*}
p^{k-l} a \equiv b \mod f_v \textrm{ and }
a \equiv p^{k-l} b \mod f_v
\end{equation*}
This can happen  only if either $k=l$ and $a=b$ or $a=b=0$. Hence only the cosets with $a=0$ can contribute to the intersection in \eqref{eq:Bowen-order} and
\begin{equation*}
\bigcap_{k=-\tau}^\tau \frac{\pi^k}{\tensor[^\sigma]{\pi}{^k}} \Lambda_v
=f_v\bigcap_{k=-\tau}^\tau \prod_{i\in\{\pm1\}}fp^{ik}\mathcal{O}_{E_{w_i}}
=f_v p_1^\tau \mathcal{O}_{E_v} \, .
\end{equation*}
\end{proof}
 If $\tau>0$ the local Euler factor is
\begin{equation*}
L_{\Lambda_{p_1}(\mathrm{x},\mathrm{y})}(s,\chi)=L_{\Lambda_{p_1}^{(-\tau,\tau)}}(s,\chi)\coloneqq\meas(\Lambda_{p_1}^\times)^{-1}\int_{p_1^\tau f_v \mathcal{O}_{E_v}} \chi(z) \|z\|^{s-1} \dif z \, .
\end{equation*}

It is easy to express $L_{\Lambda_{p_1}^{(-\tau,\tau)}}$ in terms of $L_{\mathcal{O}_{E_v}}$.
\begin{lem}\label{lem:L-loc-Bowen}
\begin{equation*}
L_{\Lambda_{p_1}^{(-\tau,\tau)}}(s,\chi)=\chi(f_v p_1^\tau)^{-1} (f_v p_1)^{-2\tau s}
\frac{\meas(\mathcal{O}_{E_v}^\times)}{\meas(\Lambda_v^\times)}
L_{\mathcal{O}_{E_v}}(s,\chi) \, .
\end{equation*}
In particular,
\begin{equation*}
L_{\Lambda_{p_1}^{(-\tau,\tau)}}(1)=f_v^{-2}p_1^{-2\tau}{\meas(\Lambda_v^\times)}^{-1}= p_1^{-2\tau-\val_{p_1}(f_v)} L_{\Lambda_{p_1}}(1) \, .
\end{equation*}
\end{lem}
\begin{proof}
Follows by the change of variable formula $z\mapsto \pi^{-\tau} z$ where we treat multiplication by $\pi^{-\tau}$ as $\mathbb{Q}_{p_1}$ endomorphism of $E_{p_1}$ with determinant $p_1^{-\tau}$.
\end{proof}

\subsection{The Global L-function}\label{apndx:global}
We now study the global $L$-function $L_{\Lambda(\mathrm{x},\mathrm{y})}(s,\chi)$ as defined in \ref{defi:L-functions}.
Our aim is to bound residue at $1$ of $L_{\Lambda(\mathrm{x},\mathrm{y})}(s)$ and the absolute value  of $L_{\Lambda(\mathrm{x},\mathrm{y})}(s,\chi)$ on the line $\Re s=1/2$.

First we state the basic structure result for the $L$-functions $L_{\Lambda(\mathrm{x},\mathrm{y})}(s,\chi)$.
\begin{prop}\label{prop:L-Euler-product}
Let $L(s,\chi)$ be the Hecke $L$-function attached to $\chi\colon \lfaktor{E^\times}{\mathbb{A}_E^\times}\to S^1$. Assume $\chi$ is unramified outside of $\ord(\mathrm{x})\ord(\mathrm{y})f$. Then the following functions are equal as meromorphic functions for $\Re s>0$
\begin{equation*}
L_{\Lambda(\mathrm{x},\mathrm{y})}(s,\chi)
=L(s,\chi)\prod_{v\mid \ord(\mathrm{x})\ord(\mathrm{y})f} L_{\Lambda_v(\mathrm{x},\mathrm{y})}(s,\chi) L_v(s,\chi)^{-1} \, ,
\end{equation*}
where
\begin{equation*}
L_v(s,\chi)=\begin{cases}
1 & v \mid \operatorname{conductor}(\chi)\\
\prod_{w\mid v} \left(1-\chi(\pi_w)^{-1}|\pi_w|_w^{-s} \right)^{-1} & v \nmid \operatorname{conductor}(\chi)
\end{cases}
\end{equation*}
and $\pi_w$ is a uniformizer of $E_w$.
\end{prop}
\begin{proof}
The formal Euler product of $L_{\Lambda(\mathrm{x},\mathrm{y})}(s,\chi)$ coincides with that of $L(s,\chi)$ for all $v\nmid \ord(\mathrm{x})\ord(\mathrm{y})f$. This implies the convergence of the Euler product of $L_{\Lambda(\mathrm{x},\mathrm{y})}(s,\chi)$ for $\Re s>1$ and that it coincides with Definition \ref{defi:L-functions}. The claimed equality follows because the finitely many factors for $v\mid\ord(\mathrm{x})\ord(\mathrm{y})f$ are all holomorphic for $\Re s>0$. The expression for the local factors of $L(s,\chi)$ is standard and follows from \eqref{eq:L-order-local-integral}.
\end{proof}

\subsubsection{The Residue at 1}
\begin{prop}\label{prop:residue}
\begin{equation*}
0<\Res_{s=1} L_{\Lambda(\mathrm{x},\mathrm{y})}(s)
\leq L(1,\chi_E)p_1^{-2\tau} \ord(\mathrm{y})^2  \prod_{p\mid f} (1-\chi_E(p)p^{-1})
\prod_{v\mid\ord(\mathrm{x})}\left[\Lambda_v^\times:\Lambda_v^\times(\mathrm{x})\right] \, .
\end{equation*}
\end{prop}
\begin{proof}
Recall that $\Res_{s=1} \zeta_{E}(s)=L(1,\chi_E)$. The claim then follows
 from Propositions \ref{prop:L-Euler-product}, \ref{prop:L-local-at-1} and Lemma \ref{lem:L-loc-Bowen}.
\end{proof}

\subsubsection{The Line Re(s)=1/2}
The following lemma is useful in order to bound the Hecke $L$-function $L(s,\chi)$ on the line $\Re s=1/2$ as the subconvexity bound involves the norm of the conductor of the character $\chi$.
\begin{lem}\label{lem:conductor}
For a prime $\mathfrak{p}$ of $E$ let $U_{\mathfrak{p}^n}=1+\mathfrak{p}^n\mathcal{O}_{E_\mathfrak{p}}\subset E_\mathfrak{p}$ be the principal unit group of order $n>0$ and set $U_0\coloneqq \mathcal{O}_{E_\mathfrak{p}}^\times$. For any decomposable compact open subgroup $\mathfrak{K}=\prod_{\mathfrak{p}} \mathfrak{K}_\mathfrak{p}<\mathbb{A}_{E,f}$ the conductor of $\mathfrak{K}$ is an ideal $\mathfrak{c}(\mathfrak{K})=\prod \mathfrak{p}^{n_\mathfrak{p}}$ of $E$ such that  $U_{\mathfrak{p}^{n_\mathfrak{p}}}$ is the maximal principal unit subgroup contained in $\mathfrak{K}_\mathfrak{p}$ for each prime $\mathfrak{p}$.

The conductor $\mathfrak{c}(\Lambda_f^\times(\mathrm{x}))$ satisfies
\begin{equation*}
\mathfrak{c}(\Lambda_f^\times(\mathrm{x})) \mid f \left(\mathcal{O}_E:\jmath_{/\Lambda}(\mathrm{x})\right) \, ,
\end{equation*}
where
\begin{equation*}
\left(\mathcal{O}_E:\jmath_{/\Lambda}(\mathrm{x})\right) \coloneqq \left\{z\in\mathcal{O}_E\mid z\cdot \jmath_{/\Lambda}(\mathrm{x}) \in \mathcal{O}_E \right\}=\bigcap_{v<\infty} \prod_{w\mid v}\begin{cases}
\mathcal{O}_{E_w} & \mathrm{x}_w\in\mathcal{O}_{E_w}\\
\mathrm{x}_w^{-1}\mathcal{O}_{E_w} & \mathrm{x}_w\not\in\mathcal{O}_{E_w}
\end{cases} \, .
\end{equation*}
In particular,
\begin{align*}
\Nr \mathfrak{c}(\Lambda_f^\times(\mathrm{x}))&\leq f^2\prod_{v<\infty} \Nr\left(\mathcal{O}_E:\jmath_{/\Lambda}(\mathrm{x})\right)\, ,\\
\Nr\left(\mathcal{O}_E:\jmath_{/\Lambda}(\mathrm{x})\right)&=
\prod_{\substack{w\mid v \\ \mathrm{x}_w\not\in \mathcal{O}_{E_w}}} |\mathrm{x}_w|_w \, .
\end{align*}
\end{lem}
\begin{proof}
For any $v<\infty$ we need to show that the following subgroup is contained in $\Lambda_v^\times$
\begin{equation*}
\prod_{w\mid v}\begin{cases}
\mathcal{O}_{E_w}^\times &  \mathrm{x}_w\in\mathcal{O}_{E_w} \textrm{ and } f_v=1\\
1+f_v\mathcal{O}_{E_w} & \mathrm{x}_w\in\mathcal{O}_{E_w} \textrm{ and } f_v\neq 1\\
1+\frac{f_v}{\mathrm{x}_w}\mathcal{O}_{E_w} & \mathrm{x}_w\not\in\mathcal{O}_{E_w}
\end{cases} \, .
\end{equation*}
Without loss of generality, we can arrange that $\mathrm{x}_v\in E_v^\times$ by adding an element of $\Lambda_v$ to $\mathrm{x}_v$. Then there is a simple expression for $\Lambda_v^\times(\mathrm{x})$
\begin{equation*}
\Lambda_v^\times(\mathrm{x})=1+\frac{\Lambda_v}{\mathrm{x}_v}\cap \Lambda_v^\times \, .
\end{equation*}
If $\Lambda_v=\mathcal{O}_{E_v}$ then
\begin{equation*}
\mathcal{O}_{E_v}^\times(\mathrm{x})=1+\frac{\mathcal{O}_{E_v}}{\mathrm{x}_v}\cap \mathcal{O}_{E_v}^\times=\prod_{w\mid v}\begin{cases}
\mathcal{O}_{E_w}^\times & \mathrm{x}_w\in\mathcal{O}_{E_w}\\
1+\frac{\mathcal{O}_{E_w}}{\mathrm{x}_w} & \mathrm{x}_w\not\in\mathcal{O}_{E_w}
\end{cases}
\end{equation*}
as required.

Assume next $\Lambda_v=\mathbb{Z}_v+f_v\mathcal{O}_{E_v}$ is a non-maximal order, then
\begin{align*}
\Lambda_v^\times(\mathrm{x})=1+\frac{\Lambda_v}{\mathrm{x}_v}\cap \Lambda_v^\times
&\supseteq 1+\frac{f_v\mathcal{O}_{E_v}}{\mathrm{x}_v}\cap 1+f_v\mathcal{O}_{E_v}=1+f_v\left(\frac{\mathcal{O}_{E_v}}{\mathrm{x}_v}\cap\mathcal{O}_{E_v}\right)
\\
&=\prod_{w\mid v} \begin{cases}
1+f_v\mathcal{O}_{E_w} & \mathrm{x}_w\in\mathcal{O}_{E_w}\\
1+\frac{f_v}{\mathrm{x}_w}\mathcal{O}_{E_w} & \mathrm{x}_w\not\in\mathcal{O}_{E_w}
\end{cases} \, .
\end{align*}
\end{proof}

\begin{prop}\label{prop:L1/2-bound}
If $\Re s=1/2$ and $\val_{p_1}(f)=0$ then
\begin{equation*}
|L_{\Lambda(\mathrm{x},\mathrm{y})}(s)|
\ll_\varepsilon (f\ord(\mathrm{y}))^{\varepsilon}
|L(s,\chi)|p_1^{-\tau} \ord(\mathrm{y})^2 12^{\omega(\ord(\mathrm{x}))}
\left(\Nr \left(\mathcal{O}_E:\jmath_{/\Lambda}(\mathrm{x})\right) \right)^{-1/2}
\prod_{v\mid\ord(\mathrm{x})} \left[\Lambda_v^\times:\Lambda_v^\times(\mathrm{x})\right] \, ,
\end{equation*}
where the definition of $\left(\mathcal{O}_E:\jmath_{/\Lambda}(\mathrm{x})\right)$ is as in Lemma \ref{lem:conductor} above and in particular
\begin{equation*}
\Nr \left(\mathcal{O}_E:\jmath_{/\Lambda}(\mathrm{x})\right)=\prod_{v<\infty}\prod_{\substack{w\mid v \\ \mathrm{x}_w\not\in \mathcal{O}_{E_w}}} |\mathrm{x}_w|_w \, .
\end{equation*}
\end{prop}
\begin{proof}
Assume $\Re s=1/2$. The trivial bound $|L_v(s,\chi)^{-1}|\leq 3$ holds for any place $v<\infty$ because $|\chi_E(p)|\leq 1$ and $p\geq 2$ for $p$ -- the residue characteristic for $\mathbb{Q}_v$.
This bound in conjunction with Propositions \ref{prop:L-Euler-product}, \ref{prop:Lx<=Lxy} and Lemma \ref{lem:L-loc-Bowen} then imply
\begin{align}\label{eq:L1/2-combined}
|L_{\Lambda(\mathrm{x},\mathrm{y})}(s)|
&\leq
|L(s,\chi)| p_1^{-\tau} \ord(\mathrm{y}) 3^{\omega(f\ord(\mathrm{x})\ord(\mathrm{y}))}  \\
&\cdot\prod_{v\mid f\ord(\mathrm{x})\ord(\mathrm{y})}
\frac{\meas_v\left(\Lambda_v^\times(\ord(\mathrm{y})\mathrm{x})\right)f_v L_{\Lambda_v(\ord(\mathrm{y})\mathrm{x})}(1/2)}
{\meas_v\left(\Lambda_v^\times(\mathrm{x})\right)f_v} \, . \nonumber
\end{align}
We bound the nominators in the product aboves using Corollaries \ref{cor:L1/2-x=0} and \ref{cor:L1/2-non-max}
\begin{align*}
\prod_{v\mid f\ord(\mathrm{x})\ord(\mathrm{y})}
&\meas_v\left(\Lambda_v^\times(\ord(\mathrm{y})\mathrm{x})\right)f_v
L_{\Lambda_v(\ord(\mathrm{y})\mathrm{x})}(1/2)
\leq
4^{\omega(f)}
\prod_{p^n\mid f} n\\
&\cdot
\prod_{v\mid\ord(\mathrm{x})} \prod_{w\mid v}
\begin{cases}
|\ord(\mathrm{y})\mathrm{x}_w|_w^{-1/2} & |\mathrm{x}_w|_w\geq |\ord(\mathrm{y})|_w^{-1}\\
1 &|\mathrm{x}_w|_w< |\ord(\mathrm{y})|_w^{-1}
\end{cases}\\
&\ll_\varepsilon f^\varepsilon \ord(\mathrm{y}) \prod_{v\mid\ord(\mathrm{x})}
\prod_{\substack{w\mid v\\\mathrm{x}_w\not\in\mathcal{O}_{E_w}}}|\mathrm{x}_w|_w^{-1/2} \, .
\end{align*}
where we have used the inequality $n+3\leq 4n$ which is valid for all $n\geq 1$ and the elementary inequality $\prod_{p^n \mid f} n\leq d(f)\ll f^\varepsilon$, where $d(f)$ is the number of divisors of $f$.

Lemma \ref{lem:local-volumes} implies that the product of the denominators is equal to
\begin{align*}
\prod_{v\mid f\ord(\mathrm{x})\ord(\mathrm{y})} \left(\meas_v\left(\Lambda_v^\times(\mathrm{x})\right)f_v\right)^{-1}
=&\prod_{v\mid\ord(\mathrm{x})} \left[\Lambda_v^\times:\Lambda_v^\times(\mathrm{x})\right]\\
&\prod_{p\mid f\ord(\mathrm{x})\ord(\mathrm{y})}(1-p^{-1})^{-1}(1-\chi_E(p)p^{-1})^{-1}
\prod_{\substack{p\mid f}}(1-\chi_E(p)p^{-1})\\
&\ll 4^{\omega(f\ord(\mathrm{x})\ord(\mathrm{y}))}
\prod_{v\mid\ord(\mathrm{x})} \left[\Lambda_v^\times:\Lambda_v^\times(\mathrm{x})\right] \, .
\end{align*}
Combining these inequalities for the nominator and denominator of the product in \eqref{eq:L1/2-combined} and Lemma \ref{lem:conductor} we arrive at the claimed inequality.
\end{proof}

\bibliographystyle{alpha}
\bibliography{torsion_kuga_sato_bib}
\end{document}